\documentclass[11pt]{report}
\usepackage{amssymb}
\usepackage{amsmath}
\usepackage{amsfonts}
\usepackage{sw20mitt}

\newtheorem{theorem}{Theorem}[chapter]

\newtheorem{corollary}[theorem]{Corollary}

\newtheorem{lemma}[theorem]{Lemma}

\newtheorem{problem}[theorem]{Problem}
\newtheorem{proposition}[theorem]{Proposition}

\newenvironment{proof}[1][Proof]{\textbf{#1.} }{\ \rule{0.5em}{0.5em}}
\input{tcilatex}
\begin{document}

\title{Group-theoretic Methods for the Bounding the Exponent of Matrix
Multiplication}
\author{Sandeep Rajkumar Murthy \and (born April 4, 1977, Hyderabad, India)}
\date{July 16, 2007}
\thisdegree{Master of Science (Logic)}
\university{Universiteit van Amsterdam}
\super{Prof. T. H. Koornwinder (KdV\ Institute for Mathematics, UvA)\\
Prof. P. van Emde Boas (ILLC, UvA)\\
Prof.\ R. van der Waall (KdV\ Institute for Mathematics, UvA)}
\tableofcontents

\textbf{\pagebreak }

\textbf{Acknowledgements}

\bigskip

\bigskip

This Masters thesis was submitted to the Board of Examiners in partial
fulfillment of the requirements for the degree of Master of Science in Logic
at the Universiteit van Amsterdam (UvA), The Netherlands, on August 21,
2007. \ Members of the Board of Examiners were:

\bigskip 

\hspace{0.65in}Prof. Dick de Jongh, Institute for Logic, Language and
Computation (ILLC), UvA

\hspace{0.65in}Prof. Tom Koornwinder, Korteweg de Vries (KdV) Institute for
Mathematics, UvA

\hspace{0.65in}Prof. Peter van Emde Boas, KdV\ Institute for Mathematics, UvA

\hspace{0.65in}Prof. Robert van der Waall, KdV\ Institute for Mathematics,
UvA

\bigskip 

The thesis was completed under the kind supervision of Prof. Tom Koornwinder
of the Korteweg de Vries (KdV) Institute for Mathematics.

\bigskip

The author also acknowledges the help of Prof. Peter van Emde Boas of the
ILLC in relation to Chapter 2 dealing with the algebraic theory of
complexity of matrix multiplication, and of Prof. R. van der Waal of the KdV
Institute in relation to Chapter 3 dealing with the representation theory of
finite groups.

\bigskip

Acknowledgements are also due to the coauthors of the papers on which the
thesis is based, Dr. Chris Umans (Dept. of Computer Science, California
Institute of Technology), and Dr. Henry Cohn (Microsoft Research), for some
clarifying remarks and comments.

\bigskip

\pagebreak

\chapter{\protect\huge Introduction}

\bigskip

\section{The Exponent $\protect\omega $\ of Matrix Multiplication}

\bigskip

Matrix multiplication is a fundamental operation in linear algebra. \ For
any given field $K$, the (asymptotic) complexity of matrix multiplication
over $K$ is measured by a real parameter $\omega (K)>0$, called the \textit{%
exponent of matrix multiplication} over $K$, which is defined to be the
smallest real number $\omega >0$ such that for an arbitrary degree of
precision $\epsilon >0$, two $n\times n$ $K$-matrices can be multiplied
using an algorithm using $O(n^{\omega +\epsilon })$ number of non-division
arithmetical operations, i.e. less than some constant $\geq 1$ multiple of $%
n^{\omega +\epsilon }$ number of multiplications, additions or subtractions.
\ The notation $\omega (K)$ indicates a dependency on the ground field $K$,
but we usually have in mind the complex field $K=%
\mathbb{C}
$, which is general enough for most purposes. \ It is proved that $\omega $
determines the complexities of many other linear operations, e.g. matrix
inversion, determinants, etc., and these are concisely covered by B\"{u}%
rgisser et. al., \textit{Chapter 16 }in [BCS1997].

\bigskip

If we denote by $M_{K}\left( n\right) $ the total number of arithmetical
operations for performing $n\times n$ matrix multiplication over a field $K$%
, then by the standard algorithm, $M_{K}\left( n\right)
=2n^{3}-n^{2}=O(n^{3})$, because one needs to perform $n^{3}$
multiplications, and $n^{3}-n^{2}$ additions of the resulting products. \
This is equivalent to an upper bound of $3$ for $\omega $. \ Since the
product of two $n\times n$ matrices consists of $n^{2}$ entries, one needs
to perform a total number of operations which is \textit{at least} some
constant $\geq 1$ multiple of the $n^{2}$ entries. \ This is written as $%
M_{K}(n)=\Omega (n^{2})$, which is equivalent to a lower bound of $2$ for $%
\omega $. \ Strassen in 1969 obtained the first important result that $%
\omega <2.81$ using his result that $2\times 2$ matrix multiplication could
be performed using $7$ multiplications, not $8$, as in the standard
algorithm [STR1969, p. 355]. \ In 1984, Pan improved this to $2.67$, using a
variant of Strassen's approach [PAN1984, p. 400]. \ It has been conjectured
for twenty years that $\omega =2$, but the best known result is that $\omega
<2.38$, due to Coppersmith and Winograd [CW1990, p. 251]. \ In all these
approaches, estimates for $\omega $ depend on the number of main running
steps in their algorithms.

\bigskip

\section{Groups and Matrix Multiplication}

\bigskip

In a recent series of papers in 2003 and 2005, Cohn and Umans put forward an
entirely different approach using fairly elementary methods from group
theory to describe the complexity of matrix multiplication.

\bigskip

The approach is based on two important facts.

\bigskip

\subsection{\textit{Realizing Matrix Multiplications via Finite Groups}}

\bigskip

(I) A (nontrivial) finite group $G$ which has a triple of index subsets $S$, 
$T$, $U\subseteq G$ of sizes $\left\vert S\right\vert =n$, $\left\vert
T\right\vert =m$, $\left\vert U\right\vert =p$, such that $s^{^{\prime
}}s^{-1}t^{^{\prime }}t^{-1}u^{^{\prime }}u^{-1}=1_{G}\Longleftrightarrow
s^{^{\prime }}s^{-1}=t^{^{\prime }}t^{-1}=u^{^{\prime }}u^{-1}=1_{G}$, for
all elements $s^{^{\prime }},s$ $\epsilon $ $S$, $t^{^{\prime }},t$ $%
\epsilon $ $T$, $u^{^{\prime }},u$ $\epsilon $ $U$, \textit{realizes}
multiplication of $n\times m$ by $m\times p$ matrices over $%
\mathbb{C}
$, in the sense that the entries of a given $n\times m$ complex matrix $%
A=\left( A_{i,j}\right) $ and an $m\times p$ matrix $B=\left( B_{k,l}\right) 
$, can be indexed by the subsets $S,$ $T,$ $U$ as $A=\left( A_{s,t}\right)
_{s\epsilon S,t\epsilon T}$ and $B=\left( B_{t^{\prime },u}\right)
_{t^{\prime }\epsilon T,u\text{ }\epsilon U}$ , then injectively embedded in
the regular group algebra $%
\mathbb{C}
G$ of $G$ as the elements $\overline{A}=\underset{s\text{ }\epsilon \text{ }%
S,\text{ }t\text{ }\epsilon \text{ }T}{\sum }A_{s,t}s^{-1}t$ and $\overline{B%
}=\underset{t^{\prime }\text{ }\epsilon \text{ }T,\text{ }u\text{ }\epsilon 
\text{ }U}{\sum }B_{t^{\prime },u}t^{^{\prime }-1}u$, and the matrix product 
$AB$ can be computed by the rule that the $s^{^{\prime \prime }},u^{^{\prime
\prime }}$-th entry $\left( AB\right) _{s^{^{\prime \prime }},u^{^{\prime
\prime }}}$ is the coefficient of the term $s^{^{\prime \prime
}-1}u^{^{\prime \prime }}$ in the group algebra product $\overline{A}%
\overline{B}=\underset{s\text{ }\epsilon \text{ }S,\text{ }t\text{ }\epsilon 
\text{ }T\text{ }}{\sum }\underset{t^{^{\prime }}\text{ }\epsilon \text{ }T,%
\text{ }u\text{ }\epsilon \text{ }U}{\sum }A_{s,t}B_{t^{^{\prime
}},u}s^{-1}tt^{^{\prime }-1}u$ (see \textit{Theorem 4.13}). \ In this case,
by definition, $G$ is said to support $n\times m$ by $m\times p$ matrix
multiplication, equivalently, to \textit{realize} the matrix tensor $%
\left\langle n,m,p\right\rangle $, whose size we define as $nmp$, and the
subsets $S$, $T$, $U$ are said to have the \textit{triple product property}
(TPP) and be an \textit{index triple} of $G$ corresponding to the tensor $%
\left\langle n,m,p\right\rangle $ (see \textit{Chapter 4}).

\bigskip

\subsection{\textit{Wedderburn's Theorem}}

\bigskip

(II) By Wedderburn's theorem there is an isomorphism $%
\mathbb{C}
G\cong \underset{\varrho \text{ }\epsilon \text{ }Irrep(G)}{\oplus }%
\mathbb{C}
^{d_{\varrho }\times d_{\varrho }}$ of the regular group algebra $%
\mathbb{C}
G$ of $G$, where $\underset{\varrho \text{ }\epsilon \text{ }Irrep(G)}{%
\oplus }$ $%
\mathbb{C}
^{d_{\varrho }\times d_{\varrho }}$ is a block-diagonal matrix algebra of
dimension $\underset{\varrho \text{ }\epsilon \text{ }Irrep(G)}{\sum }%
d_{\varrho }^{2}=\left\vert G\right\vert $, and the $%
\mathbb{C}
^{d_{\varrho }\times d_{\varrho }}$ are its irreducible subalgebras of
dimensions $d_{\varrho }^{2}$, and the $d_{\varrho }$ are the degrees of the
distinct irreducible characters of $G$, i.e. the dimensions $d_{\varrho
}=Dim $ $\varrho $ of the distinct (inequivalent) irreducible
representations $\varrho $ $\epsilon $ $Irrep(G)$ of $G$. \ Any such
isomorphism constitutes a group discrete Fourier transform (DFT) for $G$,
and, further, we can deduce that $\left\vert G\right\vert \leq \mathfrak{R}(%
\mathfrak{m}_{_{%
\mathbb{C}
G}})\leq \underset{\varrho \text{ }\epsilon \text{ }Irrep(G)}{\sum }%
\mathfrak{R}\left( \left\langle d_{\varrho },d_{\varrho },d_{\varrho
}\right\rangle \right) $, where $\mathfrak{R}(\mathfrak{m}_{_{%
\mathbb{C}
G}})$ and $\mathfrak{R}\left( \left\langle d_{\varrho },d_{\varrho
},d_{\varrho }\right\rangle \right) $ are the ranks of the bilinear
multiplication maps $\mathfrak{m}_{_{%
\mathbb{C}
G}}$ and $\left\langle d_{\varrho },d_{\varrho },d_{\varrho }\right\rangle $
in $%
\mathbb{C}
G$ and $%
\mathbb{C}
^{d_{\varrho }\times d_{\varrho }}$, respectively.

\bigskip

\subsection{\textit{A Group-theoretic DFT Algorithm for Matrix Multiplication%
}}

\bigskip

This suggests the following group-theoretic DFT algorithm for $n\times n$
matrix multiplication via a finite group $G$ realizing $n\times n$ matrix
multiplication via subsets $S$, $T$, $U\subseteq G$ having the triple
product property.

\bigskip

\begin{enumerate}
\item Injectively embed $n\times n$ complex matrices $A=\left(
A_{i,j}\right) _{1\leqslant i,j\leqslant n}$ and $B=\left( B_{j^{^{\prime
}},k}\right) _{1\leqslant j^{^{\prime }},k\leqslant n}$ into $%
\mathbb{C}
G$ as elements $\overline{A}=\underset{s\text{ }\epsilon \text{ }S,\text{ }t%
\text{ }\epsilon \text{ }T}{\sum }A_{s,t}s^{-1}t$ and $\overline{B}=\underset%
{t^{\prime }\text{ }\epsilon \text{ }T,\text{ }u\text{ }\epsilon \text{ }U}{%
\sum }B_{t^{\prime },u}t^{^{\prime }-1}u$ via $S$, $T$, $U$ (as described in
(I)).

\item Use a discrete group Fourier transform (DFT) for $G$, $DFT:$ $%
\mathbb{C}
G\cong \underset{\varrho \text{ }\epsilon \text{ }Irrep(G)}{\oplus }%
\mathbb{C}
^{d_{\varrho }\times d_{\varrho }}$ to compute the transforms $\widehat{A}%
=DFT(\overline{A})$ and $\widehat{B}=DFT(\overline{B})$.

\item Compute the block-diagonal matrix product of transforms, $\widehat{C}=%
\widehat{A}\widehat{B}$.

\item Recover the vector $\overline{A}\overline{B}=\overline{C}%
=DFT^{-1}\left( \widehat{C}\right) $ from its transform by the inverse group
DFT.

\item Fill the matrix $AB$ from $\overline{A}\overline{B}=\underset{s\text{ }%
\epsilon \text{ }S,\text{ }t\text{ }\epsilon \text{ }T\text{ }}{\sum }%
\underset{t^{^{\prime }}\text{ }\epsilon \text{ }T,\text{ }u\text{ }\epsilon 
\text{ }U}{\sum }A_{s,t}B_{t^{^{\prime }},u}s^{-1}tt^{^{\prime }-1}u$ by the
rule that for each $s^{^{\prime \prime }}$ $\epsilon $ $S$, $u^{^{\prime
\prime }}$ $\epsilon $ $U$, the $s^{^{\prime \prime }},u^{^{\prime \prime }}$%
-th entry $\left( AB\right) _{s^{^{\prime \prime }},u^{^{\prime \prime }}}=$
coefficient $\underset{t,\text{ }t^{^{\prime }}\text{ }\epsilon \text{ }T}{%
\sum }A_{s,t}B_{t^{^{\prime }},u}$ of the term $s^{-1}tt^{^{\prime }-1}u$ in 
$\overline{A}\overline{B}$ for which $s=s^{^{\prime \prime }},$ $%
t=t^{^{\prime }},$ $u=u^{^{\prime \prime }}$.
\end{enumerate}

\subparagraph{\protect\bigskip}

\subsection{\textit{The Complexity of Matrix Multiplications Realized by
Groups}}

\bigskip

In the approach we describe, estimates for $\omega $ can be derived from
certain numerical parameters relating to the efficiency with which groups
realize matrix multiplications, and also to the degrees of their irreducible
characters.

\bigskip

\begin{description}
\item[\textit{(1)}] \textit{Pseudoexponents} $\alpha (G)$, defined by $%
\alpha (G):=\log _{z^{^{\prime }}(G)^{1/3}}\left\vert G\right\vert $, where $%
\left\langle n^{^{\prime }},m^{^{\prime }},p^{^{\prime }}\right\rangle $ is
a matrix tensor of maximal size $z^{^{\prime }}(G)=n^{^{\prime }}m^{^{\prime
}}p^{^{\prime }}>1$ realized by $G$, and uniquely determining $\alpha (G)$.
\ We prove that $\left\vert G\right\vert \leq z^{^{\prime }}(G)<\left\vert
G\right\vert ^{\frac{3}{2}}$, which is equivalent to $2<\alpha (G)\leq 3$,
and that $\alpha (G)=3$ whenever $G$ is Abelian. \ The $\alpha (G)$ are
measures of the efficiency with which the groups $G$ realize or embed matrix
multiplication, and the closer $\alpha (G)$ is to $2$ the higher the
embedding efficiency (section \textit{4.2.1}).

\item[\textit{(2)}] \textit{Parameters }$\gamma (G)$, defined by $\gamma
(G):=$ $\underset{\gamma \text{ }>\text{ }0}{\inf }\left\vert G\right\vert ^{%
\frac{1}{\gamma }}=d^{^{\prime }}(G)$, where $d^{^{\prime }}(G)$ is any
maximal irreducible character degree of $G$, for which we can easily prove
that $2<2\frac{\log \left\vert G\right\vert }{\log \left( \left\vert
G\right\vert -1\right) }\leq \gamma (G)\leq 2\frac{\log \left\vert
G\right\vert }{\log \left( \left\vert G\right\vert -c(G)\right) }$, where $%
c(G)$ is the class number of $G$, and that $2<2\frac{\log \left\vert
G\right\vert }{\log \left( \left\vert G\right\vert -1\right) }<\gamma (G)<2%
\frac{\log \left\vert G\right\vert }{\log \left( \left\vert G\right\vert
-c(G)\right) }$, iff $G$ is non-Abelian ((section \textit{4.2.2}))

\item[\textit{(3)}] Sums of powers of irreducible character degrees, $%
D_{r}(G)=\underset{\varrho \text{ }\epsilon \text{ }Irrep(G)}{\sum }%
d_{\varrho }^{r}$, $r\geq 0$. \ Facts from representation theory are that $%
D_{0}(G)=c(G)$, and $D_{2}(G)=\left\vert G\right\vert $ (section \textit{3.2.%
}).
\end{description}

\subsubsection{\protect\bigskip}

\subsection{\textit{Relations and Results for the Exponent }$\protect\omega $%
}

\bigskip

The four most important relations which we prove using single groups $G$ are
the following.

\bigskip

\begin{equation*}
\begin{tabular}{|l|l|}
\hline
\textit{(1.1) \ } $\left\vert G\right\vert ^{\frac{\omega }{\alpha (G)}%
}\leqslant $ $D_{\omega }(G)$ & this is equivalent to $2\leqslant \omega
\leqslant \alpha (G)\log _{\left\vert G\right\vert }D_{\omega }(G)$ \\ \hline
\textit{(1.2) \ } $D_{r}(G)\leq \left\vert G\right\vert ^{\frac{\left(
r-2\right) }{\gamma (G)}+1}$ & real $r\geq 2$ \\ \hline
\textit{(1.3) \ \ }$\left( nmp\right) ^{\frac{1}{3}}\leq d^{^{\prime
}}(G)^{1-\frac{2}{\omega }}\left\vert G\right\vert ^{\frac{1}{\omega }}$ & 
if $G$ realizes a tensor $\left\langle n,m,p\right\rangle $ \\ \hline
\textit{(1.4) \ \ }$\omega \leq \alpha (G)\left( \frac{\gamma (G)-2}{\gamma
(G)-\alpha (G)}\right) $ & if $\alpha (G)<\gamma (G)$ \\ \hline
\end{tabular}%
\end{equation*}

\bigskip

By applying these relations, we have also been able to derive a number of
results about how to prove estimates for $\omega $ using non-Abelian groups.
\ These are listed below.

\begin{proposition}
$\omega \leq t<3$ for some $t>2$, if there is a non-Abelian finite group $G$
with pseudoexponent $\alpha (G)$ and parameter $\gamma (G)$ such that $%
\alpha (G)<\gamma (G)$ and $\alpha (G)\left( \frac{\gamma (G)-2}{\gamma
(G)-\alpha (G)}\right) \leq t$. \ An equivalent, but more precise statement
is that $\omega \leq t<3$, for some $t>2$, if there is a non-Abelian finite
group $G$ realizing matrix multiplication of maximal size $z^{^{\prime }}(G)$
and with a maximal irreducible character degree $d^{^{\prime }}(G)$ such
that $z^{^{\prime }}(G)^{\frac{1}{3}}>d^{^{\prime }}(G)$ and $\left\vert
G\right\vert \leq \frac{z^{^{\prime }}(G)^{\frac{t}{3}}}{d^{^{\prime
}}(G)^{(t-2)}}$. (Corollary 4.28)
\end{proposition}

\begin{proposition}
If $\left\{ G_{k}\right\} $ is a family of non-Abelian groups such that $%
\alpha \left( G_{k}\right) \equiv \alpha _{k}=2+o(1)$, and $\gamma \left(
G_{k}\right) \equiv \gamma _{k}=2+o(1)$, and $\alpha _{k}-2=o(\gamma _{k}-2)$%
, as $k\longrightarrow \infty $, then $\omega =2$. \ (Corollary 4.29)
\end{proposition}

\begin{proposition}
Let $\left\{ G_{k}\right\} $ be a family of non-Abelian groups $G_{k}$
realizing matrix multiplications of largest sizes $z_{k}^{^{\prime }}\equiv
z^{^{\prime }}(G_{k})$ and with largest irreducible character degrees $%
d_{k}^{^{\prime }}\equiv d_{k}^{^{\prime }}(G)$. \ Then, $(1)$ $\omega =2$
if $\left\vert G_{k}\right\vert ^{\frac{1}{2}}-z_{k}^{^{\prime }\frac{1}{3}%
}=o(1)$ and $\left( \left\vert G_{k}\right\vert -1\right) ^{\frac{1}{2}%
}-d_{k}^{^{\prime }}=o(1)$ such that $\left\vert G_{k}\right\vert ^{\frac{1}{%
2}}-z_{k}^{^{\prime }\frac{1}{3}}=$ $o\left( \left( \left\vert
G_{k}\right\vert -1\right) ^{\frac{1}{2}}-d_{k}^{^{\prime }}\right) $ as $%
k\longrightarrow \infty $. \ And more generally, $(2)$ $\omega =2$ if $%
\left\vert G_{k}\right\vert \longrightarrow \infty $ as $k\longrightarrow
\infty $ and there exists a sequence $\left\{ C_{k}\right\} $ of constants $%
C_{k}$ for the $G_{k}$ such that $2\leq C_{k}\leq \left\vert
G_{k}\right\vert -1$, $\left\vert G_{k}\right\vert \geq C_{k}\left( 1+\frac{1%
}{C_{k}-1}\right) $, $C_{k}\longrightarrow \infty $, $C_{k}=o(\left\vert
G_{k}\right\vert )$, $\left( \left\vert G_{k}\right\vert -C_{k}\right) ^{%
\frac{1}{2}}-d_{k}^{^{\prime }}=o(1)$, $\left\vert G_{k}\right\vert ^{\frac{1%
}{2}}-z_{k}^{^{\prime }\frac{1}{3}}=o(1)$ and $\left\vert G_{k}\right\vert ^{%
\frac{1}{2}}-z_{k}^{^{\prime }\frac{1}{3}}=$ $o\left( \left( \left\vert
G_{k}\right\vert -C_{k}\right) ^{\frac{1}{2}}-d_{k}^{^{\prime }}\right) $,
as $k\longrightarrow \infty $. (Theorem 4.30)
\end{proposition}

\bigskip

\subsection{\textit{Realizing Simultaneous, Independent Matrix
Multiplications via Groups}}

\bigskip

In \textit{Chapter 5}, we introduce a more general concept of \textit{%
simultaneous triple product property }(STPP) (also due to Cohn and Umans,
[CUKS2005]). \ To be more precise, a collection $\left\{
(S_{i},T_{i},U_{i})\right\} _{i\text{ }\epsilon \text{ }I}$ of triples $%
(S_{i},T_{i},U_{i})$ of subsets $S_{i},T_{i},U_{i}\subseteq G$, of sizes $%
\left\vert S_{i}\right\vert =m_{i}$, $\left\vert T_{i}\right\vert =p_{i}$, $%
\left\vert U_{i}\right\vert =q_{i}$ respectively, is said to satisfy the 
\textit{simultaneous triple product property} (STPP) iff it is the case that
each triple $(S_{i},T_{i},U_{i})$ satisfies the TPP\ and $s_{i}^{^{\prime
}}s_{j}^{^{-1}}t_{j}^{^{\prime }}t_{k}^{^{-1}}u_{k}^{^{\prime
}}u_{i}^{^{-1}}=1_{G}\Longrightarrow i=j=k$, for all $s_{i}^{^{\prime
}}s_{j}^{^{-1}}$ $\epsilon $ $Q(S_{i},S_{j}),$ $t_{j}^{^{\prime
}}t_{k}^{^{-1}}$ $\epsilon $ $Q(T_{j},T_{k}),$ $u_{k}^{^{\prime
}}u_{i}^{^{-1}}$ $\epsilon $ $Q(U_{k},U_{i}),$ $i,j,k$ $\epsilon $ $I$. \ In
this case, $G$ is said to \textit{simultaneously realize }the corresponding
collection $\left\{ \left\langle m_{i},p_{i},q_{i}\right\rangle \right\} _{i%
\text{ }\epsilon \text{ }I}$ of tensors through a corresponding collection $%
\left\{ (S_{i},T_{i},U_{i})\right\} _{i\text{ }\epsilon \text{ }I}$, which
is called a \textit{collection of simultaneous index triples}. \ \ The
importance of STPP is that it describes how a finite group may realize
several independent matrix multiplications simultaneously such that the
total complexity of these several matrix multiplications cannot exceed the
complexity of one multiplication in its regular group algebra. \ The first
set of important results about $\omega $ relating to the STPP are summarised
by the following proposition.

\bigskip

\begin{proposition}
If $\left\{ \left\langle m_{i},p_{i},q_{i}\right\rangle \right\} _{i=1}^{r}$%
is a collection of $r$ tensors simultaneously realized by a group $G$ then 
\begin{equation*}
(1)\underset{i=1}{\overset{r}{\sum }}\left( m_{i}p_{i}q_{i}\right) ^{\frac{%
\omega }{3}}\leq D_{\omega }(G)\hspace{0.65in}\text{( part }(1)\text{ of
Corollary 5.2)}
\end{equation*}%
and if these tensors are all identical, say, $\left\langle
m_{i},p_{i},q_{i}\right\rangle =\left\langle n,n,n\right\rangle $, for all $%
1\leq i\leq r,$ then%
\begin{equation*}
(2)\text{ }\omega \leq \frac{\log \left\vert G\right\vert -\log r}{\log n}.%
\hspace{0.65in}\text{(Corollary 5.3)}
\end{equation*}
\end{proposition}

\bigskip

The most useful types of groups for estimates for $\omega $ using the STPP
seem to be wreath product groups $H\wr Sym_{n}$, for which we prove that $%
D_{\omega }(H\wr Sym_{n})\leq \left( n!\right) ^{\omega -1}\left\vert
H\right\vert ^{n}$ (\textit{Lemma 5.7}). \ Using the latter result, the most
general result which we've obtained (\textit{Proposition 5.10}) in this
regard involves the wreath product groups $H\wr Sym_{n}$ where $H$ is
Abelian.

\bigskip

\begin{proposition}
For any $n$ triples $S_{i},T_{i,}U_{i}\subseteq H$ of sizes $\left\vert
S_{i}\right\vert =m_{i},\left\vert T_{i}\right\vert =p_{i},\left\vert
U_{i}\right\vert =q_{i},$ $1\leq i\leq n$ satisfy the STPP in an Abelian
group $H$, there is a unique number $1\leq k_{n}\leq \left( n!\right) ^{3}$
of triples of permutations, $\sigma _{j},\tau _{j},\upsilon _{j}$ $\epsilon $
$Sym_{n},$ $1\leq j\leq k_{n}$, such that the $k_{n}$ permuted product
triples $\overset{n}{\underset{i=1}{\tprod }}S_{\sigma _{j}\left( i\right)
}\wr Sym_{n},\overset{n}{\underset{i=1}{\tprod }}T_{\tau _{j}\left( i\right)
}\wr Sym_{n},\overset{n}{\underset{i=1}{\tprod }}U_{\upsilon _{j}\left(
i\right) }\wr Sym_{n}$, $1\leq j\leq k_{n}$, satisfy the STPP\ in $H\wr
Sym_{n}$, and such that $H\wr Sym_{n}$ realizes the square product tensor $%
\left\langle n!\overset{n}{\underset{i=1}{\tprod }}m_{i},n!\overset{n}{%
\underset{i=1}{\tprod }}p_{i},n!\overset{n}{\underset{i=1}{\tprod }}%
q_{i}\right\rangle $ $k_{n}$ times simultaneously, such that%
\begin{equation*}
\omega \leq \frac{n\log \left\vert H\right\vert -\log n!-\log k_{n}}{\log 
\sqrt[3]{\overset{n}{\underset{i=1}{\tprod }}m_{i}p_{i}q_{i}}}.
\end{equation*}
\end{proposition}

\subsection{\textit{Estimates for the Exponent }$\protect\omega $}

\bigskip

In \textit{Chapter 6}, we give several estimates for $\omega $ using Abelian
groups $H$ or wreath products involving them, $H\wr Sym_{n}$.

\bigskip

\begin{equation*}
\begin{tabular}{|l|l|l|}
\hline
$\omega <$ & \emph{Group} & \emph{Reference} \\ \hline
$2.93$ & $Cyc_{41}^{\times 3}\wr Sym_{2}$ & {\small section }\textit{6.2.2}
\\ \hline
$2.82$ & $\left( Cyc_{m}^{\times 3}\right) ^{\times m}\wr Sym_{2^{m}}$ & 
{\small section }\textit{6.2.3} \\ \hline
$2.82$ & $Cyc_{16}^{\times 3}$ & {\small section }\textit{6.2.1} \\ \hline
\end{tabular}%
\end{equation*}

\bigskip

We conclude with the observation that $\left( Cyc_{n}^{\times 3}\right)
^{\times n}\wr Sym_{2^{n}}$ (i.e. $\left( \left( Cyc_{n}^{\times 3}\right)
^{\times n}\right) ^{\times 2^{n}}\rtimes Sym_{2^{n}}$) realizes the product
tensor $\left\langle 2^{n}!\left( n-1\right) ^{n2^{n}},2^{n}!\left(
n-1\right) ^{n2^{n}},2^{n}!\left( n-1\right) ^{n2^{n}}\right\rangle $ some $%
1\leq k_{2^{n}}\leq \left( 2^{n}!\right) ^{3}$ times simultaneously such that%
\begin{equation*}
\omega \leq \frac{2^{n}\log n^{3n}-\log 2^{n}!-\log k_{2^{n}}}{2^{n}n\log
\left( n-1\right) }.
\end{equation*}

\bigskip

For example, if $k_{2^{n}}=\left( 2^{n}!\right) ^{3}$ then $\omega \leq 
\frac{2^{n}\log n^{3n}-4\log 2^{n}!}{2^{n}n\log \left( n-1\right) }$, the
latter achieving a minimum of $2.012$ for $n=6$. \ In general, for the
groups $\left( Cyc_{n}^{\times 3}\right) ^{\times n}\wr Sym_{2^{n}}$ the
closer $k_{2^{n}}$ is to $\left( 2^{n}!\right) ^{3}$, the closer $\omega $
is to $2.012$ (from the upper side).

\bigskip

It is one of the original suggestions in the thesis that given an Abelian
group $H$ with a given STPP\ family $\left\{ \left( S_{i},T_{i,}U_{i}\right)
\right\} _{i=1}^{n}$ it is therefore useful to know the how to choose
triples of permutations $\sigma _{j},\tau _{j},\upsilon _{j}$ $\epsilon $ $%
Sym_{n}$ in order that a maximum number $1\leq k_{n}\leq \left( n!\right)
^{3}$ of triples $\overset{n}{\underset{i=1}{\tprod }}S_{\sigma _{j}\left(
i\right) }\wr Sym_{n},\overset{n}{\underset{i=1}{\tprod }}T_{\tau _{j}\left(
i\right) }\wr Sym_{n},\overset{n}{\underset{i=1}{\tprod }}U_{\upsilon
_{j}\left( i\right) }\wr Sym_{n}$, $1\leq j\leq k_{n}$, satisfy the STPP\ in 
$H\wr Sym_{n}$. \ Using such groups and their triples in this way, the
sharpest upper bounds for $\omega $ will occur where the ratio $k_{n}/\left(
n!\right) ^{3}$ is highest.

\pagebreak \bigskip

\chapter{\protect\huge Algebraic Complexity of Matrix Multiplication}

\bigskip

In section \textit{2.1 }we describe matrix multiplication as a bilinear map
describing multiplication in matrix algebras, and introduce a certain
measure of the complexity of matrix multiplication defined in terms of the
concept of rank of bilinear map. \ Then, in section \textit{2.2} we
introduce the exponent $\omega $ as an asymptotic, real-valued measure of
complexity, and conclude by describing the fundamental relations between the
bilinear and the asymptotic measures.

\bigskip

\section{Bilinear Complexity of Matrix Multiplication}

\bigskip

\subsection{\textit{Matrix Multiplication as a Bilinear Map}}

\bigskip

If $U$ and $V$ are two $K$-spaces of dimensions $n$ and $m$, respectively,
with bases $\left\{ u_{i}\right\} _{1\leq i\leq n}$ and $\left\{
v_{j}\right\} _{1\leq j\leq m}$ respectively, then for any third space $W$,
a map $\phi :U\times V\longrightarrow W$ which satisfies the condition%
\begin{eqnarray*}
&&\phi \left( \kappa _{1}u_{1}+\kappa _{2}v_{1},\kappa _{3}u_{2}+\kappa
_{4}v_{2}\right) \\
&=&\kappa _{1}\kappa _{3}\phi \left( u_{1},u_{2}\right) +\kappa _{1}\kappa
_{4}\phi \left( u_{1},v_{2}\right) +\kappa _{2}\kappa _{3}\phi \left(
v_{1},u_{2}\right) +\kappa _{2}\kappa _{4}\phi \left( v_{1},v_{2}\right)
\end{eqnarray*}

for all scalars $\kappa _{1},\kappa _{2},\kappa _{3},\kappa _{4}$ $\epsilon $
$K$ and vectors $u_{1},u_{2}$ $\epsilon $ $U$, $v_{1},v_{2}$ $\epsilon $ $V$%
, is called a $K$-\textit{bilinear map}, or simply, a \textit{bilinear map},
on $U$ and $V$. \ The map $K^{n\times m}\times K^{m\times p}\longrightarrow
K^{n\times p}$ describing multiplication of $n\times m$ by $m\times p$
matrices over $K$ is such a bilinear map, which we denote by $\left\langle
n,m,p\right\rangle _{K}$.\ We call the integers $n,m,p$ the \textit{%
components} of the map $\left\langle n,m,p\right\rangle $, which is also
called a tensor, ([BCS1997, p. 361].

\bigskip

For a $K$-space $U$, the set of all linear forms (functionals) $%
f:U\longrightarrow K$ on $U$, i.e linear maps of $U$ into its ground field,
forms a $K$-space $U^{\ast }$, equidimensional with $U$, called its \textit{%
dual space}. \ The matrix vector space $K^{n\times m}$ has the basis $%
\left\{ E_{ij}\right\} _{\substack{ 1\leq i\leq n  \\ 1\leq j\leq m}}$,
where $E_{ij}$ is the $n\times m$ matrix with a $1$ in its $\left(
i,j\right) ^{th} $ entry and a $0$ everywhere else, and the dual space $%
K^{n\times m^{\ast }}$ has the dual basis $\left\{ e_{ij}^{\ast }\right\} 
_{\substack{ 1\leq i\leq n  \\ 1\leq j\leq m}}$ where $e_{ij}^{\ast }$ is a
map $K^{n\times m}\longrightarrow K$ which sends any $n\times m$ matrix $A$
over $K$ to its $\left( i,j\right) ^{th}$ entry $\left( A\right) _{ij}$. \
We denote the zero $n\times m$ matrix of $K^{n\times m}$ by \textbf{O}$%
_{n\times m}$.

\bigskip

\subsection{\textit{Rank of Matrix Multiplication}}

\bigskip

The set $Bil_{K}\left( U,V;W\right) $ of all bilinear maps on two $K$-spaces 
$U$ and $V$ into a third space $W$ also forms a $K$-space, e.g. $%
\left\langle n,m,p\right\rangle _{K}$ $\epsilon $ $Bil_{K}\left( K^{n\times
m},K^{m\times p};K^{n\times p}\right) $. \ If $U=V=W$, we write $%
Bil_{K}\left( U\right) $ for $Bil_{K}\left( U,V;W\right) $. \ For any $\phi $
$\epsilon $ $Bil_{K}\left( U,V;W\right) $, there is a smallest positive
integer $r$ such that for every pair $\left( u,v\right) $ $\epsilon $ $%
U\times V$, $\phi \left( u,v\right) $ has the bilinear representation%
\begin{equation*}
\phi \left( u,v\right) =\underset{i=1}{\overset{r}{\sum }}f_{i}^{\ast
}\left( u\right) g_{i}^{\ast }\left( v\right) w_{i}
\end{equation*}%
where $f_{i}^{\ast }$ $\epsilon $ $U^{\ast },g_{i}^{\ast }$ $\epsilon $ $%
V^{\ast },w_{i}$ $\epsilon $ $W$ correspond to $\phi $, and the sequence of $%
r$ triples, $f_{1}^{\ast },g_{1}^{\ast },w_{1};$ $f_{2}^{\ast },g_{2}^{\ast
},w_{2};$ $....$ $;$ $f_{r}^{\ast },g_{r}^{\ast },w_{r}$ is called a \textit{%
bilinear computation} for $\phi $ of length $r$ [BCS1997, p. 354]. \ The 
\textit{bilinear complexity} or \textit{rank} $\mathfrak{R}\left( \phi
\right) $ of $\phi $ is defined by%
\begin{equation*}
\mathfrak{R}\left( \phi \right) :=\text{ }\min \left\{ r\text{ }\epsilon 
\text{ }%
\mathbb{Z}
^{+}\text{ }|\text{ }\phi \left( u,v\right) =\underset{i=1}{\overset{r}{\sum 
}}f_{i}^{\ast }\left( u\right) g_{i}^{\ast }\left( v\right) w_{i},\text{ }%
\left( u,v\right) \text{ }\epsilon \text{ }U\times V\right\} .
\end{equation*}%
where $f_{i}^{\ast }$ $\epsilon $ $U^{\ast },g_{i}^{\ast }$ $\epsilon $ $%
V^{\ast },w_{i}$ $\epsilon $ $W$ uniquely correspond to $\phi $, i.e. $%
\mathfrak{R}\left( \phi \right) $ is length $r$ of the shortest bilinear
computation for $\phi $ [BCS1997, p. 354]. \ For example, if $%
U=V=W=K_{Diag}^{n\times n}$, where $K_{Diag}^{n\times n}$ is the space of
all $n\times n$ diagonal matrices over $K$ with pointwise multiplication,
then $f_{i}^{\ast }=g_{i}^{\ast }=e_{ii}^{\ast }$, $w_{i}=E_{ii}$, $1\leq
i\leq r$, and $r=n$. \ In the same way, the rank $\mathfrak{R}\left(
\left\langle n,m,p\right\rangle \right) $ of the tensor $\left\langle
n,m,p\right\rangle $, i.e. the rank of $n\times m$ by $m\times p$ matrix
multiplication, is the smallest positive integer $r$ such that every product 
$AB$ $\epsilon $ $K^{n\times p}$ of an $n\times m$ matrix $A$ $\epsilon $ $%
K^{n\times m}$ and an $m\times p$ matrix $B$ $\epsilon $ $K^{m\times p}$ has
the bilinear representation%
\begin{equation*}
\left\langle n,m,p\right\rangle \left( A,B\right) =AB=\underset{i=1}{\overset%
{r}{\sum }}f_{i}^{\ast }\left( A\right) g_{i}^{\ast }\left( B\right) C_{i}
\end{equation*}%
where $f_{i}^{\ast }$ $\epsilon $ $K^{n\times m^{\ast }}$, $g_{i}^{\ast }$ $%
\epsilon $ $K^{m\times p^{\ast }}$, $C_{i}$ $\epsilon $ $K^{n\times p}$, $%
1\leq i\leq r$. \ For example, $\mathfrak{R}\left( \phi \right) =n$ for any $%
\phi $ $\epsilon $ $Bil_{K}\left( K_{Diag}^{n\times n}\right) $, e.g. $%
\mathfrak{R}\left( \left\langle n,n,n\right\rangle _{Diag}\right) =n$, where 
$\left\langle n,n,n\right\rangle _{Diag}$ is the multiplication map of $%
n\times n$ diagonal matrices.. \ For example, if $K^{n}$ is the $n$
dimensional space of all $n$-tuples over $K$, with a pointwise
multiplication map $\left\langle n\right\rangle :K^{n}\times
K^{n}\longrightarrow K^{n}$ of rank $\mathfrak{R}\left( \left\langle
n\right\rangle \right) =\left\langle n\right\rangle $ then $\mathfrak{R}%
\left( \phi \right) =n$ for any $\phi $ $\epsilon $ $Bil_{K}\left(
K^{n}\right) $ if $\phi \cong _{K}\left\langle n\right\rangle $. \ This
shows that $\mathfrak{R}\left( \left\langle n,n,n\right\rangle \right) \geq
n $, because $K^{n}\cong _{K}K_{Diag}^{n\times n}\leq _{K}K^{n\times n}$. \
Another property of tensors $\left\langle n,m,p\right\rangle $ is invariance
under permutations of their components, [BCS1997, pp. 358-359].

\bigskip

\begin{proposition}
$\mathfrak{R}\left( \left\langle n,m,p\right\rangle \right) =\mathfrak{R}%
\left( \left\langle \mu \left( n\right) ,\mu \left( m\right) ,\mu \left(
p\right) \right\rangle \right) $, for any permutation $\mu $ $\epsilon $ $%
Sym_{3}$.
\end{proposition}

For bilinear maps $\phi $ $\epsilon $ $Bil_{K}\left( U,V;W\right) $ and $%
\phi ^{^{\prime }}$ $\epsilon $ $Bil_{K}\left( U^{^{\prime }},V^{^{\prime
}};W^{^{\prime }}\right) $, $\phi $ is said to be a \textit{restriction} of $%
\phi ^{^{\prime }}$, and we write $\phi \leq _{K}\phi ^{^{\prime }}$, if
there exist linear maps ($K$-space homomorphisms) $\alpha :U\longrightarrow
U^{^{\prime }},\beta :V\longrightarrow V^{^{\prime }}$, $\gamma ^{^{\prime
}}:W^{^{\prime }}\longrightarrow W$ such that $\phi \left( u,v\right)
=\gamma ^{^{\prime }}\circ \phi ^{^{\prime }}\circ \left( \alpha \times
\beta \right) \left( u,v\right) $, for all $\left( u,v\right) $ $\epsilon $ $%
U\times V$. \ In this regard, a basic result is the following.

\begin{proposition}
For any bilinear maps $\phi $ $\epsilon $ $Bil_{K}\left( U,V;W\right) $ and $%
\phi ^{^{\prime }}$ $\epsilon $ $Bil_{K}\left( U^{^{\prime }},V^{^{\prime
}};W^{^{\prime }}\right) $, $\phi \leq _{K}\phi ^{^{\prime }}$ implies $%
\mathfrak{R}\left( \phi \right) \leq \mathfrak{R}\left( \phi ^{^{\prime
}}\right) $.
\end{proposition}

\begin{proof}
For bilinear maps $\phi $ $\epsilon $ $Bil_{K}\left( U,V;W\right) $ and $%
\phi ^{^{\prime }}$ $\epsilon $ $Bil_{K}\left( U^{^{\prime }},V^{^{\prime
}};W^{^{\prime }}\right) $ with ranks $\mathfrak{R}\left( \phi \right) =r$
and $\mathfrak{R}\left( \phi ^{^{\prime }}\right) =r^{^{\prime }}$,
respectively, the assumption that $\phi \leq _{K}\phi ^{^{\prime }}$, by
definition, implies there are linear maps $\alpha :U\longrightarrow
U^{^{\prime }},\beta :V\longrightarrow V^{^{\prime }}$, $\gamma ^{^{\prime
}}:W^{^{\prime }}\longrightarrow W$ such that $\phi \left( u,v\right)
=\gamma ^{^{\prime }}\circ \phi ^{^{\prime }}\circ \left( \alpha \times
\beta \right) \left( u,v\right) $, for all $\left( u,v\right) $ $\epsilon $ $%
U\times V$. \ By these maps, for an arbitrary $\left( u,v\right) $ $\epsilon 
$ $U\times V$%
\begin{eqnarray*}
&&\gamma ^{^{\prime }}\circ \phi ^{^{\prime }}\circ \left( \alpha \times
\beta \right) \left( u,v\right) \\
&=&\gamma ^{^{\prime }}\left( \phi ^{^{\prime }}\left( \left( \alpha \times
\beta \right) \left( u,v\right) \right) \right) \\
&=&\gamma ^{^{\prime }}\left( \phi ^{^{\prime }}\left( \alpha \left(
u\right) ,\beta \left( v\right) \right) \right) \\
&=&\gamma \left( \underset{j=1}{\overset{r^{^{\prime }}}{\sum }}%
f_{j}^{^{\prime }\ast }\left( \alpha \left( u\right) \right) g_{j}^{^{\prime
}\ast }\left( \beta \left( v\right) \right) w_{j}^{^{\prime }}\right) \\
&=&\underset{j=1}{\overset{r^{^{\prime }}}{\sum }}\gamma \left(
f_{j}^{^{\prime }\ast }\left( \alpha \left( u\right) \right) g_{j}^{^{\prime
}\ast }\left( \beta \left( v\right) \right) w_{j}^{^{\prime }}\right) \\
&=&\underset{j=1}{\overset{r^{^{\prime }}}{\sum }}f_{j^{^{\prime }\ast
}}\left( u\right) g_{j^{^{\prime }\ast }}\left( v\right) w_{j^{^{\prime }}}%
\text{ \ \ \ (}+\text{)} \\
&=&\phi \left( u,v\right) ,
\end{eqnarray*}%
where $f_{j^{^{\prime }\ast }}$ $\epsilon $ $U^{\ast },$ $g_{j^{^{\prime
}\ast }}$ $\epsilon $ $V^{\ast },$ $w_{j^{^{\prime }}}$ $\epsilon $ $W,$ $%
1\leq j\leq r^{^{\prime }}$, and $f_{j}^{^{\prime }\ast },g_{j}^{^{\prime
}\ast },w_{j};$ $f_{j}^{^{\prime }\ast }$ $\epsilon $ $U^{^{\prime }\ast
},g_{j}^{^{\prime }\ast }$ $\epsilon $ $V^{^{\prime }\ast },w_{j}^{^{\prime
}}$ $\epsilon $ $W^{^{\prime }},$ $1\leq j\leq r^{^{\prime }}$ is the
minimal bilinear computation for $\phi ^{^{\prime }}$. \ Since we must have $%
\phi \left( u,v\right) =\underset{i=1}{\overset{r}{\sum }}f_{i}^{\ast
}\left( u\right) g_{i}^{\ast }\left( v\right) w_{i}$, where $f_{i}^{\ast
},g_{i}^{\ast },w_{i};$ $f_{i}^{\ast }$ $\epsilon $ $U^{\ast },g_{i}^{\ast }$
$\epsilon $ $V^{\ast },w_{i}$ $\epsilon $ $W,$ $1\leq i\leq r$ is the
minimal bilinear computation for $\phi $, the minimum number $r$ of terms
which can occur in sums of the type ($+$) must be $\leq r^{^{\prime }}$,
i.e. $r\leq r^{^{\prime }}$.
\end{proof}

\bigskip

Two bilinear maps $\phi $ $\epsilon $ $Bil_{K}\left( U,V;W\right) $ and $%
\phi ^{^{\prime }}$ $\epsilon $ $Bil_{K}\left( U^{^{\prime }},V^{^{\prime
}};W^{^{\prime }}\right) $ are said to be isomorphic if there exist
isomorphisms $\alpha :U\longrightarrow U^{^{\prime }}$, $\beta
:V\longrightarrow V^{^{\prime }}$, and $\gamma :W\longrightarrow W^{^{\prime
}}$ such that $\gamma \circ \phi =\phi ^{^{\prime }}\circ \left( \alpha
\times \beta \right) $, [BCS1997, p. 355]. \ The following proposition is a
basic result for isomorphism of bilinear maps.

\bigskip

\begin{corollary}
For any bilinear maps $\phi $ $\epsilon $ $Bil_{K}\left( U,V;W\right) $ and $%
\phi ^{^{\prime }}$ $\epsilon $ $Bil_{K}\left( U^{^{\prime }},V^{^{\prime
}};W^{^{\prime }}\right) $, $\phi \cong _{K}\phi ^{^{\prime }}$ implies $%
\phi \leq _{K}\phi ^{^{\prime }}$ and $\phi ^{^{\prime }}\leq _{K}\phi $, in
which case also $\mathfrak{R}\left( \phi \right) =\mathfrak{R}\left( \phi
^{^{\prime }}\right) $.
\end{corollary}

\begin{proof}
Consequence of \textit{Proposition 2.2.}
\end{proof}

\bigskip

As an example of restrictions, consider the matrix spaces $U=K^{n\times m},$ 
$V=K^{m\times p},$ $W=K^{n\times p}$ and $U^{^{\prime }}=K^{n^{^{\prime
}}\times m^{^{\prime }}},$ $V^{^{\prime }}=K^{m^{^{\prime }}\times
p^{^{\prime }}},$ $W^{^{\prime }}=K^{n^{^{\prime }}\times p^{^{\prime }}}$,
where $n\leq n^{^{\prime }},$ $m\leq m^{^{\prime }},$ $p\leq p^{^{\prime }}$%
, and $\phi =\left\langle n,m,p\right\rangle $ and $\phi ^{^{\prime
}}=\left\langle n^{^{\prime }},m^{^{\prime }},p^{^{\prime }}\right\rangle $.
\ Then, there is a natural, injective linear map $\alpha :K^{n\times
m}\longrightarrow K^{n^{^{\prime }}\times m^{^{\prime }}}$ which embeds an $%
n\times m$ matrix $A$ into $K^{n^{^{\prime }}\times m^{^{\prime }}}$ as an $%
n^{^{\prime }}\times m^{^{\prime }}$ matrix $A^{^{\prime }}$ having $A$ as
an $n\times m$ block in its top-left corner and $0$'s everywhere else. \
There are analogous injective linear embedding maps $\beta :K^{m\times
p}\longrightarrow K^{m^{^{\prime }}\times p^{^{\prime }}}$ and $\gamma
:K^{n\times p}\longrightarrow K^{n^{^{\prime }}\times p^{^{\prime }}}$ for $%
K^{m\times p}$ and $K^{n\times p}$, respectively. \ The product map $\alpha
\times \beta :K^{n\times m}\times K^{m\times p}\longrightarrow
K^{n^{^{\prime }}\times m^{^{\prime }}}\times K^{m^{^{\prime }}\times
p^{^{\prime }}}$ will be injective and there will be a natural surjective
linear map $\gamma ^{^{\prime }}:K^{n^{^{\prime }}\times p^{^{\prime
}}}\longrightarrow K^{n\times p}$ which is the identity map on left upper $%
n\times p$ blocks of $n^{^{\prime }}\times p^{^{\prime }}$ matrices. \ Hence
we have a restriction $\left\langle n,m,p\right\rangle =\gamma ^{^{\prime
}}\circ \left\langle n^{^{\prime }},m^{^{\prime }},p^{^{\prime
}}\right\rangle \circ \left( \alpha \times \beta \right) $ of $\left\langle
n^{^{\prime }},m^{^{\prime }},p^{^{\prime }}\right\rangle $ to $\left\langle
n,m,p\right\rangle $. \ Informally, we have proved the following [BCS1997,
p. 357 \& 362].

\bigskip

\begin{proposition}
If $n\leq n^{^{\prime }},$ $m\leq m^{^{\prime }},$ $p\leq p^{^{\prime }}$
then $\left\langle n,m,p\right\rangle \leq _{K}\left\langle n^{^{\prime
}},m^{^{\prime }},p^{^{\prime }}\right\rangle $, and $\mathfrak{R}\left(
\left\langle n,m,p\right\rangle \right) \leq \mathfrak{R}\left( \left\langle
n^{^{\prime }},m^{^{\prime }},p^{^{\prime }}\right\rangle \right) $.
\end{proposition}

\bigskip

If $U$ and $V$ are two algebras of dimensions $n$ and $m$, respectively,
then their direct sum $U\oplus V$ is an $n+m$ dimensional $K$-space which
has as a basis the union of the bases $\left\{ \left( u_{i},0\right)
\right\} _{1\leq i\leq n}$ and $\left\{ \left( 0,v_{j}\right) \right\}
_{1\leq j\leq m}$ where $\left\{ u_{i}\right\} _{1\leq i\leq n}$ and $%
\left\{ v_{j}\right\} _{1\leq j\leq m}$ are the bases of $U$ and $V$
respectively, and their Kronecker product $U\otimes V$ is an $nm$
dimensional $K$-space of sums of \textit{dyads }$u\otimes v$, $u$ $\epsilon $
$U$, $v$ $\epsilon $ $V$, which has as a basis $\left\{ u_{i}\otimes
v_{j}\right\} _{\substack{ 1\leq i\leq n  \\ 1\leq j\leq m}}$. \ If $U$ and $%
V $ are two algebras of dimensions $n$ and $m$, respectively, then $U\otimes
V$ becomes an $nm$-dimensional algebra with multiplication with the property
that $\left( u\otimes v\right) \left( u^{^{\prime }}\otimes v^{^{\prime
}}\right) =\left( uu^{^{\prime }}\otimes vv^{^{\prime }}\right) $ for any
pair of elements $u\otimes v,$ $u^{^{\prime }}\otimes v^{^{\prime }}$ $%
\epsilon $ $U\otimes V$. \ For bilinear maps $\phi $ $\epsilon $ $%
Bil_{K}\left( U,V;W\right) $ and $\phi ^{^{\prime }}$ $\epsilon $ $%
Bil_{K}\left( U^{^{\prime }},V^{^{\prime }};W^{^{\prime }}\right) $, their
direct sum $\phi \oplus \phi ^{^{\prime }}$ $\epsilon $ $Bil_{K}\left(
U\oplus U^{^{\prime }},V\oplus V^{^{\prime }};W\oplus W^{^{\prime }}\right) $
and Kronecker product $\phi \otimes \phi ^{^{\prime }}$ $\epsilon $ $%
Bil_{K}\left( U\otimes U^{^{\prime }},V\otimes V^{^{\prime }};W\otimes
W^{^{\prime }}\right) $ can be defined and satisfy $\mathfrak{R}\left( \phi
\oplus \phi ^{^{\prime }}\right) \leq \mathfrak{R}\left( \phi \right) +%
\mathfrak{R}\left( \phi ^{^{\prime }}\right) $ and $\mathfrak{R}\left( \phi
\otimes \phi ^{^{\prime }}\right) \leq \mathfrak{R}\left( \phi \right) 
\mathfrak{R}\left( \phi ^{^{\prime }}\right) $ [BCS1997, p. 360]. \ An
important fact here is that to each bilinear map $\phi $ $\epsilon $ $%
Bil\left( U,V;W\right) $ there exists one and only one unique tensor $%
t_{\phi }$ $\epsilon $ $U^{\ast }\otimes V^{\ast }\otimes W$, called the 
\textit{structural tensor }of $\phi $, [BCS1997, p. 358], i.e. $Bil\left(
U,V;W\right) \cong _{K}U^{\ast }\otimes V^{\ast }\otimes W$. \ Therefore,
the isomorphism of two bilinear maps, as described before \textit{Corollary
2.3}, is equivalent to the isomorphism of their corresponding structural
tensors, and therefore, the rank of a bilinear map is equal to the rank of
its structural tensor.

\bigskip

For tensors we have an important but easily provable result [BCS1997, pp.
360-361].

\bigskip

\begin{proposition}
For tensors $\left\langle n,m,p\right\rangle $ and $\left\langle n^{^{\prime
}},m^{^{\prime }},p^{^{\prime }}\right\rangle $, 
\begin{equation*}
(1)\text{ }\left\langle n,m,p\right\rangle \oplus \left\langle n^{^{\prime
}},m^{^{\prime }},p^{^{\prime }}\right\rangle \leq _{K}\left\langle
n+n^{^{\prime }},m+m^{^{\prime }},p+p^{^{\prime }}\right\rangle
\end{equation*}%
and 
\begin{equation*}
(2)\text{ }\left\langle n,m,p\right\rangle \otimes \left\langle n^{^{\prime
}},m^{^{\prime }},p^{^{\prime }}\right\rangle \cong _{K}\left\langle
nn^{^{\prime }},mm^{^{\prime }},pp^{^{\prime }}\right\rangle .
\end{equation*}
\end{proposition}

\bigskip

Direct sums of matrix tensors describe block diagonal matrix multiplication,
and Kronecker products of tensors describe block matrix multiplication. \ A
basic result which we will use is the following, as defined in [BCS1997]
just before Def. (14.18).

\bigskip

\begin{proposition}
For tensors $\left\langle n,m,p\right\rangle $ and $\left\langle n^{^{\prime
}},m^{^{\prime }},p^{^{\prime }}\right\rangle $,%
\begin{equation*}
(1)\text{ }\mathfrak{R}\left( \left\langle n,m,p\right\rangle \oplus
\left\langle n^{^{\prime }},m^{^{\prime }},p^{^{\prime }}\right\rangle
\right) \leq \mathfrak{R}\left( \left\langle n,m,p\right\rangle \right) +%
\mathfrak{R}\left( \left\langle n^{^{\prime }},m^{^{\prime }},p^{^{\prime
}}\right\rangle \right)
\end{equation*}%
and 
\begin{equation*}
(2)\text{ }\mathfrak{R}\left( \left\langle n,m,p\right\rangle \otimes
\left\langle n^{^{\prime }},m^{^{\prime }},p^{^{\prime }}\right\rangle
\right) \leq \mathfrak{R}\left( \left\langle n,m,p\right\rangle \right)
\cdot \mathfrak{R}\left( \left\langle n^{^{\prime }},m^{^{\prime
}},p^{^{\prime }}\right\rangle \right) .
\end{equation*}
\end{proposition}

\subsection{\textit{Matrix Algebras}}

\bigskip

A $K$-algebra $A$ is a vector space $A$ defined over a field $K$, together
with a vector multiplication map $\phi _{A}:A\times A\longrightarrow A$
which is bilinear on $A$, in the sense described above, with a unique unit $%
1_{A}$ which coincides with the unit of $A$ as a multiplicative monoid. \
(Here, by definition an\ algebra $A$ has a unit.) \ The dimension of the
algebra $A$ is defined to be its dimension as a vector space. \ We denote
the unit of $A$ by $1_{A}$. $\ A$ is called \textit{associative} iff $\phi
_{A}$ is associative, and \textit{commutative} iff \ $\phi _{A}$ is
commutative. \ The rank $\mathfrak{R}\left( A\right) $ of $A$ is defined to
be the rank $\mathfrak{R}\left( \phi _{A}\right) $ of $\phi _{A}$, and is a
bilinear measure of the multiplicative complexity in $A$. \ For example, the
matrix space $K^{n\times m}$ is a matrix $K$-algebra iff $n=m$. \ $%
K^{n\times n}$ is an $n^{2}$ dimensional matrix $K$-algebra with a bilinear
map $\left\langle n,n,n\right\rangle $ describing multiplication of $n\times
n$ by $n\times n$ matrices. \ We say that $n$ is the \textit{order} of the
algebra $K^{n\times n}$.

\bigskip

If $A$ and $B$ are two $K$-algebras then a linear map $\varphi
:A\longrightarrow B$ which carries vector multiplication in $A$ onto vector
multiplication in $B$ is called an \textit{algebra homomorphism}, or simply,
an \textit{algebra morphism}, between $A$ and $B$, i.e. if for any $%
a,a^{^{\prime }}$ $\epsilon $ $A$, $\varphi (a,a^{^{\prime }})=\varphi
\left( a\right) \varphi \left( a^{^{\prime }}\right) $, and $\varphi \left(
1_{A}\right) =1_{B}$. \ A simple example is the inclusion homomorphism $%
\varkappa _{2}$ of the algebra $K_{Diag}^{2\times 2}$ of all diagonal $%
2\times 2$ matrices over $K$ into the algebra $K^{2\times 2}$ of all $%
2\times 2$ matrices. \ $\varphi $ is an algebra isomorphism $A\cong _{K}B$
iff $\phi _{A}\cong _{K}\phi _{B}$ ($\Longrightarrow \mathfrak{R}\left( \phi
_{A}\right) =\mathfrak{R}\left( \phi _{B}\right) $), [BCS1997, p. 356]. \
For example, $\left\langle 2,2,2\right\rangle _{Diag}\leq _{K}\left\langle
2,2,2\right\rangle $ and $\mathfrak{R}\left( \left\langle 2,2,2\right\rangle
_{Diag}\right) =2\leq \mathfrak{R}\left( \left\langle 2,2,2\right\rangle
\right) $, and $\mathfrak{R}\left( \left\langle 2,2,2\right\rangle
_{Diag}\right) =2=\mathfrak{R}\left( \left\langle 2\right\rangle \right) $
because $K_{Diag}^{2\times 2}\cong _{K}K^{2}$, where $K^{2}\leq
_{K}K^{2\times 2}$. \ In general, for any $n$-dimensional algebra $A$ it is
the case that $\mathfrak{R}\left( \phi _{A}\right) =n$ iff $A\cong _{K}K^{n}$%
, or equivalently iff $\phi _{A}\cong _{K}\left\langle n\right\rangle $,
[BCS1997, p. 364].

\bigskip

The following is a general result for matrix algebras.

\bigskip

\begin{proposition}
For positive integers $n,$ $n^{^{\prime }}$,%
\begin{equation*}
(1)\text{ if }n\leq n^{^{\prime }}\text{ then }\left\langle
n,n,n\right\rangle \leq _{K}\left\langle n^{^{\prime }},n^{^{\prime
}},n^{^{\prime }}\right\rangle \text{ and }\mathfrak{R}\left( \left\langle
n,n,n\right\rangle \right) \leq \mathfrak{R}\left( \left\langle n^{^{\prime
}},n^{^{\prime }},n^{^{\prime }}\right\rangle \right)
\end{equation*}%
and%
\begin{equation*}
(2)\text{ }\left\langle n,n,n\right\rangle \cong _{K}\left\langle
n^{^{\prime }},n^{^{\prime }},n^{^{\prime }}\right\rangle \text{ and }%
\mathfrak{R}\left( \left\langle n,n,n\right\rangle \right) =\mathfrak{R}%
\left( \left\langle n^{^{\prime }},n^{^{\prime }},n^{^{\prime
}}\right\rangle \right) \text{ iff }n=n^{^{\prime }}.
\end{equation*}
\end{proposition}

\bigskip

If $\left\{ K^{n_{i}\times n_{i}}\right\} $ is a finite collection of matrix
algebras $K^{n_{i}\times n_{i}}$ of orders $n_{i}$, then $\underset{i}{%
\oplus }K^{n_{i}\times n_{i}}$ is a direct sum matrix algebra of order $%
\underset{i}{\sum }n_{i}$, in which multiplication is block diagonal and is
described by the direct sum tensor $\underset{i}{\oplus }\left\langle
n_{i},n_{i},n_{i}\right\rangle \cong \left\langle \underset{i}{\sum }n_{i},%
\underset{i}{\sum }n_{i},\underset{i}{\sum }n_{i}\right\rangle $; and $%
\underset{i}{\otimes }K^{n_{i}\times n_{i}}$ is a Kronecker product matrix
algebra of order $\underset{i}{\tprod }n_{i}$, in which multiplication is
described by the Kronecker product tensor $\underset{i}{\otimes }%
\left\langle n_{i},n_{i},n_{i}\right\rangle \cong \left\langle \underset{i}{%
\tprod }n_{i},\underset{i}{\tprod }n_{i},\underset{i}{\tprod }%
n_{i}\right\rangle $. \ Using \textit{Proposition 2.6}, we have the
following result.

\bigskip

\begin{proposition}
For a finite set of positive integers $n_{i}$ 
\begin{equation*}
(1)\text{ }\mathfrak{R}\left( \underset{i}{\oplus }\left\langle
n_{i},n_{i},n_{i}\right\rangle \right) \leq \text{ }\underset{i}{\sum }%
\mathfrak{R}\left( \left\langle n_{i},n_{i},n_{i}\right\rangle \right)
\end{equation*}%
and%
\begin{equation*}
(2)\text{ }\mathfrak{R}\left( \underset{i}{\otimes }\left\langle
n_{i},n_{i},n_{i}\right\rangle \right) \leq \text{ }\underset{i}{\tprod }%
\mathfrak{R}\left( \left\langle n_{i},n_{i},n_{i}\right\rangle \right) .
\end{equation*}
\end{proposition}

\bigskip

When $n_{i}=n$, for $1\leq i\leq r$, we shall denote by $\left\langle
n,n,n\right\rangle ^{\otimes r}$ the $r$-fold Kronecker product $\underset{%
1\leq i\leq r}{\otimes }\left\langle n,n,n\right\rangle $. \ By part $(2)$
of \textit{Proposition 2.5} $\left\langle n,n,n\right\rangle ^{\otimes
r}\cong \left\langle n^{r},n^{r},n^{r}\right\rangle $ and $\mathfrak{R}%
\left( \left\langle n,n,n\right\rangle ^{\otimes r}\right) =\mathfrak{R}%
\left( \left\langle n^{r},n^{r},n^{r}\right\rangle \right) $. \ The
following is a relevant proposition.

\bigskip

\begin{proposition}
For positive integers $r,n$, $\mathfrak{R}\left( \left\langle
n^{r},n^{r},n^{r}\right\rangle \right) \leq \mathfrak{R}\left( \left\langle
n,n,n\right\rangle \right) ^{r}.$
\end{proposition}

\bigskip

\subsection{\textit{The Rank of }$2\times 2$\textit{\ Matrix Multiplication
is at most }$7$}

\bigskip

We explain here Strassen's result that the rank of $2\times 2$ matrix
multiplication is at most $7$. \ If $A=\left[ 
\begin{array}{cc}
A_{11} & A_{12} \\ 
A_{21} & A_{22}%
\end{array}%
\right] $ and $B=\left[ 
\begin{array}{cc}
B_{11} & B_{12} \\ 
B_{21} & B_{22}%
\end{array}%
\right] $ are given $2\times 2$ matrices, then by the formulas%
\begin{eqnarray*}
P_{1} &=&\left( A_{11}+A_{22}\right) \left( B_{11}+B_{22}\right) , \\
P_{2} &=&\left( A_{21}+A_{22}\right) B_{11}, \\
P_{3} &=&A_{11}\left( B_{12}-B_{22}\right) , \\
P_{4} &=&\left( -A_{11}+A_{21}\right) \left( B_{11}+B_{12}\right) , \\
P_{5} &=&\left( A_{11}+A_{12}\right) B_{22}, \\
P_{6} &=&A_{22}\left( -B_{11}+B_{21}\right) , \\
P_{7} &=&\left( A_{12}-A_{22}\right) \left( B_{21}+B_{22}\right)
\end{eqnarray*}%
we will be able to recover their product $AB=C=\left[ 
\begin{array}{cc}
C_{11} & C_{12} \\ 
C_{21} & C_{22}%
\end{array}%
\right] $ by the formulas%
\begin{eqnarray*}
C_{11} &=&P_{1}+P_{6}-P_{5}+P_{7}, \\
C_{12} &=&P_{3}+P_{5}, \\
C_{21} &=&P_{2}+P_{6}, \\
C_{22} &=&P_{1}-P_{2}+P_{3}+P_{4},
\end{eqnarray*}%
using a total of $7=\#\left\{
P_{1},P_{2},P_{3},P_{4},P_{5},P_{6},P_{7}\right\} $ multiplications and $18$
additions/subtractions [PAN1984, p. 394].

\bigskip

If we then define $7$ paired linear forms $f_{i}^{\ast },g_{i}^{\ast }$ $%
\epsilon $ $K^{2\times 2^{\ast }},$ $1\leq i\leq 7$, by%
\begin{equation*}
\begin{tabular}{ll}
$f_{1}^{\ast }(A)=A_{11}+A_{22}$ & $g_{1}^{\ast }(B)=B_{11}+B_{22}$ \\ 
$f_{2}^{\ast }(A)=A_{21}+A_{22}$ & $g_{2}^{\ast }(B)=B_{11}$ \\ 
$f_{3}^{\ast }(A)=A_{11}$ & $g_{3}^{\ast }(B)=B_{12}-B_{22}$ \\ 
$f_{4}^{\ast }(A)=-A_{11}+A_{21}$ & $g_{4}^{\ast }(B)=B_{11}+B_{12}$ \\ 
$f_{5}^{\ast }(A)=A_{11}+A_{12}$ & $g_{5}^{\ast }(B)=B_{22}$ \\ 
$f_{6}^{\ast }(A)=A_{22}$ & $g_{6}^{\ast }(B)=-B_{11}+B_{21}$ \\ 
$f_{7}^{\ast }(A)=A_{12}-A_{22}$ & $g_{7}^{\ast }(B)=B_{21}+B_{22}$%
\end{tabular}%
\end{equation*}

then there are matrices $C_{i}$ $\epsilon $ $K^{2\times 2}$, $1\leq i\leq 7$%
, such that%
\begin{equation*}
AB=\underset{i=1}{\overset{7}{\sum }}f_{i}^{\ast }(A)g_{i}^{\ast }(B)C_{i}
\end{equation*}%
hence $\mathfrak{R}\left( \left\langle 2,2,2,\right\rangle \right) \leq 7$
[BCS1997, pp. 10-13]. \ We state this formally, for future reference.

\bigskip

\begin{proposition}
$\mathfrak{R}\left( \left\langle 2,2,2,\right\rangle \right) \leq 7$.
\end{proposition}

\bigskip

In general, if $n=2m$, this algorithm allows us to multiply $2m\times 2m$
matrices with $7$ multiplications and $18$ additions/subtractions of $%
m\times m$ matrices, for $m\geq 1$. \ We define an integer function $%
T_{K}\left( n\right) $ by

\bigskip

\begin{description}
\item[\textit{(2.1)}] $T_{K}\left( n\right) :=$ $\underset{t\text{ }\epsilon 
\text{ }%
\mathbb{Z}
^{+}}{\min }$ \textit{two }$n\times n$ \textit{matrices can be multiplied
using }$t$\textit{\ multiplications, additions, or subtractions}.
\end{description}

\bigskip

If $n$ is some power $2^{m}$ of $2$, then we can partition two $2^{m}\times
2^{m}$ matrices into four $2^{m-1}\times 2^{m-1}$ blocks each, and view
these blocks as inputs to the original algorithm, and by a recursive
application of this procedure we obtain for $T_{K}\left( n\right) $ the
following recursion formula [BCS1997, pp. 12-13].

\bigskip

\begin{description}
\item[\textit{(2.2)}] $T_{K}\left( n\right) \leq 7T_{K}\left( \frac{1}{2}%
n\right) +18\left( \frac{1}{2}n\right) ^{2}$.
\end{description}

In 1971, Winograd proved a stronger result that $\mathfrak{R}\left(
\left\langle 2,2,2,\right\rangle \right) =7$, [WIN1971].

\bigskip

\section{\protect\Large Asymptotic Complexity of Matrix Multiplication}

\bigskip

Here we introduce the exponent $\omega $ describing the asymptotic
complexity of matrix multiplication, including Strassen's estimate of $%
\omega <2.81$, and conclude with some fundamental relations between $\omega $
and the ranks of matrix tensors, which we shall use later in our analysis
and estimates of $\omega $ in \textit{Chapter 6}.

\bigskip

\subsection{\textit{The Exponent of Matrix Multiplication}}

\bigskip

We denote by $M_{K}\left( n\right) $ the total number of arithmetical
operations $\left\{ \times ,+,-\right\} $ needed to multiply $n\times n$
matrices over $K$. \ This is defined more formally as:

\bigskip

\begin{description}
\item[\textit{(2.3)}] $M_{K}\left( n\right) :=L_{K[X,Y]}^{tot}\left(
n\right) :=L_{K[X,Y]}^{tot}\left( \left\{ \underset{1\leq j\leq n}{\sum }%
X_{ij}Y_{jk};\text{ }1\leq i,k\leq n\right\} \right) ,$
\end{description}

\bigskip

where $K[X,Y]$ is the ring of bivariate polynomials over $K$, and the
expression on the right is an \textit{exact} measure of the total number $%
tot $ of arithmetical operations $\left\{ \times ,+,-\right\} $ needed to
multiply two $n\times n$ matrices of a given set of indeterminates $X_{ij},$ 
$Y_{jk}$ $\in $ $K$, $1\leq i,j,k\leq n$, over $K$ without divisions
[BCS1997, p. 108, p. 126, p. 375].

\bigskip

The \textit{exponent of matrix multiplication over }$K$ is the real number $%
\omega \left( K\right) >0$ defined by

\bigskip

\begin{description}
\item[\textit{(2.4)}] $\omega \left( K\right) :=$ $\inf \left\{ h\text{ }%
\epsilon \text{ }%
\mathbb{R}
^{+}\text{ }|\text{ }M_{K}\left( n\right) =O\left( n^{h}\right) ,\text{ }%
n\longrightarrow \infty \right\} $.
\end{description}

\bigskip

The notation $\omega \left( K\right) $ is intended to indicate a possible
dependency on the ground field $K$. \ It has been proved that $\omega \left(
K\right) $ is unchanged if we replace $K$ by any algebraic extension $%
\overline{K}$ [BCS1997, p. 383]. \ It has also been proved that $\omega
\left( K\right) $ is determined only by the characteristic $Char$ $K$ of $K$%
, such that $\omega (K)=\omega (%
\mathbb{Q}
)$ if $Char$ $K=0$, and $\omega (K)=\omega (%
\mathbb{Z}
_{p})$ otherwise, where $%
\mathbb{Z}
_{p}$ is the finite field of integers modulo a prime $p$, of characteristic $%
p$, [PAN1984].. \ Since $Char$ $%
\mathbb{C}
=Char$ $%
\mathbb{R}
=Char$ $%
\mathbb{Q}
=0$, this means that $\omega (%
\mathbb{C}
)=\omega \left( 
\mathbb{R}
\right) =\omega \left( 
\mathbb{Q}
\right) $. \ In this chapter, we shall continue to indicate the ground field
dependency in writing $\omega \left( K\right) $, but in later chapters we
shall drop this formalism and simply write $\omega $, since our concern will
be with complex matrix multiplication, which is general enough for most
purposes.

\bigskip

Returning to $M_{K}\left( n\right) $, by the standard algorithm for $n\times
n$ matrix multiplication, the $n^{2}$ entries $C_{ik}$ of an $n\times n$
matrix product $C=AB$ are given by the formula $C_{ik}=\underset{1\leq j\leq
n}{\sum }A_{ij}B_{jk}$, for all $1\leq i,k\leq n$. \ In using the standard
algorithm, we will be using $n^{3}$ multiplications, and $n^{3}-n^{2}$
additions of the resulting products, which yields an upper estimate $%
M_{K}\left( n\right) =2n^{3}-n^{2}<2n^{3}=O(n^{3})$, i.e. $M_{K}\left(
n\right) <$ $C^{^{\prime }}n^{3}$ for the constant $C^{^{\prime }}=2$, and
implies an upper bound of $3$ for $\omega $ [BCS1997, p. 375]. \ For the
lower bound, we note that since the product of two $n\times n$ matrices
consists of $n^{2}$ entries, one needs to perform a total number of
operations which is \textit{at least} some constant $C\geq 1$ multiple of
the $n^{2}$ entries, which we denote by $M_{K}(n)=\Omega (n^{2})$, and is
equivalent to a lower bound of $2$ for $\omega $ [BCS1997, p. 375]. \ (Our
focus will be on the upper bounds for $M_{K}\left( n\right) $ since we are
interested in worst case complexity.) \ Informally, we have proved the
following elementary result.

\bigskip

\begin{proposition}
For every field $K$, $(1)$ $2\leq \omega \left( K\right) \leq 3$, and $(2)$ $%
\omega \left( K\right) =h$ $\epsilon $ $\left[ 2,3\right] $ iff $\Omega
(n^{h+\epsilon })=M_{K}\left( n\right) =O\left( n^{h+\varepsilon }\right) $,
where $h$ is uniquely minimal for any given degree of precision $\varepsilon
>0$.
\end{proposition}

\bigskip

The connection between the exponent $\omega $ and the concept of bilinear
rank is established by the following important proposition, [BCS1997, pp.
376-377].

\bigskip

\begin{proposition}
For every field $K$%
\begin{equation*}
\omega \left( K\right) =\inf \left\{ h\text{ }\epsilon \text{ }%
\mathbb{R}
^{+}\text{ }|\text{ }\mathfrak{R}\left( \left\langle n,n,n\right\rangle
\right) =O\left( n^{h}\right) ,\text{ }n\longrightarrow \infty \right\} .
\end{equation*}
\end{proposition}

\bigskip

This means for any given degree of precision $\varepsilon >0$, with respect
to a given field $K$, there exists a constant $C_{K,\varepsilon }\geq 1$,
independent of $n$, such that $\mathfrak{R}\left( \left\langle
n,n,n\right\rangle \right) \leq C_{K,\varepsilon }n^{\omega \left( K\right)
+\varepsilon }$ for all $n$. \ It is conjectured that $\omega \left( 
\mathbb{C}
\right) =2$, [CU2003]. \ Henceforth, $\omega $ shall denote $\omega \left( 
\mathbb{C}
\right) $ and in the concluding sections we shall describe some important
relations between $\omega $ and the concept of tensor rank, introduced
earlier, which describe the conditions for realizing estimates of $\omega $
of varying degrees of sharpness.

\bigskip

\subsection{\textit{Relations between the Rank of Matrix Multiplication and
the Exponent }$\protect\omega $}

\bigskip

Taking tensor product powers in the estimate $\mathfrak{R}\left(
\left\langle 2,2,2\right\rangle \right) \leq 7$ (\textit{Proposition 2.10})
we have, by part $(2)$ of \textit{Proposition 2.9},%
\begin{equation*}
\mathfrak{R}\left( \left\langle 2^{n},2^{n},2^{n}\right\rangle \right) =%
\mathfrak{R}\left( \left\langle 2,2,2\right\rangle ^{\otimes n}\right) \leq 
\mathfrak{R}\left( \left\langle 2,2,2\right\rangle \right) ^{n}\leq 7^{n}%
\text{.}
\end{equation*}%
Since for all positive integers $n\geq 2$, $n\leq 2^{\left\lceil \log
_{2}n\right\rceil }=n+\varepsilon _{n}$, where $\varepsilon _{n}>0$ is a
residual depending on $n$, and $\left\lceil \cdot \right\rceil $ denotes the
ceiling function for real numbers, using \textit{Proposition 2.9 }again we
have%
\begin{eqnarray*}
\mathfrak{R}\left( \left\langle n,n,n\right\rangle \right) &\leq &\mathfrak{R%
}\left( \left\langle 2^{\left\lceil \log _{2}n\right\rceil },2^{\left\lceil
\log _{2}n\right\rceil },2^{\left\lceil \log _{2}n\right\rceil
}\right\rangle \right) \\
&\leq &\mathfrak{R}\left( \left\langle 2,2,2\right\rangle \right)
^{\left\lceil \log _{2}n\right\rceil } \\
&\leq &7^{\left\lceil \log _{2}n\right\rceil } \\
&\leq &7n^{\log _{2}7}\approx 7n^{2.807}.
\end{eqnarray*}%
By \textit{Proposition 2.12 }this gives Strassen's estimate $\omega <2.81$
[STR1969, pp. 354-356]. \ The best estimate of $\omega $ is Coppersmith and
Winograd's result that $\omega <2.38$ [CW1990, p. 251].

\bigskip

Assume that $\mathfrak{R}\left( \left\langle n,m,p\right\rangle \right) \leq
s$ for positive integers $n,$ $m,$ $p,$ and $s$. \ By \textit{Proposition
2.1 }and part $(2)$ of \textit{Proposition 2.5}, $\left\langle
nmp,nmp,nmp\right\rangle \cong \left\langle n,m,p\right\rangle \otimes
\left\langle m,p,n\right\rangle \otimes \left\langle p,n,m\right\rangle $. \
Then, we see that%
\begin{eqnarray*}
&&\mathfrak{R}\left( \left\langle nmp,nmp,nmp\right\rangle \right) \\
&=&\mathfrak{R}\left( \left\langle n,m,p\right\rangle \otimes \left\langle
m,p,n\right\rangle \otimes \left\langle p,n,m\right\rangle \right) \\
&\leq &\mathfrak{R}\left( \left\langle n,m,p\right\rangle ^{\otimes 3}\right)
\\
&\leq &\mathfrak{R}\left( \left\langle n,m,p\right\rangle \right) ^{3} \\
&\leq &s^{3}.
\end{eqnarray*}%
i.e. that $\left( nmp\right) ^{\omega }\leq s^{3}$, which is equivalent to $%
\left( nmp\right) ^{\frac{\omega }{3}}\leq s$. \ We have proved the
following result, which we shall repeatedly use later [BCS1997, p. 380]. \
Since $\mathfrak{R}\left( \left\langle n,m,p\right\rangle \right) $ is, by
definition, a positive integer, and $\mathfrak{R}\left( \left\langle
n,m,p\right\rangle \right) \leq \mathfrak{R}\left( \left\langle
n,m,p\right\rangle \right) $, we have proved the following$.$

\bigskip

\begin{proposition}
$\left( nmp\right) ^{\frac{\omega }{3}}\leq \mathfrak{R}\left( \left\langle
n,m,p\right\rangle \right) $ for any positive integers $n,m,p$.
\end{proposition}

\bigskip

This is equivalent to $\omega \leq \frac{\log \mathfrak{R}\left(
\left\langle n,m,p\right\rangle \right) }{\log \left( nmp\right) ^{1/3}}$,
for any positive integers $n,$ $m,$ $p$, a consequence which occurs in a
group-theoretic context as shown in \textit{Chapter 4}. \ Informally, we can
understand $nmp$ to be "size" of $n\times m$ by $m\times p$ matrix
multiplication, and $\left( nmp\right) ^{\frac{1}{3}}$ to be the (geometric)
mean of this size. \ The above proposition has as a generalization the
following statement.

\bigskip

\begin{proposition}
$\underset{i}{\sum }\left( n_{i}m_{i}p_{i}\right) ^{\frac{\omega }{3}}\leq 
\mathfrak{R}\left( \underset{i}{\oplus }\left\langle
n_{i},m_{i},p_{i}\right\rangle \right) $, for any finite set of positive
integer triples $n_{i},$ $m_{i},$ $p_{i}$.
\end{proposition}

\bigskip

This is a formulation in terms of ordinary rank $\mathfrak{R}$ of Sch\"{o}%
nhage's asymptotic direct sum inequality involving the related but
approximative concept of border rank \underline{$\mathfrak{R}$}, which we
shall not discuss further [BCS1997, p. 380]. \ In essence, \textit{%
Proposition 2.14 }means that the complexity of several, simultaneous
independent matrix multiplications is at least the sum of the mean sizes of
the multiplications to the power $\omega $, a consequence which occurs in a
group-theoretic context as shown in \textit{Chapter 5}.

\pagebreak \bigskip

\chapter{{\protect\huge Basic Representation Theory}\protect\bigskip}

We start with some basic theory of representations and character theory of
finite groups, focusing in particular on various relations and estimates for
sums of powers of the distinct irreducible group character degrees, which
will be important in the central analysis in \textit{Chapter 4}. \ Then we
proceed to describe basic facts about multiplicative complexity in regular
group algebras, and conclude with an outline of the discrete Fourier
transforms for groups and their computational complexities.

\section{Basic Representation Theory and Character Theory of Finite Groups}

\bigskip

\subsection{\textit{Representations, }$%
\mathbb{C}
G$\textit{-Modules and Characters}}

\bigskip

In this thesis, a (finite-dimensional) \textit{representation\ }$\pi $ of a
finite group $G$ is defined as a group homomorphism $\pi :G\longrightarrow
GL(V)$, where $V$ is a finite-dimensional, complex vector space, and where $%
GL(V)$ is the group of all linear operators mapping $V$ to itself. \ In
particular, when $V=%
\mathbb{C}
^{n}$ then $GL\left( V\right) =GL\left( n,%
\mathbb{C}
\right) $, the group of all invertible $n\times n$ complex matrices, and $%
\pi $ is called a \textit{matrix representation} of $G$. \ For example, $G$
always has the trivial representation $\iota $ on $V$, defined by $%
g\longrightarrow 1_{GL(V)},$ whenever $g$ $\epsilon $ $G$, where $1_{GL(V)}$
is the identity automorphism of $V$. \ If $\pi $ is a representation of $G$
we will call $V$ the \textit{target space} of $\pi $, and define the \textit{%
dimension of }$\pi $ to be the dimension of $V$, i.e. $Dim$ $\pi :=Dim$ $V$.
\ For each $g$ $\epsilon $ $G$, $\pi \left( g\right) $ is an automorphism $%
V\longrightarrow V$ of $V$, and we note that $\pi $ satisfies $\pi (gh)=\pi
(g)\pi (h)$, $\pi (g^{-1})=\left( \pi (g)\right) ^{-1}$, for all $g,h$ $%
\epsilon $ $G$, and that in particular, $\pi (1_{G})=1_{GL\left( V\right) }$%
, the trivial automorphism of $V$. \ A representation of $G$, the so-called (%
\textit{right}) \textit{regular representation}, exists when $V=$ $%
\mathbb{C}
^{G}=\left\{ f\text{ }|\text{ }f:G\longrightarrow 
\mathbb{C}
\right\} $, the $\left\vert G\right\vert $-dimensional, associative $%
\mathbb{C}
$-algebra of all complex-valued maps $f:$ $G\longrightarrow 
\mathbb{C}
$ on $G$, with standard basis $\mathcal{B}_{G}=\left\{ e_{g}\text{ }|\text{ }%
e_{g}\text{ }\epsilon \text{ }%
\mathbb{C}
^{G}\text{, }e_{g}(h)=\delta _{g,h}\text{, }h\text{ }\epsilon \text{ }%
G\right\} $ of the $\left\vert G\right\vert $ indicator maps $%
e_{g}\longleftrightarrow g$ $\epsilon $ $G$. \ $%
\mathbb{C}
^{G}$ can be identified with the set $%
\mathbb{C}
G=\left\{ \underset{g\text{ }\epsilon \text{ }G}{\sum }f_{g}g\text{ }|\text{ 
}f\text{ }\epsilon \text{ }%
\mathbb{C}
^{G},f(g)\equiv f_{g}\right\} $ of all formal linear sums of group elements
with coefficients as their $f$-values, for each $f$ $\epsilon $ $%
\mathbb{C}
^{G}$. \ $%
\mathbb{C}
G$ constitutes a $\left\vert G\right\vert $-dimensional vector space over $%
\mathbb{C}
$ and it is a $%
\mathbb{C}
$-algebra called the \textit{regular group algebra of }$G$, admitting as
basis elements the elements of $G$ itself, yielding the \textit{regular basis%
}.\ The endomorphisms $f=\underset{g^{^{\prime }}\text{ }\epsilon \text{ }G}{%
\sum }f_{g^{^{\prime }}}g^{^{\prime }}\mapsto gf:=\underset{g^{^{\prime }}%
\text{ }\epsilon \text{ }G}{\sum }f_{g^{^{\prime }}}\left( gg^{^{\prime
}}\right) $, $f$ $\epsilon $ $%
\mathbb{C}
G$, arbitrary fixed $g$ $\epsilon $ $G$, describe permutation left-actions
on regular basis components of $f$ $\epsilon $ $%
\mathbb{C}
G$ for elements $g$ $\epsilon $ $G$ via their uniquely associated
permutations $\mu _{g}$ $\epsilon $ $Sym_{\left\vert G\right\vert }$, and
have unique associated permutation matrices $\left[ g\right] _{G}$ $\epsilon 
$ $GL(\left\vert G\right\vert ,%
\mathbb{C}
)$. \ The regular representation of $G$, denoted by $\rho _{%
\mathbb{C}
G}$, is the mapping $g\longmapsto \left[ g\right] _{G}=\rho _{%
\mathbb{C}
G}(g)$, $g$ $\epsilon $ $G$.

\bigskip

Let $\pi $ be a representation of $G$ on a vector space $V$. \ If $W$ is a
subspace of $V$ which is invariant under the automorphisms $\pi (g)$, i.e. $%
\pi (g)(W)\subseteq W$, for all $g$ $\epsilon $ $G$, then $W$ is called $\pi 
$\textit{-invariant}. \ The restriction $\pi \downarrow W$ of $\pi $ to a $%
\pi $-invariant subspace $W$ of $V$ produces a representation $\rho $ of $G$
on $W$ called a \textit{subrepresentation} of $\pi $, which can be called a 
\textit{component }(representation) of $\pi $, and conversely, every
subrepresentation $\rho $ of a representation $\pi $ of $G$ on $V$ is the
restriction $\pi \downarrow W$ of $\pi $ to some $\pi $-invariant subspace $%
W $ of $V$ depending on $\rho $: $\rho $ is given by $\rho (g)(w)=\pi (g)(w)$%
, for all $g$ $\epsilon $ $G$, $w$ $\epsilon $ $W$. \ It follows that the
dimension of a subrepresentation of a representation of $G$ cannot exceed
the dimension of the representation. \ $\pi $ is called \textit{irreducible}
iff it contains no nontrivial subrepresentations, otherwise it is called 
\textit{reducible}. \ $G$ admits always has the trivial $1$-dimensional
irreducible representation $\iota _{1}$, as defined by $g\longrightarrow
\left( 1\right) $, and conversely any $1$-dimensional representation is
irreducible; equivalently, the dimension of a reducible representation is at
least $2$. \ If $\rho $ is any other representation of $G$ on a vector space 
$W$, then $\pi $ and $\rho $ are called \textit{equivalent }(notation $\pi
\sim \rho $) iff there is a vector space isomorphism $T:V\cong W$ such that $%
T(gv):=T(\pi (g)(v))=\rho (g)T(v)=:gT(v)$, for all $g$ $\epsilon $ $G$, $v$ $%
\epsilon $ $V$. \ Otherwise, i.e. if such a $T$ does not exist for $\pi $
and $\rho $, they are called \textit{inequivalent}, and we denote this by $%
\pi \nsim \rho $. \ Equivalence of representations is an equivalence
relation. \ If $\varrho $ and $\varsigma $ are two irreducible
representations of $G$ then we define the delta quantity $\delta _{\varrho
,\varsigma }$ as $\delta _{\varrho ,\varsigma }=1$ iff $\varrho \sim
\varsigma $ and $\delta _{\varrho ,\varsigma }=0$ iff $\varrho \nsim
\varsigma $.

\bigskip

Given a representation $\pi $ of $G$ on a vector space $V$, there is a
naturally defined multiplication map $G\times V\longrightarrow V$ describing
left-action $\left( g,v\right) \longmapsto gv:=\pi \left( g\right) v$ of $G$
on $V$, which is associative: $(gh)(v)=g(hv)$; has a natural identity: $%
1_{G}v=v$ for all $v$ $\epsilon $ $V$; is invertible: $%
v=g^{-1}(gv)=g(g^{-1}v)$; is homogenous with respect to scalar multiples of
vectors: $g(\lambda v)=\lambda (gv)$; and is right-linear with respect to $%
\mathbb{C}
G$-multiplication: $g(v+v^{^{\prime }})=gv+gv^{^{\prime }}$; for all
elements $g,h$ $\epsilon $ $G$, vectors $v,$ $v^{^{\prime }}$ $\epsilon $ $V$%
, scalars $\lambda $ $\epsilon $ $%
\mathbb{C}
$. \ The space $V$ under this multiplication is called a $%
\mathbb{C}
G$\textit{-module}, and its dimension is its dimension as a vector space. \
A special kind of $%
\mathbb{C}
G$-module, called the \textit{regular} $%
\mathbb{C}
G$-module, occurs when $V=$ $%
\mathbb{C}
G=\left\{ \underset{g\text{ }\epsilon \text{ }G}{\sum }\lambda _{g}g\text{ }|%
\text{ }\lambda _{g}\text{ }\epsilon \text{ }%
\mathbb{C}
\right\} $, the regular group algebra of $G$, discussed above. \ A subspace $%
W$ of $V$ is called a $%
\mathbb{C}
G$\textit{-submodule} of $V$ iff it is closed under the map $G\times
V\longrightarrow V$ via $\pi $. \ A $%
\mathbb{C}
G$-submodule $W$ of $V$ under the representation $\pi $ always corresponds
to some subrepresentation $\rho $ of $\pi $. $\ V$ and the zero subspace $%
\left\{ \mathbf{O}_{V}\right\} $ always form trivial $%
\mathbb{C}
G$-submodules of $V$, and $V$ is an \textit{irreducible }$%
\mathbb{C}
G$\textit{-module} iff it contains no nontrivial $%
\mathbb{C}
G$-submodules. \ Two $%
\mathbb{C}
G$-modules $V$ and $W$, corresponding to representations $\pi $ and $\rho $,
respectively, of $G$, are called equivalent iff $\pi $ is equivalent to $%
\rho $, otherwise they are called inequivalent.

\bigskip

Given a representation $\pi $ of $G$ on a space $V$, the \textit{character} $%
\chi _{\pi }$ of $\pi $ (also, the character of $V$ as a $%
\mathbb{C}
G$-module) stands for the map $G\longrightarrow 
\mathbb{C}
$ defined by $\chi _{\pi }(g):=Tr\left( \pi (g)\right) $, $g$ $\epsilon $ $G$%
, where $Tr(\cdot )$ is the trace map for operators. \ The \textit{degree} $%
d_{\pi }$ of the character $\chi _{\pi }$ is defined to be the dimension of
its underlying representation $\pi $, i.e. $d_{\pi }:=\chi _{\pi
}(1_{G})=Tr(\pi (1_{G}))=Tr(1_{GL(V)})=Dim$ $\pi =Dim$ $V$. \ For example,
the character $\chi _{\iota _{1}}$ of the trivial irreducible representation 
$\iota _{1}$ is defined by $g\longmapsto 1,$ $g$ $\epsilon $ $G$. \ The set
of all characters $\chi _{\pi }$ of $G$ is denoted by $\widehat{G}$. \
Characters of degree $1$ are called linear. \ A character $\chi _{\pi }$ is
said to be irreducible iff its underlying representation $\pi $ is
irreducible, otherwise it is said to be reducible. \ All linear characters
are irreducible. \ The characters $\chi _{\pi }$ and $\chi _{\rho }$ of
equivalent representations $\pi $ and $\rho $, respectively, are the same,
and, conversely, $\pi $ and $\rho $ are equivalent if $\chi _{\pi }=\chi
_{\rho }$. \ The character of the regular representation $\rho _{%
\mathbb{C}
G}$ of $G$, called the \textit{regular character} of $G$, denoted by $\chi _{%
\mathbb{C}
G}$, is the map $G\longrightarrow 
\mathbb{C}
$ defined by $g\longmapsto Tr(\chi _{%
\mathbb{C}
G}(g))$, $g$ $\epsilon $ $G$. \ Observe that $\chi _{%
\mathbb{C}
G}$ takes the value $\left\vert G\right\vert $ if $g=1_{G}$, and $0$
otherwise. \ It holds that $d_{\chi _{%
\mathbb{C}
G}}=\left\vert G\right\vert $. \ We define the \textit{inner product} of two
characters $\chi _{\pi }$ and $\chi _{\rho }$ of $G$ by $\left\langle \chi
_{\pi },\chi _{\rho }\right\rangle =\left\vert G\right\vert ^{-1}\underset{g%
\text{ }\epsilon \text{ }G}{\sum }\chi _{\pi }\left( g\right) \overline{\chi
_{\rho }\left( g\right) }$.

\subsection{\textit{Canonical Decompositions for Regular Representations, }$%
\mathbb{C}
G$\textit{-Modules, and Characters}}

\bigskip

A representation $\pi $ of $G$ is called \textit{completely reducible} iff
its target $%
\mathbb{C}
G$-module $V$ is the direct sum $V=\underset{\varrho _{\pi }\text{
irreducible}}{\oplus }R_{\varrho _{\pi }}$ of a finite number of irreducible 
$%
\mathbb{C}
G$-modules $R_{\varrho _{\pi }}$, in which case $\pi $ is the direct sum of
a finite number of irreducible representations $\varrho _{\pi }$ of $G$,
occuring with multiplicity $l_{\pi ,\varrho }=\left\langle \chi _{\pi
},\varrho _{\pi }\right\rangle >0$, called the \textit{irreducible components%
} of $\pi $, of dimensions $d_{\varrho _{\pi }}$. \ The following is a
fundamental theorem in this regard [SER1977].

\bigskip

\begin{theorem}
(Maschke)\qquad Every finite-dimensional representation of a finite group is
completely reducible.
\end{theorem}

\bigskip

The character $\chi _{\pi }$ of $\pi $ is said to completely reducible iff $%
\pi $ is completely reducible, and will take the form $\chi _{\pi }=$ $%
\underset{\varrho _{\pi }\text{ irreducible}}{\sum }\chi _{\varrho _{\pi }}$%
, and will have degree $d_{\pi }=\underset{\varrho _{\pi }\text{ irreducible}%
}{\sum }d_{\varrho _{\pi }}$, which is, by definition, the dimension of $\pi 
$ and of $V$.

\bigskip

The following theorem is crucial here [HUP1998], [SER1977].

\bigskip

\begin{theorem}
(Frobenius)\qquad If $G$ is a finite group then $(1)$ the number of its
distinct irreducible representations $\varrho $ is equal to the number of
its distinct conjugacy classes, denoted by $c(G)$, which is called its class
number. \ $(2)$ The characters $\chi _{\varrho }$ of the irreducible
representations $\varrho $ form an orthonormal basis of $%
\mathbb{C}
G$, i.e. $\left\langle \chi _{\varrho },\chi _{\varsigma }\right\rangle
=\delta _{\varrho ,\varsigma }$, for distinct irreducible representations $%
\varrho $ and $\varsigma $. \ $(3)$ An arbitrary representation $\pi $ of $G$%
, or its target $%
\mathbb{C}
G$-module $V_{\pi }$, is irreducible iff its character $\chi _{\pi }$
satisfies $\left\langle \chi _{\pi },\chi _{\pi }\right\rangle =1$. \ $(4)$ $%
\left\vert \chi _{\varrho }(g)\right\vert \leq \left\vert \chi _{\varrho
}(1_{G})\right\vert =d_{\varrho }$, $g$ $\epsilon $ $G$, for any irreducible
character $\chi _{\varrho }$ of $G$.
\end{theorem}

\bigskip

There are exactly $c(G)$ distinct irreducible representations $\varrho $
upto equivalence, $c(G)$ distinct irreducible $%
\mathbb{C}
G$-modules $R_{\varrho }$ upto equivalence, and $c(G)$ distinct irreducible
characters of $G$ upto equivalence. \ The collections of these are denoted
by $Irrep(G)$, $Irr_{%
\mathbb{C}
G}(G)$, $Irr(G)$, respectively. \ It holds that for any $\varrho $, $%
\varsigma $ $\epsilon $ $Irrep(G)$, $\left\langle \chi _{\varrho },\chi
_{\varsigma }\right\rangle =\delta _{\varrho ,\varsigma }$. \ We note the
following result, [HUP1998].

\bigskip

\begin{theorem}
\qquad $\left\langle \chi _{%
\mathbb{C}
G},\chi _{\varrho }\right\rangle =\chi _{\varrho }(1_{G})=d_{\varrho }$, for
each $\chi _{\varrho }$ $\epsilon $ $Irr(G)$.
\end{theorem}

\bigskip

$\left\langle \chi _{%
\mathbb{C}
G},\chi _{\varrho }\right\rangle $ is the multiplicity of an $\varrho $ $%
\epsilon $ $Irrep(G)$ in $\rho _{%
\mathbb{C}
G}$, which we denote by $l_{%
\mathbb{C}
G,\varrho }$, the above result states that every $\varrho $ $\epsilon $ $%
Irrep(G)$ occurs in the regular representation $\rho _{%
\mathbb{C}
G}$ exactly $l_{%
\mathbb{C}
G,\varrho }=Dim$ $\varrho =d_{\varrho }$ number of times, where $l_{%
\mathbb{C}
G,\varrho }\geq 1$, since $d_{\varrho }\geq 1$. \ This means, equivalently,
that every distinct irreducible $%
\mathbb{C}
G$-module $R_{\varrho }$ $\epsilon $ $Irr_{%
\mathbb{C}
G}(G)$, and distinct irreducible character $\chi _{\varrho }$ $\epsilon $ $%
Irr(G)$, occur in $%
\mathbb{C}
G$ and $\chi _{%
\mathbb{C}
G}$, respectively, all occur with positive multiplicity equal to $l_{%
\mathbb{C}
G,\varrho }=d_{\varrho }$. \ The $\rho _{%
\mathbb{C}
G}$, $%
\mathbb{C}
G$, and $\chi _{%
\mathbb{C}
G}$ decompose into a direct sum of \textit{isotypic components} $\left\{ 
\underset{i_{\varrho }=1}{\overset{d_{\varrho }}{\oplus }}\varrho \text{ }|%
\text{ }\varrho \text{ }\epsilon \text{ }Irrep(G)\right\} $, $\left\{ 
\underset{i_{\varrho }=1}{\overset{d_{\varrho }}{\oplus }}R_{\varrho }\text{ 
}|\text{ }\varrho \text{ }\epsilon \text{ }Irrep(G)\right\} $, and $\left\{ 
\underset{i_{\varrho }=1}{\overset{d_{\varrho }}{\oplus }}\chi _{\varrho }%
\text{ }|\text{ }\varrho \text{ }\epsilon \text{ }Irrep(G)\right\} $
respectively, in the following way:

\bigskip

\begin{description}
\item[\textit{(3.1)}] $\rho _{%
\mathbb{C}
G}=\underset{\varrho \text{ }\epsilon \text{ }Irrep(G)}{\oplus }d_{\varrho
}\varrho =\underset{\varrho \text{ }\epsilon \text{ }Irrep(G)}{\oplus }%
\underset{i_{\varrho }=1}{\overset{d_{\varrho }}{\oplus }}\varrho $,
\end{description}

\bigskip

\begin{description}
\item[\textit{(3.2)}] $%
\mathbb{C}
G=\underset{\varrho \text{ }\epsilon \text{ }Irrep(G)}{\oplus }d_{\varrho
}R_{\varrho }=\underset{\varrho \text{ }\epsilon \text{ }Irrep(G)}{\oplus }%
\underset{i_{\varrho }=1}{\overset{d_{\varrho }}{\oplus }}R_{\varrho }$,
\end{description}

\bigskip

\begin{description}
\item[\textit{(3.3)}] $\chi _{%
\mathbb{C}
G}=\underset{\varrho \text{ }\epsilon \text{ }Irrep(G)}{\sum }d_{\varrho
}\chi _{\varrho }=\underset{\varrho \text{ }\epsilon \text{ }Irrep(G)}{\sum }%
\underset{i_{\varrho }=1}{\overset{d_{\varrho }}{\sum }}\chi _{\varrho }$.
\end{description}

\bigskip

If we put $d_{%
\mathbb{C}
G}=Dim$ $\rho _{%
\mathbb{C}
G}=Dim$ $%
\mathbb{C}
G=\chi _{%
\mathbb{C}
G}(1_{G})$, then we get:

\bigskip

\begin{description}
\item[\textit{(3.4)}] $d_{%
\mathbb{C}
G}=\left\vert G\right\vert =\underset{d_{\varrho }\text{ }\epsilon \text{ }%
cd(G)}{\sum }l_{%
\mathbb{C}
G,\varrho }d_{\varrho }=\underset{d_{\varrho }\text{ }\epsilon \text{ }cd(G)}%
{\sum }d_{\varrho }^{2}$.
\end{description}

\bigskip

This shows that $d_{\varrho }^{2}$ $\leq $ $\left\vert G\right\vert $, for
any $\varrho $ $\epsilon $ $Irrep(G)$.

\bigskip

\subsection{\textit{Induced and Restricted Representations, }$%
\mathbb{C}
G$\textit{-Modules, and Characters}}

\bigskip

A representation $\pi $ of a group $G$, with target $%
\mathbb{C}
G$-module $V_{\pi }$ and character $\chi _{\pi }$, is said to be an \textit{%
induction} of a representation $\theta $ of a subgroup $H\leq G$ if $V_{\pi
}=\underset{\kappa \text{ }\epsilon \text{ }G/H}{\oplus }W_{\kappa }$, where 
$G/H$ stands for the set of $[G:H]=\frac{\left\vert G\right\vert }{%
\left\vert H\right\vert }$ distinct left $G$-cosets $\kappa =s_{\kappa }H$
of $H$ with the $s_{\kappa }$ $\epsilon $ $G$ being their distinct coset
representatives, $W_{1_{G}}\equiv W$ is a $Res_{H}^{G}\pi $-invariant
subspace of $V$ such that $\theta $ is the representation $\theta
:H\longrightarrow GL(Dim$ $W,%
\mathbb{C}
)$; $\left\{ W_{\kappa }\right\} _{\kappa \text{ }\epsilon \text{ }G/H}$ is
the set of $[G:H]$ distinct, $Dim$ $W$-dimensional subspaces of $V_{\pi }$
given by $W_{\kappa }=\pi (s_{\kappa })(W)$, $\kappa $ $\epsilon $ $G/H$
[SER1977]. \ In this case, every $v$ $\epsilon $ $V_{\pi }$ has the form $%
\underset{\kappa \text{ }\epsilon \text{ }G/H,\text{ }w_{\kappa }\text{ }%
\epsilon \text{ }W_{\kappa }}{\sum }w_{\kappa }$; $\pi $ is the direct sum $%
\pi =\underset{\kappa \text{ }\epsilon \text{ }G/H}{\oplus }\theta _{\kappa
} $ of the subrepresentations $\theta _{\kappa }$ of $G$ on the components $%
W_{\kappa }$ of $V_{\pi }$ of the form $g\longmapsto \left[ g\right]
_{\kappa }\left( w_{\kappa }\right) $, $\kappa $ $\epsilon $ $G/H$, $\left[ g%
\right] _{\kappa }$ $\epsilon $ $GL(Dim$ $W_{\kappa },%
\mathbb{C}
)$; $\pi $ is said to be induced by $\theta $ and denoted by $\pi
=Ind_{H}^{G}\theta $; $W$ becomes a $%
\mathbb{C}
G$-module $W_{\theta }$ and a $%
\mathbb{C}
G$-submodule of $V_{\pi }$; $V_{\pi }$ is of dimension $[G:H]\cdot Dim$ $%
W_{\theta }$ and is said to be induced by $W_{\theta }$ and denoted by $%
V_{\pi }=Ind_{H}^{G}W_{\theta }$. \ $Ind_{H}^{G}\theta $ will have the
canonical decomposition $\underset{\varrho \text{ }\epsilon \text{ }Irrep(G)}%
{\oplus }l_{Ind_{H}^{G}\theta ,\varrho }\varrho $, its character $\chi
_{Ind_{H}^{G}\theta }$ will have the decomposition $\chi _{Ind_{H}^{G}\theta
}=\underset{\varrho \text{ }\epsilon \text{ }Irrep(G)}{\sum }%
l_{Ind_{H}^{G}\theta ,\varrho }\chi _{\varrho }$, and dimension $%
d_{Ind_{H}^{G}\theta }=\underset{\varrho \text{ }\epsilon \text{ }Irrep(G)}{%
\sum }l_{Ind_{H}^{G}\theta ,\varrho }d_{\varrho }=[G:H]\cdot Dim$ $W_{\theta
}=\frac{\left\vert G\right\vert }{\left\vert H\right\vert }d_{\theta }$,
where $l_{Ind_{H}^{G}\theta ,\varrho }=\left\langle \chi _{Ind_{H}^{G}\theta
},\chi _{\varrho }\right\rangle $ are the multiplicities of the $\varrho $ $%
\epsilon $ $Irrep(G)$ in $Ind_{H}^{G}\theta $, not all simultaneously equal
to $0$ [SER1977]. \ The following proposition summarizes some elementary
properties of dimensions of induced representations.

\bigskip

\begin{proposition}
For any representation $\theta $ of a subgroup $H\leq G$: $(1)$ $%
d_{Ind_{H}^{G}\theta }=\frac{\left\vert G\right\vert }{\left\vert
H\right\vert }d_{\theta }$; for any $\varrho $ $\epsilon $ $Irrep(G)$, $(2)$ 
$d_{\varrho }$ $|$ $d_{Ind_{H}^{G}\theta }$; $(3)$ $d_{\varrho
}=d_{Ind_{H}^{G}\theta }$ iff $l_{Ind_{H}^{G}\theta ,\varsigma }=\delta
_{\varsigma ,\varrho }$ for all $\varsigma $ $\epsilon $ $Irrep(G)$, i.e.
iff $Ind_{H}^{G}\theta $ is an irreducible representation of $G$ equivalent
to (or coinciding with) exactly one $\varrho $ $\epsilon $ $Irrep(G)$.
\end{proposition}

\bigskip

A representation $\xi $ of a subgroup $H\leq G$ is said to be the \textit{%
restriction} of a representation $\pi $ of $G$ if $\xi (h)=\pi (h)$ for all $%
h$ $\epsilon $ $H$, and this is denoted by $\xi =Res_{H}^{G}\pi $, in which
case the target $%
\mathbb{C}
G$-module $U_{\xi }$ of $\xi $ is said to be a restriction of the $%
\mathbb{C}
G$-module $V_{\pi }$ of $\pi $ denoted by $U_{\xi }=Res_{H}^{G}V_{\pi }$
[SER1977]. \ $Res_{H}^{G}\pi $ will have the canonical decomposition $%
\underset{\vartheta \text{ }\epsilon \text{ }Irrep(H)}{\oplus }%
l_{Res_{H}^{G}\pi ,\vartheta }\vartheta $, its character $\chi
_{Res_{H}^{G}\pi }$ will have the decomposition $\chi _{Res_{H}^{G}\pi }=%
\underset{\vartheta \text{ }\epsilon \text{ }Irrep(H)}{\sum }%
l_{Res_{H}^{G}\pi ,\vartheta }\chi _{\vartheta }$, and it will have the
dimension $d_{Res_{H}^{G}\pi }=d_{\pi }=\underset{\vartheta \text{ }\epsilon 
\text{ }Irrep(H)}{\sum }l_{Res_{H}^{G}\pi ,\vartheta }d_{\vartheta }$, where 
$l_{Res_{H}^{G}\pi ,\vartheta }=\left\langle \chi _{Res_{H}^{G}\pi },\chi
_{\vartheta }\right\rangle $ are the nonnegative multiplicities of the $%
\vartheta $ $\epsilon $ $Irrep(H)$ in $Res_{H}^{G}\pi $, not all
simultaneously equal to $0$. \ The following proposition summarizes some
elementary properties of dimensions of restricted representations.

\bigskip

\begin{proposition}
For any representation $\pi $ of $G\geq H$, $(1)$ $d_{\pi
}=d_{Res_{H}^{G}\pi }$; for any $\vartheta $ $\epsilon $ $Irrep(H)$, $(2)$ $%
d_{\vartheta }\leq $ $d_{\pi }$ if $l_{Res_{H}^{G}\pi ,\vartheta }>0$; $(3)$ 
$d_{\vartheta }=d_{\pi }$ iff $l_{Res_{H}^{G}\pi ,\varphi }=\delta _{\varphi
,\vartheta }$ for all $\varphi $ $\epsilon $ $Irrep(H)$, i.e. iff $%
Res_{H}^{G}\pi $ is an irreducible representation of $H$ equivalent to
(coinciding with) exactly one $\vartheta $ $\epsilon $ $Irrep(H)$.
\end{proposition}

\bigskip

The following is a famous theorem of Frobenius describing a reciprocity
between induction and restriction of irreducible representations.

\bigskip

\begin{theorem}
(Frobenius)\qquad For a subgroup $H\leq G$, and any $\vartheta $ $\epsilon $ 
$Irrep(H)$ and $\varrho $ $\epsilon $ $Irrep(G)$, $l_{Res_{H}^{G}\varrho
,\vartheta }=l_{\varrho ,Ind_{H}^{G}\vartheta }$.
\end{theorem}

\bigskip

This is called the \textit{Frobenius reciprocity law}, and in this thesis is
useful in deriving information about size estimates relating to the
dimensions of the distinct irreducible representations of groups.

\bigskip

\subsection{\textit{Irreducible Character Degrees}}

\bigskip

Henceforth, let us denote by $cd(G)$ the set of degrees of the distinct
irreducible characters of a finite group $G$. \ Formally:

\bigskip

\begin{description}
\item[\textit{(3.5)}] $cd(G):=\left\{ d_{\varrho }=\chi _{\varrho }(1_{G})%
\text{ }|\text{ }\chi _{\varrho }\text{ }\epsilon \text{ }Irr(G)\text{, }%
\varrho \text{ }\epsilon \text{ }Irrep(G)\right\} $.
\end{description}

\bigskip

We call $cd(G)$ the \textit{character degree set} of $G$, where we know $%
\left\vert cd(G)\right\vert =\left\vert Irrep(G)\right\vert =c(G)$ holds. \
The following is a fundamental theorem about irreducible character degrees
[HUP1998].

\bigskip

\begin{theorem}
\qquad For any $d_{\varrho }$ $\epsilon $ $cd(G)$, $(1)$ $d_{\varrho }$
divides $[G:A]$ where $[G:A]$ is the index of any maximal Abelian normal
subgroup $A$ of $G$ (It\^{o}), $(2)$ $d_{\varrho }$ divides $[G:Z(G)]$,
where $Z(G)$ is the centre of $G$, and $(3)$ $d_{\varrho }$ divides $%
\left\vert G\right\vert $.
\end{theorem}

\bigskip

$(2)$ is a consequence of $(1)$ since if $\varrho $ is irreducible, then $%
d_{\varrho }$ must divide $[G:Z(G)]$ where $Z(G)\leq A$ for some maximal
Abelian normal subgroup $A$ of $G$. \ $(3)$ is a consequence of $(2)$ since
if $\varrho $ is irreducible then $d_{\varrho }$ $|$ $[G:Z(G)]$ by $(2)$ and
therefore $d_{\varrho }$ $|$ $\left\vert G\right\vert =[G:Z(G)]$ $\left[
Z(G):1_{G}\right] $.

\bigskip

Let us consider the irreducible character degrees of finite Abelian groups.
\ Until section \textit{3.1.5 }$G$ will be a finite Abelian group. \ Then
every conjugacy class $g^{G}$ of every element $g$ $\epsilon $ $G$ contains
only $g$ as its element, and therefore, there are in $G$ exactly as many
distinct conjugacy classes as there are elements, i.e. $\left\vert
cd(G)\right\vert =c(G)=\left\vert G\right\vert $. \ Also, $Z(G)=G$, $%
[G:Z(G)]=1$, and by \textit{Theorem 3.7}, for any $d_{\varrho }$ $\epsilon $ 
$cd(G)$,\textit{\ }$d_{\varrho }$ $|$ $[G:G]=1$, which means that $%
d_{\varrho }=1$, i.e. the dimension of every irreducible representation and
every irreducible $%
\mathbb{C}
G$-module, and the degree of every irreducible character, of a finite
Abelian group $G$ is $1$. \ Thus, the dimension $d_{\pi }$ of an arbitrary
representation $\pi $ of $G$, and of its target $%
\mathbb{C}
G$-module $V_{\pi }$, is given by $d_{\pi }=\underset{\varrho \text{ }%
\epsilon \text{ }Irrep(G)}{\sum }\left\langle \chi _{\pi },\chi _{\varrho
}\right\rangle d_{\varrho }=\underset{\varrho \text{ }\epsilon \text{ }%
Irrep(G)}{\sum }\left\langle \chi _{\pi },\chi _{\varrho }\right\rangle =%
\underset{\varrho \text{ }\epsilon \text{ }Irrep(G)}{\sum }l_{\pi ,\varrho }$%
, i.e. the sum of the multiplicities of the irreducible components of $\pi $%
. \ For each $\varrho $ $\epsilon $ $Irrep(G)$, and a given $g$ $\epsilon $ $%
G$, the matrix $\varrho (g)$ of $\varrho $ at $g$ will be a $1\times 1$
matrix given by $\left( \chi _{\varrho }(g)\right) $, and its character $%
\chi _{\varrho }$ $\epsilon $ $Irr(G)$ will have an inverse $\chi _{\varrho
}^{-1}$, defined by $\chi _{\varrho }^{-1}\left( g\right) =\chi _{\varrho
}\left( g^{-1}\right) =\overline{\chi _{\varrho }\left( g\right) }$, so that 
$\chi _{\varrho }^{-1}=\overline{\chi _{\varrho }}$; the character set $%
\widehat{G}$ will then form, under pointwise multiplication, a
multiplicative Abelian group isomorphic to $G$, called its \textit{dual }or 
\textit{character group}.

\bigskip

\subsection{\textit{Estimates for Sums of Powers of Irreducible Character
Degrees}}

\bigskip

We derive here a number of estimates relating to sums of powers of
irreducible character degrees of a finite group $G$, for which we introduce
a map $D_{r}(G):%
\mathbb{R}
^{+}\longrightarrow 
\mathbb{R}
^{+}$ defined by:

\bigskip

\begin{description}
\item[\textit{(3.6)}] $D_{r}(G)=\underset{d_{\varrho }\text{ }\epsilon \text{
}cd(G)}{\sum }d_{\varrho }^{r}=\underset{\varrho \text{ }\epsilon \text{ }%
Irrep(G)}{\sum }d_{\varrho }^{r}=\underset{\chi _{\varrho }\text{ }\epsilon 
\text{ }Irr(G)}{\sum }\chi _{\varrho }(1_{G})^{r}$,\qquad \qquad \qquad
\qquad $r\geq 1$.
\end{description}

\bigskip

$D_{r}(G)$ records the sum of the $r^{th}$ powers of the distinct
irreducible character degrees of $G$. \ For $r=0$, $D_{0}(G)=c(G)$. \ For $%
t>0$ the $t^{th}$ power of $D_{r}(G)$ is given by $D_{r}(G)^{t}=\left( 
\underset{d_{\varrho }\text{ }\epsilon \text{ }cd(G)}{\sum }d_{\varrho
}^{r}\right) ^{t}$, $r\geq 1$, and we define $D_{r}(G)^{0}:=1$. \ For $r=2$, 
$t=1$, \textit{(3.6) }occurs as the case:

\bigskip

\begin{description}
\item[\textit{(3.7)}] $D_{2}(G)=\underset{d_{\varrho }\text{ }\epsilon \text{
}cd(G)}{\sum }d_{\varrho }^{2}=\left\vert G\right\vert =\chi _{%
\mathbb{C}
G}\left( 1\right) $.
\end{description}

\bigskip

Every nontrivial group $G$ has at least two distinct irreducible characters,
each of degree $\geq 1$, and there is always at least one $d_{\varrho }$ $%
\epsilon $ $cd(G)$ such that $d_{\varrho }=1$, e.g. if $\varrho =\iota _{1}$%
. \ Therefore, $D_{1}(G)^{r}=\left( \underset{d_{\varrho }\text{ }\epsilon 
\text{ }cd(G)}{\sum }d_{\varrho }\right) ^{r}$ contains all terms $%
d_{\varrho }^{r}$ of $D_{r}(G)$\textit{\ }besides other cross-product terms $%
\geq 1$. \ Therefore:

\bigskip

\begin{description}
\item[\textit{(3.8)}] $D_{r}(G)\leq D_{1}(G)^{r}$,\qquad \qquad \qquad
\qquad $r\geq 1$.
\end{description}

\bigskip

Similarly the sum product $D_{r}(G)D_{s}(G)=\underset{d_{\varrho }\text{ }%
\epsilon \text{ }cd(G)}{\sum }d_{\varrho }^{r}\underset{d_{\varsigma }\text{ 
}\epsilon \text{ }cd(G)}{\sum }d_{\varsigma }^{s}$ contains all terms $%
d_{\varrho }^{r}$ and $d_{\varsigma }^{s}$ of $D_{r}(G)$ and $D_{s}(G)$
respectively, besides other cross-product terms $\geq 1$, and therefore

\bigskip

\begin{description}
\item[\textit{(3.9)}] $D_{r+s}(G)\leq D_{r}(G)D_{s}(G)$,\qquad \qquad \qquad
\qquad $r,s\geq 1$.
\end{description}

\bigskip

From \textit{(3.9)} we can write the following.

\bigskip

\begin{description}
\item[\textit{(3.10)}] $D_{r}(G)\leq D_{2}(G)D_{r-2}(G)=\left\vert
G\right\vert D_{r-2}(G)$,\qquad \qquad \qquad \qquad $r\geq 2$.
\end{description}

\bigskip

For a fixed group $G$, $D_{r}(G)$ is a convex function of the inverse $\frac{%
1}{r}$ of the index $r$ $\geq 1$, meaning that:

\bigskip

\begin{description}
\item[\textit{(3.11)}] $D_{s}(G)^{1/s}\leq D_{r}(G)^{1/r}$,\qquad \qquad
\qquad \qquad $1\leq r\leq s$.
\end{description}

\bigskip

There is a useful monotonicity result for sums of powers of irreducible
character degrees of subgroups.

\bigskip

\begin{lemma}
\qquad For any subgroup $H\leq $ $G$, $D_{r}(H)\leq D_{r}(G)$, and $%
D_{r}(H)<D_{r}(G)$ iff $H<G$, for all real $r\geq 1$.
\end{lemma}

\begin{proof}
For any $\varrho $ $\epsilon $ $Irrep(G)$ and $\vartheta $ $\epsilon $ $%
Irrep(H)$, by Frobenius reciprocity (\textit{Theorem 3.6}), $%
l_{Ind_{H}^{G}\vartheta ,\varrho }=l_{Res_{H}^{G}\varrho ,\vartheta }$, and,
further, by \textit{Proposition 3.5},\textit{\ }$d_{\vartheta }\leq
d_{\varrho }$ if $l_{Ind_{H}^{G}\vartheta ,\varrho }=l_{Res_{H}^{G}\varrho
,\vartheta }>0$. \ Fixing a $\varrho $ $\epsilon $ $Irrep(G)$, we have:%
\begin{equation*}
\underset{\vartheta \text{ }\epsilon \text{ }Irrep(H),\text{ }%
l_{Ind_{H}^{G}\vartheta ,\varrho }>0}{\sum }d_{\vartheta }=\underset{%
\vartheta \text{ }\epsilon \text{ }Irrep(H),\text{ }l_{Res_{H}^{G}\varrho
,\vartheta }>0}{\sum }d_{\vartheta }\leq \underset{\vartheta \text{ }%
\epsilon \text{ }Irrep(H)}{\sum }l_{Res_{H}^{G}\varrho ,\vartheta
}d_{\vartheta }=d_{Res_{H}^{G}\varrho }=d_{\varrho }\text{.}
\end{equation*}%
Taking $r^{th}$ powers in the above estimate, for $r\geq 1$, by the
reasoning in \textit{(3.11), }we\textit{\ }obtain:%
\begin{equation*}
\underset{\vartheta \text{ }\epsilon \text{ }Irrep(H),\text{ }%
l_{Res_{H}^{G}\varrho ,\vartheta }>0}{\sum }d_{\vartheta }^{r}\leq \left( 
\underset{\vartheta \text{ }\epsilon \text{ }Irrep(H),\text{ }%
l_{Res_{H}^{G}\varrho ,\vartheta }>0}{\sum }d_{\vartheta }\right) ^{r}\leq
\left( \underset{\vartheta \text{ }\epsilon \text{ }Irrep(H)}{\sum }%
l_{Res_{H}^{G}\varrho ,\vartheta }d_{\vartheta }\right) ^{r}=d_{\varrho }^{r}%
\text{.}
\end{equation*}%
If, in the above estimate, we sum over all $\varrho $ $\epsilon $ $Irrep(G)$
and omit the intermediate sums, we obtain:%
\begin{equation*}
\underset{\varrho \text{ }\epsilon \text{ }Irrep(G)}{\sum }\text{ }\underset{%
\vartheta \text{ }\epsilon \text{ }Irrep(H),\text{ }l_{Res_{H}^{G}\varrho
,\vartheta }>0}{\sum }d_{\vartheta }^{r}\leq \underset{\varrho \text{ }%
\epsilon \text{ }Irrep(G)}{\sum }d_{\varrho }^{r}=D_{r}(G)\text{.}
\end{equation*}%
Since $\underset{\vartheta \text{ }\epsilon \text{ }Irrep(H)}{\sum }%
d_{\vartheta }^{r}\leq $ $\underset{\varrho \text{ }\epsilon \text{ }Irrep(G)%
}{\sum }$ $\underset{\vartheta \text{ }\epsilon \text{ }Irrep(H),\text{ }%
l_{Res_{H}^{G}\varrho ,\vartheta }>0}{\sum }d_{\vartheta }^{r}$, we obtain
the estimate:%
\begin{equation*}
D_{r}(H)=\underset{\vartheta \text{ }\epsilon \text{ }Irrep(H)}{\sum }%
d_{\vartheta }^{r}\leq \underset{\varrho \text{ }\epsilon \text{ }Irrep(G)}{%
\sum }\text{ }\underset{\vartheta \text{ }\epsilon \text{ }Irrep(H),\text{ }%
l_{Res_{H}^{G}\varrho ,\vartheta }>0}{\sum }d_{\vartheta }^{r}\leq \underset{%
\varrho \text{ }\epsilon \text{ }Irrep(G)}{\sum }d_{\varrho }^{r}=D_{r}(G).
\end{equation*}%
This proves the first part of our claim. \ For the second part, we note
that\ $H=G$ implies that $D_{r}(H)=D_{r}(G)$. \ If\ $H<G$, i.e. is a proper
subgroup of $G$, then $H\subset G$ and $\left\vert H\right\vert <\left\vert
G\right\vert $, and $D_{1}\left( H\right) <D_{1}\left( G\right) $ (proof
left to the reader), and using the first part, it can easily be seen that $%
D_{r}(H)<D_{r}(G)$ for all $r>1$.
\end{proof}

\subsection{\textit{Estimates for Maximal Irreducible Character Degrees}}

\bigskip

Here, we derive lower and upper estimates for the maximal irreducible
character degree $d^{^{\prime }}$ of a finite group $G$ with character
degree set $cd(G)$. \ First, we note that \textit{(3.10)} may be sharpened
further using $d^{^{\prime }}$. \ By definition, $d\leq d^{^{\prime }}$ for
all $d$ $\epsilon $ $cd(G)$, so that for $r\geq 2$, $D_{r}(G)=\underset{%
d_{\varrho }\text{ }\epsilon \text{ }cd(G)}{\sum }d_{\varrho }^{r}=\underset{%
d_{\varrho }\text{ }\epsilon \text{ }cd(G)}{\sum }d_{\varrho
}^{r-2}d_{\varrho }^{2}\leq d^{^{\prime }r-2}\cdot \underset{d_{\varrho }%
\text{ }\epsilon \text{ }cd(G)}{\sum }d_{\varrho }^{2}=d^{^{\prime
}r-2}D_{2}(G)=d^{^{\prime }r-2}\left\vert G\right\vert $. \ We record this
for later use:

\bigskip

\begin{description}
\item[\textit{(3.12)}] $D_{r}(G)\leq d^{^{\prime }r-2}D_{2}(G)=d^{^{\prime
}r-2}\left\vert G\right\vert $,\qquad \qquad \qquad \qquad $r\geq 2$.
\end{description}

\bigskip

For convenience, we denote $c(G)$ by $c$. \ For each $d_{\varrho }$ $%
\epsilon $ $cd(G)$ we know that $d_{\varrho }$ $|$ $\left\vert G\right\vert $
and $1\leq d_{\varrho }^{2}\leq \left\vert G\right\vert $, and therefore,
that there is an integer $0\leq e_{\varrho }=\left\vert G\right\vert
-d_{\varrho }^{2}$ such that $d_{\varrho }^{2}+e_{\varrho }=\left\vert
G\right\vert $, i.e. that $d_{\varrho }=\left( \left\vert G\right\vert
-e_{\varrho }\right) ^{1/2}$. \ We denote by $\iota _{1}$ the trivial
irreducible representation of $G$ of dimension $d_{\iota _{1}}=1$, so that $%
e_{\iota _{1}}=\left\vert G\right\vert -d_{\iota _{1}}^{2}=\left\vert
G\right\vert -1$. \ If $e_{\varrho }=\left\vert G\right\vert -d_{\varrho
}^{2}=0$ for any $d_{\varrho }$ $\epsilon $ $cd(G)$ then $\left\vert
G\right\vert =d_{\varrho }^{2}$ and $G$ has only the one irreducible
representation $\varrho =\iota _{1}$, which is true iff $G$ is trivial. \
Thus $G$ is nontrivial iff all the $e_{\varrho }\geq 1$. \ If we let $%
e^{^{\prime }}=\left\vert G\right\vert -d^{^{\prime }2}$ then $1\leq
d^{^{\prime }}=\left( \left\vert G\right\vert -e^{^{\prime }}\right)
^{1/2}\leq \left( \left\vert G\right\vert -1\right) ^{1/2}$. \ The case $%
d^{^{\prime }}=1$ means that $d_{\varrho }=1$ for all $d_{\varrho }$ $%
\epsilon $ $cd(G)$, which is true iff $G$ is Abelian. \ We have $d^{^{\prime
}}\geq 2$ whenever $G$ is non-Abelian. \ The case $e^{^{\prime
}}=1\Longleftrightarrow $ $d^{^{\prime }}=$ $\left( \left\vert G\right\vert
-1\right) ^{1/2}\geq 1$, is exceptional: this forces $\left\vert
G\right\vert =$ $\underset{d_{\varrho }\text{ }\epsilon \text{ }cd(G)\text{ }%
}{\sum }d_{\varrho }^{2}$ to be the sum $\left\vert G\right\vert
=d^{^{^{\prime }}2}+d_{\iota _{1}}^{2}=d^{^{\prime }2}+1$, where $%
d^{^{\prime }}=Dim$ $\varrho ^{^{\prime }}$ for some $\varrho ^{^{\prime }}$ 
$\epsilon $ $Irrep(G)$ of maximal dimension. \ However, since $d^{^{\prime
}} $ $|$ $\left\vert G\right\vert $, it must be that $d^{^{\prime }}$ $|$ $1$%
, which means that $d^{^{^{\prime }}}=1$, $\left\vert G\right\vert =2$, $G$
must be a cyclic group of order $2$. \ Thus, $d^{^{\prime }}=\left(
\left\vert G\right\vert -1\right) ^{1/2}\geq 1$ can occur only for $%
d^{^{\prime }}=1$ with $G$ cyclic of order $2$. \ Thus, the case $2\leq
d^{^{^{\prime }}}<\left( \left\vert G\right\vert -1\right) ^{1/2}$ holds iff 
$G$ is a non-Abelian group, which shows that $e^{^{\prime }}=2\leq
d^{^{^{\prime }}}<\left( \left\vert G\right\vert -1\right) ^{1/2}$ implies
that $\left\vert G\right\vert \geq 6$. \ The case $1=d^{^{\prime }}<\left(
\left\vert G\right\vert -1\right) ^{1/2}$ is true iff $G$ is an Abelian
group of order $>2$.

\bigskip

Now we turn to the lower estimate for $d^{^{\prime }}$, for which $1$ is
always a trivial value. \ A tighter lower bound may easily be obtained as
follows. \ By definition, $d_{\varrho }\leq d^{^{\prime }}$ for all $%
d_{\varrho }$ $\epsilon $ $cd(G)$. \ Then, since $\left\vert G\right\vert
=D_{2}(G)=\underset{d_{\varrho }\text{ }\epsilon \text{ }cd(G)\text{ }}{\sum 
}d_{\varrho }^{2}\leq \underset{d_{\varrho }\text{ }\epsilon \text{ }cd(G)%
\text{ }}{\sum }d^{^{\prime }2}=d^{^{\prime }2}\cdot \underset{d_{\varrho }%
\text{ }\epsilon \text{ }cd(G)\text{ }}{\sum }1=cd^{^{\prime }2}$, we have
that $d^{^{\prime }}\geq \left( \frac{\left\vert G\right\vert }{c}\right)
^{1/2}$, where $1\leq c\leq \left\vert G\right\vert $. \ We consider
necessary and sufficient conditions for $G$ such that $d^{^{\prime }}$
reaches the lower bound $\left( \frac{\left\vert G\right\vert }{c}\right)
^{1/2}$ exactly. \ The case $d^{^{\prime }}=\left( \frac{\left\vert
G\right\vert }{c}\right) ^{1/2}$ is equivalent to the case $\left\vert
G\right\vert =\underset{d_{\varrho }\text{ }\epsilon \text{ }cd(G)\text{ }}{%
\sum }d_{\varrho }^{2}=cd^{^{\prime }2}$, and since $1\leq c\leq \left\vert
G\right\vert $, this can only occur iff all $d_{\varrho }=d^{^{\prime }}=1$,
which is true iff $G$ is Abelian. \ Thus, when $G$ is Abelian, $\left( \frac{%
\left\vert G\right\vert }{c}\right) ^{1/2}=1=d^{^{\prime }}\leq \left(
\left\vert G\right\vert -1\right) ^{1/2}$, with an additional equality on
the right iff $G$ is cyclic of order $2$. \ This shows that $\left( \frac{%
\left\vert G\right\vert }{c}\right) ^{1/2}=1=d^{^{\prime }}<\left(
\left\vert G\right\vert -1\right) ^{1/2}$\ iff $G$ is Abelian of order $>2$.
\ The case $d^{^{\prime }}=\left( \frac{\left\vert G\right\vert }{c}\right)
^{1/2}\geq 2$, with $G$ Abelian, is impossible, since $G$ always has the
trivial irreducible character $\chi _{\iota _{1}}$ of degree $d_{\iota
_{1}}= $ $1$, so that $\left\vert G\right\vert =d_{\iota _{1}}^{2}+\underset{%
d_{\varrho }\text{ }\epsilon \text{ }cd(G)\backslash d_{\iota _{1}}}{\sum }%
d_{\varrho }^{2}\neq cd^{^{\prime }2}$. \ Thus, $\left( \frac{\left\vert
G\right\vert }{c}\right) ^{1/2}<d^{^{\prime }}$ iff $G$ is non-Abelian,
which is true iff $d^{^{\prime }}\geq 2$, in which case $1<c<\left\vert
G\right\vert $ and $1<\left( \frac{\left\vert G\right\vert }{c}\right)
^{1/2} $ also. \ Thus, $1<\left( \frac{\left\vert G\right\vert }{c}\right)
^{1/2}<d^{^{\prime }}<\left( \left\vert G\right\vert -1\right) ^{1/2}$ holds
iff $G$ is non-Abelian, and implies that $\left\vert G\right\vert \geq 6$. \
Thus, we have proved the following result.

\bigskip

\begin{theorem}
For a nontrivial finite group $G$ with a maximal irreducible character
degree $d^{^{\prime }}(G)$ and class number $c(G)$ it holds that

\begin{equation*}
\begin{tabular}{|l|l|}
\hline
$(1)$ $\ \ 1\leq \left( \frac{\left\vert G\right\vert }{c(G)}\right)
^{1/2}\leq d^{^{\prime }}(G)\leq \left( \left\vert G\right\vert -1\right)
^{1/2}$ & general case \\ \hline
$(2)$ $\ \ 1=\left( \frac{\left\vert G\right\vert }{c(G)}\right)
^{1/2}=d^{^{\prime }}(G)\leq \left( \left\vert G\right\vert -1\right) ^{1/2}$
& iff $G$ is Abelian \\ \hline
$(3)$ $\ \ 1=\left( \frac{\left\vert G\right\vert }{c(G)}\right)
^{1/2}=d^{^{\prime }}(G)=\left( \left\vert G\right\vert -1\right) ^{1/2}$ & 
iff $G$ is cyclic of order $\left\vert G\right\vert =2$ \\ \hline
$(4)$ $\ \ 1=\left( \frac{\left\vert G\right\vert }{c(G)}\right)
^{1/2}=d^{^{\prime }}(G)<\left( \left\vert G\right\vert -1\right) ^{1/2}$ & 
iff $G$ is Abelian of order $\left\vert G\right\vert >2$ \\ \hline
$(5)$ $\ \ 1<\left( \frac{\left\vert G\right\vert }{c(G)}\right)
^{1/2}<d^{^{\prime }}(G)<\left( \left\vert G\right\vert -1\right) ^{1/2}$ & 
iff $G$ is non-Abelian ($\Longrightarrow \left\vert G\right\vert \geq 6$) \\ 
\hline
\end{tabular}%
\end{equation*}%
(Note: $(2)-(4)$ are trivial.)
\end{theorem}

\bigskip

From the upper bound for $d^{^{\prime }}(G)$ and the estimate \textit{(3.12)}%
, we obtain the estimate:

\bigskip

\begin{description}
\item[\textit{(3.13)}] $D_{r}(G)\leq \left\vert G\right\vert \left(
\left\vert G\right\vert -1\right) ^{\frac{r-2}{2}}$,\qquad \qquad \qquad
\qquad for all $r\geq 2$.
\end{description}

\bigskip

\section{Regular Group Algebras}

\bigskip

\subsection{\textit{Canonical Decomposition}}

\bigskip

Our main reference here is [BCS1997]. \ The $\varrho $ $\epsilon $ $Irrep(G)$
linearly extend to injective ($%
\mathbb{C}
$-algebra) homomorphisms $\tciFourier _{\varrho }:%
\mathbb{C}
G\longrightarrow 
\mathbb{C}
^{d_{\varrho }\times d_{\varrho }}$ of dimensions $Dim$ $\tciFourier
_{\varrho }=$ $d_{\varrho }$ defined by:

\bigskip

\begin{description}
\item[\textit{(3.14)}] $f\equiv \underset{g\text{ }\epsilon \text{ }G}{\sum }%
f(g)g\longmapsto \tciFourier _{\varrho }(f)=\underset{g\text{ }\epsilon 
\text{ }G}{\sum }f(g)\varrho (g)\equiv \widehat{f}(\varrho )$, \qquad \qquad 
$f$ $\epsilon $ $%
\mathbb{C}
G$, $\varrho $ $\epsilon $ $Irrep(G).$
\end{description}

\bigskip

By the irreducibility of the $\varrho $ $\epsilon $ $Irrep(G)$ the $%
\tciFourier _{\varrho }$ define distinct irreducible matrix representations
of $%
\mathbb{C}
G$, and form a complete set $Irrep(%
\mathbb{C}
G)$ of such representations. \ The direct sum homomorphism $\tciFourier =%
\underset{\varrho \text{ }\epsilon \text{ }Irrep(G)}{\oplus }\tciFourier
_{\varrho }:%
\mathbb{C}
G\longrightarrow \underset{\varrho \text{ }\epsilon \text{ }Irrep(G)}{\oplus 
}%
\mathbb{C}
^{d_{\varrho }\times d_{\varrho }}$ of dimension $\left\vert G\right\vert =%
\underset{\varrho \text{ }\epsilon \text{ }Irrep(G)}{\sum }d_{\varrho }^{2}$
defined by:

\bigskip

\begin{description}
\item[\textit{(3.15)}] $f\equiv \underset{g\text{ }\epsilon \text{ }G}{\sum }%
f(g)g\longmapsto \widehat{f}\equiv \tciFourier (f)=\underset{\varrho \text{ }%
\epsilon \text{ }Irrep(G)}{\oplus }\underset{g\text{ }\epsilon \text{ }G}{%
\sum }f(g)\varrho (g)=\underset{\varrho \text{ }\epsilon \text{ }Irrep(G)}{%
\oplus }\widehat{f}(\varrho )$, \qquad $f$ $\epsilon $ $%
\mathbb{C}
G$,
\end{description}

\bigskip

is injective and surjective, and, therefore, the isomorphism:

\bigskip

\begin{description}
\item[\textit{(3.16)}] $\tciFourier =\underset{\varrho \text{ }\epsilon 
\text{ }Irrep(G)}{\oplus }\tciFourier _{\varrho }:%
\mathbb{C}
G\cong _{%
\mathbb{C}
}\underset{\varrho \text{ }\epsilon \text{ }Irrep(G)}{\oplus }%
\mathbb{C}
^{d_{\varrho }\times d_{\varrho }}.$
\end{description}

\bigskip

This results in Wedderburn's theorem about the canonical decomposition of $%
\mathbb{C}
G$ as an isomorphic direct sum of $c(G)$ complex matrix algebras of orders
the $c(G)$ distinct irreducible character degrees $d_{\varrho }$ of $G$. \
Taking dimensions on both sides of \textit{(3.16)} this gives us another
version of \textit{(3.4}). \ The matrix algebra on the right of \textit{%
(3.16)}, whose elements are block-diagonal matrices, is called the \textit{%
target algebra} of $%
\mathbb{C}
G$. \ Iff $G$ is Abelian all its irreducible character degrees will be of
size $1$, the matrices of the target algebra of its regular algebra $%
\mathbb{C}
G$ will be diagonal of order $\left\vert G\right\vert $, and multiplication
in $%
\mathbb{C}
G$ will be pointwise and equal in complexity to that of diagonal matrix
multiplication of order $\left\vert G\right\vert $.

\bigskip

\subsection{\textit{Multiplicative Complexity and Rank}}

\bigskip

The (bilinear) multiplication maps of isomorphic $%
\mathbb{C}
$-algebras are isomorphic, in the sense of \textit{Corollary 2.3}, which is
also true for algebras, and by \textit{(3.16)}, we have:

\bigskip

\begin{description}
\item[\textit{(3.17)}] $\mathfrak{m}_{_{%
\mathbb{C}
G}}\cong _{%
\mathbb{C}
}\underset{\varrho \text{ }\epsilon \text{ }Irrep(G)}{\oplus }\left\langle
d_{\varrho },d_{\varrho },d_{\varrho }\right\rangle ,$
\end{description}

\bigskip

where these are the multiplication maps of $%
\mathbb{C}
G$ and its target algebra, respectively\textit{. \ }The rank of $%
\mathbb{C}
G$ is defined to be rank $\mathfrak{R}(\mathfrak{m}_{_{%
\mathbb{C}
G}})$ of its bilinear multiplication map $\mathfrak{m}_{_{%
\mathbb{C}
G}}:%
\mathbb{C}
G\times 
\mathbb{C}
G\longrightarrow 
\mathbb{C}
G$, which is defined to be the length of the minimal bilinear computation
needed to express the product of any two elements of $%
\mathbb{C}
G$ (see section \textit{2.1.2} in \textit{Chapter 2}), and, therefore,
cannot be less than $Dim$ $%
\mathbb{C}
G=\left\vert G\right\vert $. \ But, $\left\vert G\right\vert $ is also
simultaneously the rank $\mathfrak{R}_{^{\left\vert G\right\vert }}$ and the
dimension of the vector space $%
\mathbb{C}
^{\left\vert G\right\vert }$ which forms an algebra under pointwise vector
multiplication. \ Thus, we have the relation:

\begin{description}
\item[\textit{(3.18)}] $\mathfrak{R}_{^{\left\vert G\right\vert
}}=\left\vert G\right\vert \leq \mathfrak{R}(\mathfrak{m}_{_{%
\mathbb{C}
G}})=\mathfrak{R}\left( \underset{\varrho \text{ }\epsilon \text{ }Irrep(G)}{%
\oplus }\left\langle d_{\varrho },d_{\varrho },d_{\varrho }\right\rangle
\right) $ $\leq $ $\underset{\varrho \text{ }\epsilon \text{ }Irrep(G)}{\sum 
}\mathfrak{R}\left( \left\langle d_{\varrho },d_{\varrho },d_{\varrho
}\right\rangle \right) .$
\end{description}

\bigskip

Equality above holds as $\left\vert G\right\vert =\mathfrak{R}(\mathfrak{m}%
_{_{%
\mathbb{C}
G}})=\underset{\varrho \text{ }\epsilon \text{ }Irrep(G)}{\sum }\mathfrak{R}%
\left( \left\langle 1,1,1\right\rangle \right) =\underset{\varrho \text{ }%
\epsilon \text{ }Irrep(G)}{\sum }\mathfrak{R}\left( \left\langle d_{\varrho
},d_{\varrho },d_{\varrho }\right\rangle \right) $ iff $G$ is Abelian.

\bigskip

\section{\protect\Large Generalized Group Discrete Fourier Transforms}

\bigskip

\subsection{\textit{Generalized Group Discrete Fourier Transforms}}

\bigskip

Any $%
\mathbb{C}
$-algebra isomorphism $\tciFourier $\ of the form \textit{(3.16)} is called
a \textit{generalized discrete Fourier transform} ($DFT$) on $%
\mathbb{C}
G$, equivalently, a \textit{generalized group discrete Fourier transform} on 
$G$, and defines a $\left\vert G\right\vert $-dimensional matrix
representation of $%
\mathbb{C}
G$ on its target algebra also of dimension $\left\vert G\right\vert $,
where\ $%
\mathbb{C}
G$ is called the \textit{time domain} of $\tciFourier $ and the target
algebra $\underset{\varrho \text{ }\epsilon \text{ }Irrep(G)}{\oplus }%
\mathbb{C}
^{d_{\varrho }\times d_{\varrho }}$ is called its \textit{frequency space}
or \textit{Fourier domain }[BCS1997, p. 327]. \ $\tciFourier $\ is of
dimension $Dim$ $\tciFourier =$ $\left\vert G\right\vert $, and has an
invertible $\left\vert G\right\vert \times \left\vert G\right\vert $
block-diagonal matrix of full rank $\left\vert G\right\vert $ denoted by $%
[\tciFourier ]$, called a $DFT$ matrix for $%
\mathbb{C}
G$, w.r.t. any fixed bases of $%
\mathbb{C}
G$ - e.g. its regular basis $G$ - and of its target algebra. \ Every choice
of such distinct basis pairs yields a distinct $DFT$ and a $DFT$ matrix for $%
\mathbb{C}
G$. \ Any $DFT$ $\tciFourier $ on $%
\mathbb{C}
G$ has a unique direct sum decomposition $\underset{\varrho \text{ }\epsilon 
\text{ }Irrep(G)}{\oplus }\tciFourier _{\varrho }$, where the $\tciFourier
_{\varrho }$ are the irreducible components of $\tciFourier $ in \textit{%
(3.16)}, with each $\tciFourier _{\varrho }$ faithfully mapping $%
\mathbb{C}
G$ onto $%
\mathbb{C}
^{d_{\varrho }\times d_{\varrho }}$. \ Given a $DFT$ $\tciFourier $ on $%
\mathbb{C}
G$, for an $f$ $\epsilon $ $%
\mathbb{C}
G$, $\tciFourier (f)\equiv \widehat{f}$ is called the Fourier transform of $%
f $, for a fixed $\varrho $ $\epsilon $ $Irrep(G)$, $\widehat{f}(\varrho )$
defined by the formula \textit{(3.14)} is called the Fourier transform of $f$
at $\varrho $, and the elements of $\left\{ \widehat{f}(\varrho )\right\}
_{\varrho \epsilon Irrep(G)}$ are called the Fourier coefficients of $f$
which define $\widehat{f}$. \ There is an inversion formula for recovering $%
f $ from its transform [SER1977].

\bigskip

\begin{description}
\item[\textit{(3.19)}] $f(g)=\left\vert G\right\vert ^{-1}\underset{\varrho
\epsilon Irrep(G)}{\sum }d_{\varrho }\cdot Tr\left[ \widehat{f}(\varrho
)\varrho (g^{-1})\right] $,\qquad \qquad \qquad \qquad $g$ $\epsilon $ $G$.
\end{description}

\bigskip

This formula defines the inverse $DFT$ of an $f$ $\epsilon $ $%
\mathbb{C}
G$. \ If $G$ is Abelian all the $d_{\varrho }=1$, so that the $1$%
-dimensional $\varrho $ $\epsilon $ $Irrep(G)$ may be replaced, in \textit{%
(3.14)} and \textit{(3.21)}, by their irreducible characters $\chi _{\varrho
}$, all of degree $1$. \ This yields the familiar Fourier transform and
inverse transform formulas for finite Abelian groups [SER1977].

\bigskip

\begin{description}
\item[\textit{(3.20)}] $\widehat{f}(\chi _{\varrho })=$ $\underset{g\text{ }%
\epsilon \text{ }G}{\sum }f(g)\chi _{\varrho }(g)$,\qquad \qquad \qquad $%
\qquad \qquad \qquad \varrho $ $\epsilon $ $Irrep(G)$,
\end{description}

\bigskip

\begin{description}
\item[\textit{(3.21)}] $f(g)=\left\vert G\right\vert ^{-1}\underset{\varrho
\epsilon Irrep(G)}{\sum }\widehat{f}(\chi _{\varrho })\overline{\chi
_{\varrho }(g)}$,\qquad \qquad \qquad \qquad $g$ $\epsilon $ $G$.
\end{description}

\bigskip

A $DFT$ $\tciFourier $ on $%
\mathbb{C}
G$ has an invertible matrix $\left[ \tciFourier \right] $ of order $%
\left\vert G\right\vert $, whose columns correspond to the group elements $g$
$\epsilon $ $G$, and whose rows correspond to the distinct irreducible
representations $\varrho $ $\epsilon $ $Irrep(G)$. \ For any $g$ $\epsilon $ 
$G$ and any $\varrho $ $\epsilon $ $Irrep(G)$, $\varrho (g)$ is an
invertible matrix of order $d_{\varrho }$ indexed by $1\leq k_{\varrho
},l_{\varrho }\leq d_{\varrho }$, whose $d_{\varrho }^{2}$ number of entries 
$\varrho (g)_{k_{\varrho },l_{\varrho }}$ occur in $\left[ \tciFourier %
\right] $ by the following formula [SER1977].

\begin{description}
\item[\textit{(3.22)}] $\left[ \tciFourier \right] _{\varrho ,k_{\varrho
},l_{\varrho };g}=\varrho (g)_{k_{\varrho },l_{\varrho }}$.
\end{description}

\bigskip

Every choice of basis pairs in $%
\mathbb{C}
G$ and its target algebra yields a distinct $DFT$ $\tciFourier $\ and a $DFT$
matrix $\left[ \tciFourier \right] $ for $%
\mathbb{C}
G$, and the class $\{\left[ \tciFourier \right] \}$ of all $DFT$ matrices $%
[\tciFourier ]$ for $%
\mathbb{C}
G$ is determined uniquely by the $\varrho $ $\epsilon $ $Irrep(G)$.

\bigskip

\subsection{\textit{Discrete Fourier Transform on }$Sym_{3}$}

\bigskip

Here we present the example of [BCS1997, pp. 329-330]. \ Consider $Sym_{3}$,
the symmetric group of all permutations of $3$ symbols. \ The $6$
permutation elements of $Sym_{3}$ are the identity permutation $(1)$, the
three transpositions $(12)$, $(13)$, and $(23)$, and the two cycles $(123)$
and $(132)$, and these three subsets of elements form the three distinct
conjugacy classes $C_{e}$, $C_{t}$, and $C_{c}$, respectively, of $D_{3}$,
where $e$ denotes the no-change permutation, $t$ denotes a transposition,
and $c$ a $3$-cycle. \ $Sym_{3}$ has the trivial $1$-dimensional irreducible
representation $\iota $ corresponding to $e$, which is defined by $%
g\longmapsto (1)$, the nontrivial $1$-dimensional irreducible representation 
$\sigma $ corresponding to $t$, which is defined by $g\longmapsto sgn(g)$,
and the $2$-dimensional irreducible representation $\Delta $ which is
defined explicitly by the mappings: $(1)\longmapsto \left[ 
\begin{array}{cc}
1 & 0 \\ 
0 & 1%
\end{array}%
\right] $, $(12)\longmapsto \left[ 
\begin{array}{cc}
0 & 1 \\ 
1 & 0%
\end{array}%
\right] $, $(13)\longmapsto \left[ 
\begin{array}{cc}
-1 & 0 \\ 
-1 & 1%
\end{array}%
\right] $, $(23)\longmapsto \left[ 
\begin{array}{cc}
1 & -1 \\ 
0 & -1%
\end{array}%
\right] $, $(123)\longmapsto \left[ 
\begin{array}{cc}
0 & -1 \\ 
1 & -1%
\end{array}%
\right] $, $(132)\longmapsto \left[ 
\begin{array}{cc}
-1 & 1 \\ 
-1 & 0%
\end{array}%
\right] $. \ The distinct irreducible character degrees of $Sym_{3}$ are $%
d_{\iota }=1$, $d_{\sigma }=1$, and $d_{\Delta }=2$, such that $d_{\iota
}^{2}+d_{\sigma }^{2}+d_{\Delta }^{2}=6$. \ By \textit{(3.16)} the regular
group algebra $%
\mathbb{C}
Sym_{3}$ of $Sym_{3}$ then has the canonical decomposition $%
\mathbb{C}
Sym_{3}\cong _{%
\mathbb{C}
}%
\mathbb{C}
^{d_{\iota }\times d_{\iota }}\oplus 
\mathbb{C}
^{d_{\sigma }\times d_{\sigma }}\oplus 
\mathbb{C}
^{d_{\Delta }\times d_{\Delta }}=%
\mathbb{C}
^{1\times 1}\oplus 
\mathbb{C}
^{1\times 1}\oplus 
\mathbb{C}
^{2\times 2}$. \ The $DFT$ matrix $\left[ \tciFourier \right] $ for $%
\mathbb{C}
Sym_{3}$, w.r.t. canonical bases in the components of its target algebra, by
formula \textit{(3.24)}, will be the $6\times 6$ matrix:

\bigskip

$\qquad \ \ \ \ 
\begin{array}{cccccc}
(1) & \text{ \ \ }(12)\text{ \ } & \text{ }(13)\text{ \ } & \text{ \ }(23)%
\text{ \ } & \text{ }(123)\text{ \ } & \text{ \ }(132)%
\end{array}%
$

$%
\begin{array}{c}
\iota _{1,1} \\ 
\sigma _{1,1} \\ 
\Delta _{1,1} \\ 
\Delta _{1,2} \\ 
\Delta _{2,1} \\ 
\Delta _{2,2}%
\end{array}%
\left[ 
\begin{array}{cccccc}
1\text{ \ \ } & \text{ \ }1\text{ \ \ \ \ } & \text{ \ }1\text{ \ \ \ \ } & 
\text{ \ }1\text{ \ \ \ \ } & \text{ \ }1\text{ \ \ \ \ \ \ } & \text{ \ }1%
\text{ \ \ } \\ 
1\text{ \ \ } & \text{ \ }1\text{ \ \ \ \ } & \text{ \ }1\text{ \ \ \ \ } & 
-1\text{ \ \ \ \ } & -1\text{ \ \ \ \ \ \ } & -1\text{ \ \ } \\ 
1\text{ \ \ } & \text{ \ }0\text{ \ \ \ \ } & -1\text{ \ \ \ \ } & \text{ \ }%
0\text{ \ \ \ \ } & \text{ \ }1\text{ \ \ \ \ \ \ } & -1\text{ \ \ } \\ 
0\text{ \ \ } & -1\text{ \ \ \ \ } & \text{ \ }1\text{ \ \ \ \ } & \text{ \ }%
1\text{ \ \ \ \ } & -1\text{ \ \ \ \ \ \ } & \text{ \ }0\text{ \ \ } \\ 
1\text{ \ \ } & \text{ \ }1\text{ \ \ \ \ } & -1\text{ \ \ \ \ } & \text{ \ }%
1\text{ \ \ \ \ } & \text{ \ }0\text{ \ \ \ \ \ \ } & -1\text{ \ \ } \\ 
1\text{ \ \ } & -1\text{ \ \ \ \ } & \text{ \ }0\text{ \ \ \ \ } & \text{ \ }%
0\text{ \ \ \ \ } & -1\text{ \ \ \ \ \ \ } & \text{ \ }1\text{ \ \ }%
\end{array}%
\right] $.

\bigskip

We see that formula \textit{(3.24)} prescribes that the entries of the
column corresponding to a $g$ $\epsilon $ $Sym_{3}$ are the coefficients of
the linear forms obtained from the entries of the matrices $\iota (g)$, $%
\sigma (g)$, and $\Delta (g)$.

\bigskip

\subsection{\textit{Complexity and Fast Fourier Transform (FFT) Algorithms}}

\bigskip

For a given finite group $G$ with regular group algebra $%
\mathbb{C}
G$, we can think about \textit{a} discrete Fourier transform $\tciFourier $
on $G$ as a linear transformation defined by $DFT\left( f\right) =\widehat{f}
$ for any given $f$ $\epsilon $ $%
\mathbb{C}
G$, where $\widehat{f}$ is a matrix valued function on $Irrep(G)$, e.g. $%
\widehat{f}\left( \varrho \right) $ $\epsilon $ $%
\mathbb{C}
^{d_{\varrho }\times d_{\varrho }}$ for any $\varrho $ $\epsilon $ $Irrep(G)$%
. \ We see that a $DFT$ on $G$ is described by the product of a $\left\vert
G\right\vert \times \left\vert G\right\vert $ matrix by a vector of length $%
\left\vert G\right\vert $. \ If we write $M_{\tciFourier }\left( G\right) $
as the total number of arithmetical operations $\left\{ \times ,+,-\right\} $
required to implement a given group $DFT$ $\tciFourier $ on $G$, then we see
that

\bigskip

\begin{description}
\item[\textit{(3.23)}] $\left\vert G\right\vert \leq M_{\tciFourier }\left(
G\right) \leq 2\left\vert G\right\vert ^{2}.$
\end{description}

\bigskip

An \textit{efficient }group $DFT$ algorithm, also called a \textit{fast
Fourier transform }(FFT), is one which minimizes $M_{\tciFourier }\left(
G\right) $ for \textit{any} given Fourier transform $\tciFourier $, i.e. by
reducing the upper limit $2\left\vert G\right\vert ^{2}$. \ The famous
Cooley-Tukey algorithm is an FFT to compute the $DFT$ on the cyclic group $%
\mathbb{Z}
/n%
\mathbb{Z}
$ in $M_{\tciFourier }\left( 
\mathbb{Z}
/n%
\mathbb{Z}
\right) =$ $O\left( n\log n\right) $ operations [MR2003, p. 283]. \ It is
out of the scope of this thesis to discuss the FFTs on groups further, but
we will conclude by noting that according to Maslen and Rockmore, who
present an up-to-date survey of the field, $M_{\tciFourier }\left( G\right)
\leq O\left( \left\vert G\right\vert \log \left\vert G\right\vert \right) $
holds for all Abelian groups $G$, and for arbitrary groups $G$ they
conjecture that there are constants $C_{1}$ and $C_{2}$ such that $%
M_{\tciFourier }\left( G\right) \leq C_{1}\left\vert G\right\vert \log
^{C_{2}}\left\vert G\right\vert $, [MR2003, p. 286]. \ This was perhaps
suggested by the discovery for symmetric groups $Sym_{n}$ that $%
M_{\tciFourier }\left( Sym_{n}\right) =O\left( n!\log ^{2}n!\right) $
[MR2003, p. 286].

\pagebreak \bigskip

\chapter{\protect\huge Groups and Matrix Multiplication I}

\bigskip

Here we introduce the basic group-theoretic approach of embedding matrices
into group algebras via triples of subsets of the groups having the
so-called triple product property, and studying the complexity of matrix
multiplication in terms of certain numerical parameters relating to the size
of groups, the degrees of their irreducible characters, and the sizes of the
matrix multiplications realized by them. \ The methods discussed here
pertain to the group algebra embedding of a \textit{single} arbitrary pair
of matrices, and the recovery of their product via multiplication in the
group algebra, while in the sequel to this chapter, \textit{Chapter 5}, we
describe methods for the \textit{simultaneous} embedding into a group
algebra of several pairs of matrices, and recovering their independent
products simultaneously via a single multiplication in the algebra. \ In the
concluding part \textit{4.3} of the chapter, we describe a number of ways of
proving estimates for $\omega $ using the methods outlined earlier.

\section{Realizing Matrix Multiplication via Groups}

\bigskip

\subsection{\textit{Groups and Index Triples}}

\bigskip

Henceforth, $G$ shall always denote a \textit{nontrivial},\textit{\ finite}
group. \ For any nonempty subset $S\subseteq G$ its right-quotient set,
denoted by $Q(S)$, is defined as:

\bigskip

\begin{description}
\item[\textit{(4.1)}] $Q(S):=\left\{ s^{^{\prime }}s^{-1}\text{ }|\text{ }%
s,s^{^{\prime }}\text{ }\epsilon \text{ }S\right\} $.
\end{description}

\bigskip

By the definition above, $1_{G}$ $\epsilon $ $Q(S)$ necessarily for all
right-quotient sets of subsets $S$ of $G$. \ The relation

\bigskip

\begin{description}
\item[\textit{(4.2)}] $\left\vert Q(S)\right\vert \geq \left\vert
S\right\vert $
\end{description}

\bigskip

is a simple consequence of \textit{(4.1)} since a fixed $s$ $\epsilon $ $S$
the elements $s^{^{\prime }}s^{-1}$, for arbitrary $s^{^{\prime }}$ $%
\epsilon $ $S$, are distinct elements of $S$. \ Equality in \textit{(4.2)}
always holds for any subgroup $S\leq G$.

\bigskip

A triple $\left( S,T,U\right) $ of nonempty subsets $S$, $T$, $U\subseteq G$
is said to have the \textit{triple product property }(TPP) if the following
condition holds.

\bigskip

\begin{description}
\item[\textit{(4.3)}] $s^{^{\prime }}s^{-1}t^{^{\prime }}t^{-1}u^{^{\prime
}}u^{-1}=1_{G}$ $\Longleftrightarrow $ $s^{^{\prime }}s^{-1}=t^{^{\prime
}}t^{-1}$ $=u^{^{\prime }}u^{-1}=1_{G}$, \qquad $s^{^{\prime }}s^{-1}$ $%
\epsilon $ $Q(S)$, $t^{^{\prime }}t^{-1}$ $\epsilon $ $Q(T)$, $u^{^{\prime
}}u^{-1}$ $\epsilon $ $Q(U)$.
\end{description}

\bigskip

A subset triple $(S,T,U)$ of $G$ which has the triple product property is
called an \textit{index triple }of $G$, for which we can derive an
elementary property.

\bigskip

\begin{corollary}
If $(S,T,U)$ is an index triple of $G$ then the mapping $(x,y)\longmapsto
x^{-1}y$ on any distinct pair $X,$ $Y$ $\epsilon $ $\left\{ S,\text{ }T,%
\text{ }U\right\} $ is injective.
\end{corollary}

\begin{proof}
Since $S,$ $T,$ $U\subseteq G$, by assumption, satisfy the triple product
property, for arbitrary $s,$ $s^{^{\prime }}$ $\epsilon $ $S,$ $t,$ $%
t^{^{\prime }}$ $\epsilon $ $T,$ $u,$ $u^{^{\prime }}$ $\epsilon $ $U$, it
is the case that $s^{^{\prime }}s^{-1}t^{^{\prime }}t^{-1}u^{^{\prime
}}u^{-1}=1_{G}\Longleftrightarrow s^{^{\prime }}=s,$ $t^{^{\prime }}=t,$ $%
u^{^{\prime }}=u$. \ Then for arbitrary elements $s,$ $s^{^{\prime }}$ $%
\epsilon $ $S$ and $t,$ $t^{^{\prime }}$ $\epsilon $ $T$, $u,$ $u^{^{\prime
}}$ $\epsilon $ $U$, we see that 
\begin{eqnarray*}
s^{-1}t &=&s^{^{\prime }-1}t^{^{\prime }} \\
&\Longrightarrow &s^{^{\prime }}s^{-1}tt^{^{\prime }-1}uu^{^{\prime }-1}=1 \\
&\Longrightarrow &s^{^{\prime }}=s,\text{ }t^{^{\prime }}=t.
\end{eqnarray*}%
\ We can prove, in the same way, the injectivity of this mapping for all
other distinct pairs in $\left\{ S,\text{ }T,\text{ }U\right\} $.
\end{proof}

\bigskip

For any triple of subgroups $S$, $T$, $U\leq G$ the triple product property
condition may be expressed as: $stu=1_{G}$ iff $s=t=u=1_{G}$, $s$ $\epsilon $
$S$, $t$ $\epsilon $ $T$, $u$ $\epsilon $ $U$, since $Q(S)=S$, $Q(T)=T$, $%
Q(U)=U$ iff $S,T,U\leq G$.

\bigskip

The following is a trivial result.

\bigskip

\begin{lemma}
If $G$ is any group then $(1)$ $G,\left\{ 1_{G}\right\} ,\left\{
1_{G}\right\} $ is a triple of subroups having the triple product property,
and $(2)$ $G\times \left\{ 1_{G}\right\} \times \left\{ 1_{G}\right\}
,\left\{ 1_{G}\right\} \times G\times \left\{ 1_{G}\right\} ,\left\{
1_{G}\right\} \times \left\{ 1_{G}\right\} \times G$ is a triple of
subgroups having the triple product property in $G^{\times 3}$.
\end{lemma}

\bigskip

For Abelian groups there is an equivalent characterization of the triple
product property by maps, as expressed in the following lemma.

\bigskip

\begin{lemma}
If $G$ is Abelian, then $(S,T,U)$ is an index triple of $G$ iff the triple
product map $\psi :S\times T\times U\longrightarrow G$, defined by $\left(
s,t,u\right) \longmapsto stu$, is injective.
\end{lemma}

\begin{proof}
Let $G$ be Abelian. \ If subsets $S,$ $T,$ $U\subseteq G$ satisfy the triple
product property, then for any two elements $(s,t,u)$, $(s^{^{\prime
}},t^{^{\prime }},u^{^{\prime }})$ $\epsilon $ $S\times T\times U$:%
\begin{eqnarray*}
\psi (s^{^{\prime }},t^{^{\prime }},u^{^{\prime }}) &=&s^{^{\prime
}}t^{^{\prime }}u^{^{\prime }}=stu=\psi (s,t,u) \\
&\Longrightarrow &s^{^{\prime }}s^{-1}t^{^{\prime }}t^{-1}u^{^{\prime
}}u^{-1}=1_{G} \\
&\Longleftrightarrow &s^{^{\prime }}s^{-1}=1_{G},\text{ }t^{^{\prime
}}t^{-1}=1_{G},\text{ }u^{^{\prime }}u^{-1}=1_{G} \\
&\Longrightarrow &(s^{^{\prime }},t^{^{\prime }},u^{^{\prime }})=(s,t,u).
\end{eqnarray*}%
The converse is also true. \ If $\psi $ is injective on $S\times T\times U$
then for any two elements $(s,t,u)$, $(s^{^{\prime }},t^{^{\prime
}},u^{^{\prime }})$ $\epsilon $ $S\times T\times U$:%
\begin{eqnarray*}
s^{^{\prime }}s^{-1}t^{^{\prime }}t^{-1}u^{^{\prime }}u^{-1} &=&1_{G} \\
&\Longrightarrow &s^{^{\prime }}t^{^{\prime }}u^{^{\prime }}=stu \\
&\Longleftrightarrow &\left( s^{^{\prime }},t^{^{\prime }},u^{^{\prime
}}\right) =\left( s,t,u\right) \\
&\Longrightarrow &s^{^{\prime }}s^{-1}=1_{G},\text{ }t^{^{\prime
}}t^{-1}=1_{G},\text{ }u^{^{\prime }}u^{-1}=1_{G}.
\end{eqnarray*}
\end{proof}

\bigskip

The following is an elementary corollary:

\bigskip

\begin{corollary}
If $G$ is Abelian and $\left( S,T,U\right) $ is an index triple of $G$, then 
$\left\vert S\right\vert \left\vert T\right\vert \left\vert U\right\vert
\leq \left\vert G\right\vert $. \ Equivalently, if $\left( S,T,U\right) $ is
an index triple of $G$ such that $\left\vert S\right\vert \left\vert
T\right\vert \left\vert U\right\vert >|G|$ then $G$ is non-Abelian.
\end{corollary}

\begin{proof}
By \textit{Lemma 4.3, }the triple product map $\psi $ on an index triple $%
\left( S,T,U\right) $ of an Abelian group $G$ is necessarily injective and,
therefore, $S\times T\times U\cong \func{Im}\psi \subseteq G$. Taking
cardinalities on either side, we have that $\left\vert S\times T\times
U\right\vert =\left\vert S\right\vert \left\vert T\right\vert \left\vert
U\right\vert \leq \left\vert G\right\vert $. \ The equivalent statement
follows from the negation of the previous statement.
\end{proof}

\bigskip

The property of $G$ having an index triple is invariant under permutations
of the components of the triple.

\bigskip

\begin{lemma}
A subset triple $(S,T,U)$ is an index triple of $G$ iff any permuted triple $%
(\mu \left( S\right) ,\mu \left( T\right) ,\mu \left( U\right) )$ is an
index triple of $G$, for a permutation $\mu $ $\epsilon $ $Sym\{S,T,U\}$.
\end{lemma}

\begin{proof}
Assume $\left( S,T,U\right) $ has the triple product property. \ Then, for
all $s^{^{\prime }}s^{-1}$ $\epsilon $ $Q(S)$, $t^{^{\prime }}t^{-1}$ $%
\epsilon $ $Q(T)$, $u^{^{\prime }}u^{-1}$ $\epsilon $ $Q(U)$ 
\begin{eqnarray*}
t^{^{\prime }}t^{-1}s^{^{\prime }}s^{-1}u^{^{\prime }}u^{-1} &=&1_{G} \\
&\Longrightarrow &s^{^{\prime }}s^{-1}t^{^{\prime }}t^{-1}u^{^{\prime
}}u^{-1}=1_{G} \\
&\Longrightarrow &s^{^{\prime }}s^{-1}=t^{^{\prime }}t^{-1}=u^{^{\prime
}}u^{-1}=1_{G}.
\end{eqnarray*}%
This shows that $(T,S,U)$ has the triple product property if $(S,T,U)$ does,
and in the same way, we can show that $(S,U,T)$ has the triple product
property, and so do all other permuted triples $(T,U,S)$, $(U,S,T)$, $%
(U,T,S) $, because $(T,S,U)$ and $(S,U,T)$ generate the permutation group $%
Sym\{S,T,U\}$.
\end{proof}

\bigskip

\subsection{\textit{Extension Results for Index Triples}}

\bigskip

We start with a basic statement about index triples of subgroups.

\bigskip

\begin{proposition}
If $\left( S,T,U\right) $ is an index triple of a subgroup $H\leq G$ then $%
\left( S,T,U\right) $ is an index triple of $G$.
\end{proposition}

\bigskip

This follows from the definition \textit{(4.3) }of an index triple of $G$. \
Index triples of groups can also be obtained as extensions of index
triples.of normal subgroups and of their corresponding factor groups.

\bigskip

\begin{lemma}
If $\left( S_{1},S_{2},S_{3}\right) $ is an index triple of $%
H\vartriangleleft G$ and $\left( U_{1},U_{2},U_{3}\right) $ is an index
triple of $G/H$, then there exists a subset triple $\left(
T_{1},T_{2},T_{3}\right) $ of $G$, corresponding to $\left(
U_{1},U_{2},U_{3}\right) $, such that the pointwise product triple $\left(
S_{1}T_{1},S_{2}T_{2},S_{3}T_{3}\right) $ is an index triple of $G$, where
the $T_{i}\subseteq G$ are lifts to $G$ of the $U_{i}$.
\end{lemma}

\begin{proof}
Elements of $G/H$ are left-cosets $gH$ of $H$ in $G$, the subsets $U_{1},$ $%
U_{2},$ $U_{3}\subseteq G/H$ are of form 
\begin{equation*}
U_{i}=\left\{ u_{i}=v_{i}H\text{ }|\text{ }v_{i}\text{ }\epsilon \text{ }%
G\right\} ,\hspace{0.65in}1\leq i\leq 3.
\end{equation*}
\ We define \textit{lift} subsets $T_{i}\subseteq G$ of the $U_{i}$ by 
\begin{equation*}
T_{i}=\{t_{i}\text{ }\epsilon \text{ }G|\text{ }t_{i}=v_{i}h_{i},\text{ some 
}h_{i}\text{ }\epsilon \text{ }H,\text{ and }u_{i}=v_{i}H\text{ }\epsilon 
\text{ }U_{i},\text{ some }v_{i}\text{ }\epsilon \text{ }G\},\hspace{0.65in}%
1\leq i\leq 3.
\end{equation*}
\ The $T_{i}$, which are not necessarily unique for the $U_{i}$, have the
properties that $(1)$ $t_{i}H=u_{i},$ for all $t_{i}$ $\epsilon $ $T_{i}$
and $u_{i}$ $\epsilon $ $U_{i}$, $(2)$ $U_{i}\cong T_{i}$, and $(3)$ $t$ $%
\epsilon $ $T_{i}$ implies $T_{i}\cap tH=\left\{ t\right\} ,$ $1\leq i\leq 3$%
. \ Let $s_{1},s_{1}^{^{\prime }}$ $\epsilon $ $S_{1},t_{1},t_{1}^{^{\prime
}}$ $\epsilon $ $T_{1},u_{1},u_{1}^{^{\prime }}$ $\epsilon $ $U_{1}$, and $%
s_{2},s_{2}^{^{\prime }}$ $\epsilon $ $S_{2},t_{2},t_{2}^{^{\prime }}$ $%
\epsilon $ $T_{2},u_{2},u_{2}^{^{\prime }}$ $\epsilon $ $U_{2}$, and $%
s_{3},s_{3}^{^{\prime }}$ $\epsilon $ $S_{3},t_{3},t_{3}^{^{\prime }}$ $%
\epsilon $ $T_{3},u_{3},u_{3}^{^{\prime }}$ $\epsilon $ $U_{3}$ be arbitrary
elements. \ Then $s_{1}^{^{\prime }}t_{1}^{^{\prime
}}t_{1}^{-1}s_{1}^{-1}s_{2}^{^{\prime }}t_{2}^{^{\prime
}}t_{2}^{-1}s_{2}^{-1}s_{3}^{^{\prime }}t_{3}^{^{\prime
}}t_{3}^{-1}s_{3}^{-1}=1_{G}$ implies $s_{1}^{^{\prime }}t_{1}^{^{\prime
}}t_{1}^{-1}s_{1}^{-1}s_{2}^{^{\prime }}t_{2}^{^{\prime
}}t_{2}^{-1}s_{2}^{-1}s_{3}^{^{\prime }}t_{3}^{^{\prime
}}t_{3}^{-1}s_{3}^{-1}H=H$, which implies $t_{1}^{^{\prime }}H\left(
t_{1}H\right) ^{-1}t_{2}^{^{\prime }}H\left( t_{2}H\right)
^{-1}t_{3}^{^{\prime }}H\left( t_{3}H\right) ^{-1}=H$, which implies $%
u_{1}^{^{\prime }}u_{1}^{-1}u_{2}^{^{\prime }}u_{2}^{-1}u_{3}^{^{\prime
}}u_{3}^{-1}=1_{G/H}$, which implies $u_{i}^{^{\prime }}=u_{i}$ (assumption
of TPP for the $U_{i}$ in $G/H$) which implies $t_{i}^{^{\prime }}H=t_{i}H$,
which implies $t_{i}^{^{\prime }}=t_{i}$, which implies, by the first
inequality, $s_{1}^{^{\prime }}s_{1}^{-1}s_{2}^{^{\prime
}}s_{2}^{-1}s_{3}^{^{\prime }}s_{3}^{-1}=1_{H}$, which implies $%
s_{i}^{^{\prime }}=s_{i}$ (assumption of TPP for the $S_{i}$ in $%
H\vartriangleleft G$).
\end{proof}

\bigskip

In general, the pointwise product of two index triples of $G$ is not
necessarily an index triple of $G$.

\bigskip

We define the set $\mathfrak{I}(G)$ to be the collection of all index
triples of $G$. \ Formally:

\bigskip

\begin{description}
\item[\textit{(4.4)}] $\mathfrak{I}(G):=\left\{ \left( S,T,U\right) \text{ }|%
\text{ }S,\text{ }T,\text{ }U\subseteq G\text{ satisfy the triple product
property \textit{(4.3)}}\right\} .$
\end{description}

\bigskip

By \textit{Proposition 4.5 }we have:

\bigskip

\begin{description}
\item[\textit{(4.5)}] $\mathfrak{I}(H)\subseteq \mathfrak{I}(G),\qquad
\qquad \qquad \qquad H\leq G,$
\end{description}

\bigskip

and:

\bigskip

\begin{description}
\item[\textit{(4.6)}] $\left\vert \mathfrak{I}(H)\right\vert \leq \left\vert 
\mathfrak{I}(G)\right\vert ,\qquad \qquad \qquad \qquad H\leq G.$
\end{description}

\bigskip

Now we have an important extension result for the index triples of direct
product groups.

\bigskip

\begin{lemma}
If groups $G_{1}$ and $G_{2}$ have index triples $\left(
S_{1},T_{1},U_{1}\right) $ and $\left( S_{2},T_{2},U_{2}\right) $
respectively then the direct product of these triples, $\left( S_{1}\times
S_{2},T_{1}\times T_{2},U_{1}\times U_{2}\right) $, is an index triple of $%
G_{1}\times G_{2}$.
\end{lemma}

\begin{proof}
For any subsets $S_{1},T_{1},U_{1}\subseteq G_{1}$ and $S_{2},T_{2},U_{2}%
\subseteq G_{2}$, we define for $G_{1}\times G_{2}$ the subsets $S_{1}\times
S_{2}=\{(s_{1},s_{2})$ $|$ $s_{1}$ $\epsilon $ $S_{1}$, $s_{2}$ $\epsilon $ $%
S_{2}\},$ $T_{1}\times T_{2}=\{(t_{1},t_{2})$ $|$ $t_{1}$ $\epsilon $ $T_{1}$%
, $t_{2}$ $\epsilon $ $T_{2}\},$ $U_{1}\times U_{2}=\{(u_{1},u_{2})$ $|$ $%
u_{1}$ $\epsilon $ $U_{1}$, $u_{2}$ $\epsilon $ $U_{2}\}\subseteq
G_{1}\times G_{2}$. \ We call $\left( S_{1}\times S_{2},T_{1}\times
T_{2},U_{1}\times U_{2}\right) $ the direct product $\left(
S_{1},T_{1},U_{1}\right) \times \left( S_{2},T_{2},U_{2}\right) $ of the
triples $\left( S_{1},T_{1},U_{1}\right) $ and $\left(
S_{2},T_{2},U_{2}\right) $. \ It is sufficient simply to assume the triple
product property for the triples $\left( S_{1},T_{1},U_{1}\right) $ and $%
\left( S_{2},T_{2},U_{2}\right) $, and then, for arbitrary $%
s_{1},s_{1}^{^{\prime }}$ $\epsilon $ $S_{1},$ $t_{1},t_{1}^{^{\prime }}$ $%
\epsilon $ $T_{1},$ $u_{1},u_{1}^{^{\prime }}$ $\epsilon $ $U_{1},$ $%
s_{2},s_{2}^{^{\prime }}$ $\epsilon $ $S_{2},$ $t_{2},t_{2}^{^{\prime }}$ $%
\epsilon $ $T_{2},$ $u_{2},u_{2}^{^{\prime }}$ $\epsilon $ $U_{2}$, we see
that:%
\begin{eqnarray*}
&&s_{1}^{^{\prime }}s_{1}^{-1}t_{1}^{^{\prime }}t_{1}^{-1}u_{1}^{^{\prime
}}u_{1}^{-1}=1_{G_{1}},\text{ }s_{2}^{^{\prime }}s_{2}^{-1}t_{2}^{^{\prime
}}t_{2}^{-1}u_{2}^{^{\prime }}u_{2}^{-1}=1_{G_{2}} \\
&\Longleftrightarrow &s_{1}^{^{\prime }}s_{1}^{-1}=t_{1}^{^{\prime
}}t_{1}^{-1}=u_{1}^{^{\prime }}u_{1}^{-1}=1_{G_{1}},\text{ }s_{2}^{^{\prime
}}s_{2}^{-1}=t_{2}^{^{\prime }}t_{2}^{-1}=u_{2}^{^{\prime
}}u_{2}^{-1}=1_{G_{2}} \\
&\Longleftrightarrow &s_{1}^{^{\prime }}=s_{1},\text{ }t_{1}^{^{\prime
}}=t_{1},\text{ }u_{1}^{^{\prime }}=u_{1},\text{ }s_{2}^{^{\prime }}=s_{2},%
\text{ }t_{2}^{^{\prime }}=t_{2},\text{ }u_{2}^{^{\prime }}=u_{2} \\
&\Longleftrightarrow &\left( s_{1}^{^{\prime }},s_{2}^{^{\prime }}\right)
\left( s_{1},s_{2}\right) ^{-1}=\left( t_{1}^{^{\prime }},t_{2}^{^{\prime
}}\right) \left( t_{1},t_{2}\right) ^{-1}=\left( u_{1}^{^{\prime
}},u_{2}^{^{\prime }}\right) \left( u_{1},u_{2}\right) ^{-1} \\
&=&\left( 1_{G_{1}},1_{G_{2}}\right) \\
&=&\left( s_{1}^{^{\prime }}s_{1}^{-1}t_{1}^{^{\prime
}}t_{1}^{-1}u_{1}^{^{\prime }}u_{1}^{-1},s_{2}^{^{\prime
}}s_{2}^{-1}t_{2}^{^{\prime }}t_{2}^{-1}u_{2}^{^{\prime }}u_{2}^{-1}\right)
\\
&=&\left( s_{1}^{^{\prime }},s_{2}^{^{\prime }}\right) \left(
s_{1},s_{2}\right) ^{-1}\left( t_{1}^{^{\prime }},t_{2}^{^{\prime }}\right)
\left( t_{1},t_{2}\right) ^{-1}\left( u_{1}^{^{\prime }},u_{2}^{^{\prime
}}\right) \left( u_{1},u_{2}\right) ^{-1}.
\end{eqnarray*}
\end{proof}

\bigskip

If we denote by $\mathfrak{I}(G_{1}\times G_{2})$ the set of index triples
of $G_{1}\times G_{2}$, by \textit{Lemma 4.8 }we have the following.

\bigskip

\begin{description}
\item[\textit{(4.7)}] $\mathfrak{I}(G_{1}\times G_{2})\supseteq \mathfrak{I}%
(G_{1})\times \mathfrak{I}(G_{2}).$
\end{description}

\bigskip

\begin{description}
\item[\textit{(4.8)}] $\left\vert \mathfrak{I}(G_{1}\times G_{2})\right\vert
\geq \left\vert \mathfrak{I}(G_{1})\right\vert \left\vert \mathfrak{I}%
(G_{2})\right\vert .$
\end{description}

\subsection{\textit{Groups and\ Matrix Tensors}}

\bigskip

From \textit{Chapter 2}, we recall the definition of the stuctural tensor,
or simply, the matrix tensor $\left\langle n,m,p\right\rangle _{K}$ as the $%
K $-bilinear map $\left\langle n,m,p\right\rangle _{K}:K^{n\times m}\times
K^{m\times p}\longrightarrow K^{n\times p}$, which describes the
multiplication of $n$ $\times $ $m$ matrices by $m$ $\times $ $p$ matrices
over $K$, with resulting product matrices in the matrix vector space $%
K^{n\times p}$. \ We omit the subscript $K$ and say that $G$ \textit{realizes%
} a tensor $\left\langle n,m,p\right\rangle $ via an index triple $\left(
S,T,U\right) $, iff $\left\vert S\right\vert =n$, $\left\vert T\right\vert
=m $, $\left\vert U\right\vert =p$, where these are positive integers, and,
by definition, $\left( S,T,U\right) $ has the triple product property. \ In
this case, the index triple $\left( S,T,U\right) $ said to correspond to the
tensor $\left\langle n,m,p\right\rangle $. \ Loosely speaking, as will
become clear later, this means that $G$ "supports" the multiplication of $n$ 
$\times $ $m$ by $m$ $\times $ $p$ matrices indexed by subsets $S,$ $T,$ $%
U\subseteq G$, in the sense of fact (I), section 1.3, \textit{Chapter 1}. \
By assumption $G$ is nontrivial, so that by \textit{Lemma 4.2}, $G$ always
realizes the tensors $\left\langle 1,1,1\right\rangle $ and $\left\langle
2\leq \left\vert G\right\vert ,1,1\right\rangle $, and $G^{\times 3}$ always
realizes the tensor $\left\langle \left\vert G\right\vert ,\left\vert
G\right\vert ,\left\vert G\right\vert \right\rangle $. \ The following
describes an elementary property for tensors of subgroups.

\bigskip

\begin{proposition}
If $\left\langle n,m,p\right\rangle $ is a tensor realized by a subgroup $%
H\leq G$, then $\left\langle n,m,p\right\rangle $ is also realized by $G$.
\end{proposition}

\bigskip

For $G$ we define the set $\mathfrak{S}(G)$ as:

\bigskip

\begin{description}
\item[\textit{(4.9)}] $\mathfrak{S}(G):=\{\left\langle n,m,p\right\rangle $ $%
|$ $G$ realizes the tensor $\left\langle n,m,p\right\rangle \}.$
\end{description}

\bigskip

$\mathfrak{S}(G)$ is the set of all tensors $\left\langle n,m,p\right\rangle 
$ realized by $G$. \ By \textit{(4.5)-(4.6)} we have the following results.

\bigskip

\begin{description}
\item[\textit{(4.10)}] $\mathfrak{S}(H)\subseteq \mathfrak{S}(G),\qquad
\qquad \qquad \qquad H\leq G,$
\end{description}

\bigskip

and:

\bigskip

\begin{description}
\item[\textit{(4.11)}] $\left\vert \mathfrak{S}(H)\right\vert \leq
\left\vert \mathfrak{S}(G)\right\vert ,\qquad \qquad \qquad \qquad H\leq G.$
\end{description}

\bigskip

$\mathfrak{S}(G)$ is necessarily finite by the finiteness of $\mathfrak{J}%
(G) $.

\bigskip

For arbitrary tensors $\left\langle n,m,p\right\rangle $ and permutations $%
\mu $ $\epsilon $ $Sym_{3}$ we denote by $\mu \left( \left\langle
n,m,p\right\rangle \right) $ the permuted tensor $\left\langle \mu \left(
n\right) ,\mu \left( m\right) ,\mu \left( p\right) \right\rangle $. \ Then
we have an elementary result.

\bigskip

\begin{lemma}
A group $G$ realizes a tensor $\left\langle n,m,p\right\rangle $ iff it
realizes any permuted tensor $\mu \left( \left\langle n,m,p\right\rangle
\right) =\left\langle \mu \left( n\right) ,\mu \left( m\right) ,\mu \left(
p\right) \right\rangle $, where $\mu $ $\epsilon $ $Sym_{3}$.

\begin{proof}
A consequence of Lemma 4.5. (\ The reader will note the similarity between
this lemma and \textit{Proposition 2.1}. \ It will be shown that this is, in
fact, a group-theoretic version of Proposition 2.1.)
\end{proof}
\end{lemma}

\bigskip

This means that if $G$ supports multiplication of $n\times m$ by $m\times p$
matrices over $K$ it simultaneously supports multiplication of matrices,
over $K$, of all possible permutations of the dimensions $n\times m$ and $%
m\times p$, i.e. it also supports multiplication of $n\times p$ by $p\times
m $ matrices, of $m\times n$ by $n\times p$ \ of matrices, of $m\times p$ by 
$p\times n$ matrices, of $p\times n$ by $n\times m$ matrices, of $p\times m$
by $m\times p$ matrices. \ As an example, we note that by \textit{Lemma 4.1}%
, any group\textit{\ }$G$ realizes the tensors $\left\langle \left\vert
G\right\vert ,1,1\right\rangle $, $\left\langle 1,\left\vert G\right\vert
,1\right\rangle $, $\left\langle 1,1,\left\vert G\right\vert \right\rangle $%
. \ We denote $M(\left\langle n,m,p\right\rangle )$ to be the set $\left\{
\mu (\left\langle n,m,p\right\rangle )\text{ }|\text{ }\mu \text{ }\epsilon 
\text{ }Sym_{3}\right\} $ of all permutations of a given tensor $%
\left\langle n,m,p\right\rangle $. \ If we write a tensor $\left\langle
n,m,p\right\rangle $ in a "normal" form, where $n\leqslant m\leqslant p$,
then $\left\langle n,m,p\right\rangle $ can be taken to be the
representative of the set $M(\left\langle n,m,p\right\rangle )$. \ Thus, $%
\left\langle n,m,p\right\rangle $ $\epsilon $ $\mathfrak{S}(G)$ implies that 
$M(\left\langle n,m,p\right\rangle )\subseteq \mathfrak{S}(G)$, and
therefore $\mathfrak{S}(G)$ can be rewritten as:

\bigskip

\begin{description}
\item[\textit{(4.12)}] $\mathfrak{S}(G):=\underset{G\text{ realizes }%
\left\langle n,m,p\right\rangle ,\text{ }n\leq m\leq p}{\tbigcup }%
M(\left\langle n,m,p\right\rangle )$.
\end{description}

\bigskip

For arbitrary tensors $\left\langle n_{1},m_{1},p_{1}\right\rangle $ and $%
\left\langle n_{2},m_{2},p_{2}\right\rangle $ we define a pointwise
multiplication operation $\cdot $ defined by $\left\langle
n_{1},m_{1},p_{1}\right\rangle \cdot \left\langle
n_{2},m_{2},p_{2}\right\rangle :=\left\langle
n_{1}n_{1},m_{1}m_{2},p_{1}p_{2}\right\rangle $, which is associative,
commutative, and has the unit $\left\langle 1,1,1\right\rangle $. \ $%
\mathfrak{S}(G)$ need not be closed under pointwise products. \ This
operation allows us to characterize certain important extension results
about tensors.

\bigskip

\subsection{\textit{Extension Results for Matrix Tensors}}

\bigskip

\begin{lemma}
If a normal subgroup $H\vartriangleleft G$ and the corresponding factor
group $G/H$ realize the tensors $\left\langle n_{1},n_{2},n_{3}\right\rangle 
$ and $\left\langle m_{1},m_{2},m_{3}\right\rangle $, respectively, then $G$
realizes the tensor $\left\langle n_{1},n_{2},n_{3}\right\rangle \cdot
\left\langle m_{1},m_{2},m_{3}\right\rangle =\left\langle
n_{1}m_{1},n_{2}m_{2},n_{3}m_{3}\right\rangle $, where $\left\langle
n_{1},n_{2},n_{3}\right\rangle $ and $\left\langle
m_{1},m_{2},m_{3}\right\rangle $ correspond to certain representative index
triples of $H$ and $G/H$ resp., and $\left\langle
n_{1},n_{2},n_{3}\right\rangle \cdot \left\langle
m_{1},m_{2},m_{3}\right\rangle $ corresponds to the pointwise product of
these index triples.
\end{lemma}

\begin{proof}
A consequence of \textit{Lemma 4.7}.
\end{proof}

\bigskip

An analogous result applies to direct product groups.

\bigskip

\begin{lemma}
If groups $G_{1}$ and $G_{2}$ realize tensors $\left\langle
n_{1},m_{1},p_{1}\right\rangle $ and $\left\langle
n_{2},m_{2},p_{2}\right\rangle $, respectively, then their direct product $%
G_{1}\times G_{2}$ realizes the pointwise product tensor $\left\langle
n_{1},m_{1},p_{1}\right\rangle \cdot \left\langle
n_{2},m_{2},p_{2}\right\rangle =\left\langle
n_{1}n_{2},m_{1}m_{2},p_{1}p_{2}\right\rangle $, where $\left\langle
n_{1},m_{1},p_{1}\right\rangle $ and $\left\langle
n_{2},m_{2},p_{2}\right\rangle $ correspond to certain representative index
triples of $G_{1}$ and $G_{2}$ resp. and $\left\langle
n_{1},m_{1},p_{1}\right\rangle $ and $\left\langle
n_{2},m_{2},p_{2}\right\rangle $ corresponds to the direct product of these
index triples.
\end{lemma}

\begin{proof}
A consequence of \textit{Lemma 4.8}.
\end{proof}

\bigskip

Thus:

\bigskip

\begin{description}
\item[\textit{(4.13)}] $\mathfrak{S}(G_{1}\times G_{2})\supseteq \mathfrak{S}%
(G_{1})\cdot \mathfrak{S}(G_{2})$.
\end{description}

\bigskip

\subsection{\textit{Group-Algebra Embedding and Complexity of Matrix
Multiplication}}

\bigskip

The following is a fundamental result describing the embedding of matrix
multiplication into group algebras via the triple product property.

\bigskip

\begin{theorem}
If $\left\langle n,m,p\right\rangle $ $\epsilon $ $\mathfrak{S}(G)$ then $%
(1) $ $\left\langle n,m,p\right\rangle \leq _{K}\mathfrak{m}_{KG}$ and $(2)$ 
$\Re \left( \left\langle n,m,p\right\rangle \right) \leq \Re \left( 
\mathfrak{m}_{KG}\right) $.
\end{theorem}

\begin{proof}
Assume that $G$ realizes $\left\langle n,m,p\right\rangle $ through an index
triple $\left( S,T,U\right) $, i.e. subsets $S,T,U\subseteq G$ have the
triple product property, and $\left\vert S\right\vert =n$, $\left\vert
T\right\vert =m$, $\left\vert U\right\vert =p$. \ First we prove that there
exists a restriction of $\mathfrak{m}_{KG}$ to $\left\langle
n,m,p\right\rangle $, where $\mathfrak{m}_{KG}$ is the ($K$-bilinear)
multiplication map of the group $K$-algebra $KG$ of $G$. \ Let $A=\left(
A_{ij}\right) $ $\epsilon $ $K^{n\times m}$ and $B=\left( B_{j^{\prime
}k}\right) $ $\epsilon $ $K^{m\times p}$ be arbitrary $n\times m$ and $%
m\times p$ $K$-matrices respectively. \ The product of $A$ and $B$ is the $%
n\times p$ $K$-matrix $AB=C=\left( C_{ik}\right) $ with entries $C_{ik}$
determined by the formula $C_{ik}=\underset{j=j^{\prime }}{\sum }%
A_{ij}B_{jk} $, $1\leq i\leq n$, $1\leq k\leq p$. We index the entries of $A$
by $S$ and $T$, and of $B$ by $T$ and $U$ as follows:%
\begin{eqnarray*}
A_{ij} &=&A_{s,t},\qquad \qquad s=s(i),t=t(j);1\leq i\leq n,1\leq j\leq m; \\
B_{j^{\prime }k} &=&B_{t^{^{\prime }},u},\qquad \qquad t^{^{\prime
}}=t^{^{\prime }}(j^{^{\prime }}),u=u(k);1\leq j^{^{\prime }}\leq m,1\leq
k\leq p.
\end{eqnarray*}%
We define linear maps $\mathfrak{a}:K^{n\times m}\longrightarrow KG$ and $%
\mathfrak{b}:K^{m\times p}\longrightarrow KG$ for embedding the matrices $%
A=\left( A_{s,t}\right) $ and $B=\left( B_{t^{^{\prime }},u}\right) $ into
the group $K$-algebra $KG$ of $G$ as follows:%
\begin{eqnarray*}
\mathfrak{a}(A) &=&\overline{A}=\underset{s\text{ }\epsilon \text{ }S,\text{ 
}t\text{ }\epsilon \text{ }T}{\sum }A_{s,t}s^{-1}t; \\
\mathfrak{b}(B) &=&\overline{B}=\underset{t^{^{\prime }}\text{ }\epsilon 
\text{ }T,\text{ }u\text{ }\epsilon \text{ }U}{\sum }B_{t^{^{\prime
}},u}t^{^{\prime }-1}u.
\end{eqnarray*}%
The injectivity of the mappings $\left( s,t\right) \longmapsto s^{-1}t$ and $%
\left( t,u\right) \longmapsto t^{-1}u$ on $S\times T$ and $T\times U$
respectively, proved in \textit{Lemma 4.1}, means that $Ker$ $\mathfrak{a}%
=\left\{ \text{\textbf{O}}_{n\times m}\right\} $ and $Ker$ $\mathfrak{b}%
=\left\{ \text{\textbf{O}}_{m\times p}\right\} $, where \textbf{O}$_{n\times
m}$ and \textbf{O}$_{m\times p}$ are zero matrices of dimensions $n\times m$
and $m\times p$ respectively, which proves the injectivity of $\mathfrak{a}$
and $\mathfrak{b}$. \ We index the $n\times p$ matrix product $AB=C=\left(
C_{ik}\right) $ by $S$ and $U$ as follows:%
\begin{equation*}
C_{ik}=C_{s,u},\qquad \qquad s=s(i),u=u(k);1\leq i\leq n,1\leq k\leq p
\end{equation*}%
\ We define an injective linear embedding map $\mathfrak{c:}K^{n\times
p}\longrightarrow KG$ as follows:%
\begin{equation*}
\mathfrak{c}(C)=\overline{C}=\underset{s\text{ }\epsilon \text{ }S,\text{ }u%
\text{ }\epsilon \text{ }U}{\sum }C_{s,u}s^{-1}u.
\end{equation*}%
The product of $\overline{A}$ and $\overline{B}$ in $KG$ is given by:%
\begin{eqnarray*}
\overline{A}\overline{B} &=&\underset{s\text{ }\epsilon \text{ }S,\text{ }t%
\text{ }\epsilon \text{ }T}{\sum }A_{s,t}s^{-1}t\cdot \underset{t^{^{\prime
}}\text{ }\epsilon \text{ }T,\text{ }u\text{ }\epsilon \text{ }U}{\sum }%
B_{t^{^{\prime }},u}t^{^{\prime }-1}u \\
&=&\underset{s\text{ }\epsilon \text{ }S,\text{ }t\text{ }\epsilon \text{ }T%
\text{ }}{\sum }\underset{t^{^{\prime }}\text{ }\epsilon \text{ }T,\text{ }u%
\text{ }\epsilon \text{ }U}{\sum }A_{s,t}B_{t^{^{\prime
}},u}s^{-1}tt^{^{\prime }-1}u \\
&=&\underset{s\text{ }\epsilon \text{ }S,\text{ }u\text{ }\epsilon \text{ }U%
\text{ }}{\sum }\left( \underset{t,t^{^{\prime }}\text{ }\epsilon \text{ }T%
\text{ }}{\sum }A_{s,t}B_{t^{^{\prime }},u}tt^{^{\prime }-1}\right) s^{-1}u.
\end{eqnarray*}%
Clearly, for each distinct pair $s$ $\epsilon $ $S,$ $u$ $\epsilon $ $U$, $%
C_{s,u}=\underset{t=t^{^{\prime }}\text{ }\epsilon \text{ }T}{\sum }%
A_{s,t}B_{t,u}$. \ By assumption, $(S,T,U)$ is an index triple of $G$, which
means that for arbitrary elements $s^{^{\prime }}$ $\epsilon $ $S$, $%
u^{^{\prime }}$ $\epsilon $ $U$, $s^{^{\prime }-1}u^{^{\prime
}}=s^{-1}tt^{^{\prime }-1}u\Longleftrightarrow s^{^{\prime
}}s^{-1}tt^{^{\prime }-1}uu^{^{\prime }-1}=1_{G}$ is true iff $s^{^{\prime
}}=s$, $t^{^{\prime }}=t$, $u^{^{\prime }}=u$. \ But the sum of those terms
of $\overline{A}\overline{B}$ for which $t^{^{\prime }}=t$ all have the
group term $s^{^{\prime }-1}u^{^{\prime }}$ and the coefficient $\underset{%
t=t^{^{\prime }}\text{ }\epsilon \text{ }T}{\sum }A_{s^{^{\prime
}},t}B_{t,u^{^{\prime }}}$ corresponds $1$-to-$1$ with the $(i,k)^{th}$
entry $C_{ik}$ of $C$ in $K^{n\times p}$ as given above. \ Thus, 
\begin{equation*}
\mathfrak{c}(C)=\overline{C}=\overline{A}\overline{B}=\mathfrak{a}(A)%
\mathfrak{b}(B).
\end{equation*}%
Then, we define an extraction map $\mathfrak{x}:K^{n\times p}\longleftarrow
KG$ by:%
\begin{equation*}
\mathfrak{x}(\overline{A}\overline{B})=C=\left( C_{ik}\right)
\end{equation*}%
where the $C_{i^{^{\prime }}k^{^{\prime }}}$ are coefficients of the terms $%
s^{^{\prime }-1}u^{^{\prime }}$ in $\overline{A}\overline{B}$, for $1\leq
i^{^{\prime }}=i^{^{\prime }}(s^{^{\prime }})\leq n$, $1\leq k^{^{\prime
}}=k^{^{\prime }}(t^{^{\prime }})\leq p$, $s^{^{\prime }}\epsilon $ $S$, $%
t^{^{\prime }}\epsilon $ $T$. \ Clearly $\mathfrak{x=a}^{-1}\cdot \mathfrak{b%
}^{-1}$, and by $\mathfrak{x}$ we will recover the desired matrix product $%
C=AB$ from $\overline{A}\overline{B}$. \ By the injectivity of the maps $%
\mathfrak{a}$ and $\mathfrak{b}$ it follows that $\mathfrak{x=a}^{-1}\cdot 
\mathfrak{b}^{-1}=$ $\mathfrak{c}^{-1}$. \ Since, for each $\left(
A,B\right) $ $\epsilon $ $K^{n\times m}\times K^{m\times p}$, we have the
composition $\mathfrak{x}\circ \mathfrak{m}_{KG}\circ \left( \mathfrak{a}%
\times \mathfrak{b}\right) \left( A,B\right) =C$ $\epsilon $ $K^{n\times p}$%
, it follows that:%
\begin{equation*}
\mathfrak{x}\circ \mathfrak{m}_{KG}\circ \left( \mathfrak{a}\times \mathfrak{%
b}\right) =\left\langle n,m,p\right\rangle .
\end{equation*}%
i.e. $\left\langle n,m,p\right\rangle \leq _{K}\mathfrak{m}_{KG}$, which
proves $(1)$. \ $(2)$ follows from applying \textit{Proposition 2.2} to $(1)$%
.
\end{proof}

\bigskip

If $K=%
\mathbb{C}
$ then we have an elementary corollary.

\bigskip

\begin{corollary}
If $\left\langle n,m,p\right\rangle $ $\epsilon $ $\mathfrak{S}(G)$ then $%
(1) $ $\left\langle n,m,p\right\rangle \leq _{%
\mathbb{C}
}\mathfrak{m}_{%
\mathbb{C}
G}\cong _{%
\mathbb{C}
}\underset{\varrho \text{ }\epsilon \text{ }Irrep(G)}{\oplus }\left\langle
d_{\varrho },d_{\varrho },d_{\varrho }\right\rangle $, $(2)$ $\left(
nmp\right) ^{\frac{\omega }{3}}\leq \Re (\left\langle n,m,p\right\rangle
)\leq \Re (\mathfrak{m}_{%
\mathbb{C}
G})\leq $ $\underset{\varrho \text{ }\epsilon \text{ }Irrep(G)}{\sum }\Re
\left( \left\langle d_{\varrho },d_{\varrho },d_{\varrho }\right\rangle
\right) $, and $(3)$ $\left( nmp\right) ^{\frac{\omega }{3}}\leq \Re
(\left\langle n,m,p\right\rangle )\leq \left\vert G\right\vert $ if $G$ is
Abelian.
\end{corollary}

\begin{proof}
For $(1)$ we apply \textit{(3.17)} to part $(1)$ of \textit{Theorem 4.13}. \
For $(2)$\textit{\ }we apply \textit{(3.18), Proposition 2.6}, and \textit{%
Proposition 2.13} to part $(2)$ of \textit{Theorem 4.13}. \ For $(3)$ we
note that if $G$ is Abelian then $\Re (\mathfrak{m}_{%
\mathbb{C}
G})=\left\vert G\right\vert $. \ From the latter case, we can also deduce
that $\omega \leq \frac{\log \left\vert G\right\vert }{\log \left(
nmp\right) ^{1/3}}$ if $G$ is Abelian and $\left\langle n,m,p\right\rangle $ 
$\epsilon $ $\mathfrak{S}(G)$.
\end{proof}

\bigskip

\section{\protect\Large The Complexity of Matrix Multiplication Realized by
Groups}

\bigskip

\subsection{\textit{Pseudoexponents}}

\bigskip

For tensors $\left\langle n,m,p\right\rangle $ we define their \textit{size}
or \textit{order }by $z\left( \left\langle n,m,p\right\rangle \right) =nmp$,
and we write for their $t^{th}$ powers $\left( nmp\right) ^{t}$, where $t>0$
is any real number. \ If $t=\frac{1}{3}$ then $\left( nmp\right) ^{1/3}$ is
just the geometric mean of the components of $\left\langle
n,m,p\right\rangle $, i.e. the \textit{mean size} or \textit{mean order }of $%
\left\langle n,m,p\right\rangle $. \ We define by $\mathfrak{S}^{^{\prime
}}(G)$ the set of all tensors of size $>1$ realized by a nontrivial group $G$%
, i.e. $\mathfrak{S}^{^{\prime }}(G)=\mathfrak{S}(G)\backslash \left\{
\left\langle 1,1,1\right\rangle \right\} $, and for any tensor $\left\langle
n,m,p\right\rangle $ $\epsilon $ $\mathfrak{S}^{^{\prime }}(G)$, $nmp\geq 2$%
. \ Since $G$ is nontrivial, by \textit{Lemma 4.1}, $G$ always realizes the
tensor $\left\langle 2\leq \left\vert G\right\vert ,1,1\right\rangle $ and $%
\left\vert \mathfrak{S}^{^{\prime }}(G)\right\vert \geq 1$.

\bigskip

We define for $G$ a number called its \textit{pseudoexponent} $\alpha (G)$
by:

\bigskip

\begin{description}
\item[\textit{(4.14)}] $\alpha (G):=\underset{\left\langle
n,m,p\right\rangle \text{ }\epsilon \text{ }\mathfrak{S}^{\prime }(G)\text{ }%
}{\min }\left( \log _{\left( nmp\right) ^{1/3}}|G|\right) \equiv \log _{%
\underset{\left\langle n,m,p\right\rangle \text{ }\epsilon \text{ }\mathfrak{%
S}^{\prime }(G)}{\max }\left( nmp\right) ^{1/3}}|G|$.
\end{description}

\bigskip

By the definition and finiteness of $\mathfrak{S}^{^{\prime }}(G)$, $\alpha
(G)$ necessarily exists. \ From above, we see that $\alpha (G)$ is uniquely
determined by a \textit{maximal }tensor $\left\langle n^{^{\prime
}},m^{^{\prime }},p^{^{\prime }}\right\rangle \neq \left\langle
1,1,1\right\rangle $ realized by $G$, i.e. a tensor $\left\langle
n^{^{\prime }},m^{^{\prime }},p^{^{\prime }}\right\rangle $ $\epsilon $ $%
\mathfrak{S}^{^{\prime }}(G)$ such that $n^{^{\prime }}m^{^{\prime
}}p^{^{\prime }}\geq nmp$ for all other tensors $\left\langle
n,m,p\right\rangle $ $\epsilon $ $\mathfrak{S}^{^{\prime }}(G)$. \ Formally:

\bigskip

\begin{description}
\item[\textit{(4.15)}] $z^{^{\prime }}(G):=z\left( \left\langle n^{^{\prime
}},m^{^{\prime }},p^{^{\prime }}\right\rangle \right) :=\underset{%
\left\langle n,m,p\right\rangle \text{ }\epsilon \text{ }\mathfrak{S}%
^{^{\prime }}(G)}{\max }nmp$.
\end{description}

\bigskip

It follows that:

\bigskip

\begin{description}
\item[\textit{(4.16)}] $1<nmp\leq n^{^{\prime }}m^{^{\prime }}p^{^{\prime }}$%
,\qquad \qquad \qquad \qquad $\left\langle n,m,p\right\rangle $ $\epsilon $ $%
\mathfrak{S}(G)$.
\end{description}

\bigskip

The components of $\left\langle n^{^{\prime }},m^{^{\prime }},p^{^{\prime
}}\right\rangle $ are the sizes $\left\vert S^{^{\prime }}\right\vert
=n^{^{\prime }}$, $\left\vert T^{^{\prime }}\right\vert =m^{^{\prime }}$, $%
\left\vert U^{^{\prime }}\right\vert =p^{^{\prime }}$ of a maximal index
triple $\left( S^{^{\prime }},T^{^{\prime }},U^{^{\prime }}\right) $ of $G$,
which need not be unique. \ $\alpha (G)$ can be redefined as:

\bigskip

\begin{description}
\item[\textit{(4.17)}] $\alpha (G):=\log _{z^{^{\prime }}(G)^{1/3}}|G|$.
\end{description}

\bigskip

This is equivalent to:

\bigskip

\begin{description}
\item[\textit{(4.18)}] $z^{^{\prime }}(G)=|G|^{\frac{3}{\alpha (G)}}$.
\end{description}

\bigskip

We note that $z^{^{\prime }}(G)$ can be understood as the maximal size of
matrix multiplication supported by $G$, and $z^{^{\prime }}(G)^{\frac{1}{3}}$
is the geometric mean of this maximal size.

\bigskip

It follows from \textit{Proposition 4.9} that:

\bigskip

\begin{description}
\item[\textit{(4.19)}] $z^{^{\prime }}(H)\leq z^{^{\prime }}(G),\qquad
\qquad \qquad \qquad H\leq G.$
\end{description}

\bigskip

The following is an immediate consequence of \textit{Lemma 4.11}.

\bigskip

\begin{description}
\item[\textit{(4.20)}] $z^{^{\prime }}(G_{1})z^{^{\prime }}(G_{2})\leq
z^{^{\prime }}(G),\qquad \qquad \qquad \qquad G=G_{1}\times G_{2}$.
\end{description}

\bigskip

It follows from \textit{(4.14)} that:

\bigskip

\begin{description}
\item[\textit{(4.21)}] $\alpha (G)\leqslant \log _{\left( nmp\right)
^{1/3}}|G|$,\qquad \qquad \qquad \qquad $\left\langle n,m,p\right\rangle $ $%
\epsilon $ $\mathfrak{S}(G)$.
\end{description}

\bigskip

For any positive integer $n$ it follows that:

\bigskip

\begin{description}
\item[\textit{(4.22)}] $\alpha (G)\leqslant \log _{n}|G|$,\qquad \qquad
\qquad \qquad $\qquad \left\langle n,n,n\right\rangle $ $\epsilon $ $%
\mathfrak{S}(G)$.

\item \bigskip
\end{description}

The lower and upper bounds for $\alpha (G)$ are determined by the following
fundamental lemma.

\bigskip

\begin{lemma}
For any group $G$, $2<\alpha (G)\leq 3$. \ If $G$ is Abelian then $\alpha
(G)=3$.
\end{lemma}

\begin{proof}
By \textit{Lemma 4.2 }$G$ realizes the tensor $\left\langle
1,1,|G|\right\rangle $ of size $|G|$, which shows that $\alpha (G)\leq 3$ by 
\textit{(4.21)}. \ For the lower bound, we note first that, for any index
triple $(S,T,U)$ $\epsilon $ $\mathfrak{I}(G)$ with associated tensor $%
\left\langle n,m,p\right\rangle $ $\epsilon $ $\mathfrak{S}(G)$, by \textit{%
Lemma 4.1}, the mappings $(s,t)\longmapsto s^{-1}t$, $(s,u)\longmapsto
s^{-1}u$, $(t,u)\longmapsto t^{-1}u$ on $S\times T$, $S\times U$, and $%
T\times U$ respectively, are injective, which means that $nm\leq \left\vert
G\right\vert $, $np\leq \left\vert G\right\vert $, and $mp\leq \left\vert
G\right\vert $. \ We now prove that if equalities hold in these inequalities
then $p=1$, or $m=1$, or $n=1$, respectively, starting with $nm\leq
\left\vert G\right\vert $. \ Assume that $nm=\left\vert G\right\vert
\Longleftrightarrow S^{-1}T=G$. \ Then, for arbitrary elements $u,$ $%
u^{^{\prime }}$ $\epsilon $ $U$ and $s^{^{\prime }}$ $\epsilon $ $S$ and $t$ 
$\epsilon $ $T$, there exist unique $s$ $\epsilon $ $S$ and $t^{^{\prime }}$ 
$\epsilon $ $T$ such that $s^{-1}t^{^{\prime }}=s^{^{\prime }-1}uu^{^{\prime
}-1}t$, which implies $s^{^{\prime }}s^{-1}t^{^{\prime }}t^{-1}u^{^{\prime
}}u^{-1}=s^{^{\prime }}s^{^{\prime }-1}uu^{^{\prime }-1}tt^{-1}u^{^{\prime
}}u^{-1}=uu^{^{\prime }-1}u^{^{\prime }}u^{-1}=1_{G}$. \ Hence, $u^{^{\prime
}}u^{-1}=1_{G}$, and $u^{^{\prime }}=u$, i.e. $\left\vert U\right\vert =p=1$%
. \ In the same way we can prove that $np=\left\vert G\right\vert
\Longrightarrow m=1$, and $mp=\left\vert G\right\vert \Longrightarrow n=1$.
\ Hence, if $\left\langle n,m,p\right\rangle $ $\epsilon $ $\mathfrak{S}(G)$
then \textit{not} all of $n,$ $m,$ $p$ $=1$, so that \textit{not }all of $%
nm, $ $np,$ $mp=\left\vert G\right\vert $, i.e. if $G$ realizes an index
triple $(S,T,U)$ corresponding to a tensor $\left\langle n,m,p\right\rangle $
then $\left( nmp\right) ^{2}<\left\vert G\right\vert ^{3}$ and, therefore, $%
nmp<\left\vert G\right\vert ^{\frac{3}{2}}$. \ Therefore, $z^{^{\prime
}}(G)<\left\vert G\right\vert ^{\frac{3}{2}}$ maximally, and $\alpha
(G)=\log _{z^{^{\prime }}(G)^{1/3}}\left\vert G\right\vert >\log _{\left(
\left\vert G\right\vert ^{3/2}\right) ^{1/3}}\left\vert G\right\vert =2$%
.\linebreak \linebreak Finally, if $G$ is Abelian and has the maximal matrix
tensor $\left\langle n^{^{\prime }},m^{^{\prime }},p^{^{\prime
}}\right\rangle $ of size $z^{^{\prime }}(G)=n^{^{\prime }}m^{^{\prime
}}p^{^{\prime }}$ then, by the \textit{Corollary 4.4} to \textit{Lemma 4.3}, 
$z^{^{\prime }}(G)\leq \left\vert G\right\vert $. \ This implies that $\log
_{\left\vert G\right\vert ^{1/3}}\left\vert G\right\vert =3\leq \log
_{z^{^{\prime }}(G)^{1/3}}\left\vert G\right\vert =\alpha (G)$, i.e. $\alpha
(G)=3$. \ The negation of the preceding statement is that $\alpha (G)<3$
implies that $G$ is non-Abelian.
\end{proof}

\bigskip

This leads to an elementary corollary.

\bigskip

\begin{corollary}
If $\left\langle n,m,p\right\rangle $ $\epsilon $ $\mathfrak{S}(G)$, $\left(
nmp\right) ^{\frac{1}{3}}<\left\vert G\right\vert ^{\frac{1}{2}}$.
\end{corollary}

\bigskip

By \textit{Lemma 4.15}:

\bigskip

\begin{description}
\item[\textit{(4.23)}] $\left\vert G\right\vert \leq z^{^{\prime
}}(G)<\left\vert G\right\vert ^{3/2}$.
\end{description}

\bigskip

The following is another elementary result.

\bigskip

\begin{corollary}
$\alpha (G)<3$ iff $z^{^{\prime }}(G)>\left\vert G\right\vert $. \
Equivalently, $\alpha (G)<3$ iff $nmp>\left\vert G\right\vert $, for some
tensor $\left\langle n,m,p\right\rangle $ $\epsilon $ $\mathfrak{S}\left(
G\right) $.
\end{corollary}

\bigskip

Thus, the closer $z^{^{\prime }}(G)$ is to $\left\vert G\right\vert $, the
closer $\alpha (G)$ is to $3$ and $G$ is close to being the $3^{rd}$ power
of the maximal mean order of matrix multiplication that it supports. \ The
closer $z^{^{\prime }}(G)$ is to $\left\vert G\right\vert ^{3/2}$, the
closer $\alpha (G)$ is close to $2$ and $G$ is close to being the $2^{nd}$
power of maximal mean order of matrix multiplication it supports. \ For
example, if $G_{1}$ and $G_{2}$ are two finite groups such that $\alpha
(G_{1})\leq \alpha (G_{2})$, then $G_{1}$ is at least as "efficient" in
supporting matrix multiplication as $G_{2}$, or more efficient if $\alpha
(G_{1})<\alpha (G_{2})$.

\bigskip

For subgroups we have an easy result.

\bigskip

\begin{lemma}
For any nontrivial subgroup $H\leq G$, $\alpha (G)\leq \log _{z^{^{\prime
}}(H)^{1/3}}\left[ G:H\right] +\alpha (H).$
\end{lemma}

\begin{proof}
$\left\vert G\right\vert =\left[ G:H\right] \left\vert H\right\vert $, and
from \textit{(4.19)} $z^{^{\prime }}(H)\leq z^{^{\prime }}(G)$. \ Therefore, 
\begin{eqnarray*}
\alpha (G) &=&\frac{\log \left\vert G\right\vert }{\log z^{^{\prime
}}(G)^{1/3}} \\
&=&\frac{\log \left[ G:H\right] }{\log z^{^{\prime }}(G)^{1/3}}+\frac{\log
\left\vert H\right\vert }{\log z^{^{\prime }}(G)^{1/3}} \\
&\leq &\frac{\log \left[ G:H\right] }{\log z^{^{\prime }}(H)^{1/3}}+\frac{%
\log \left\vert H\right\vert }{\log z^{^{\prime }}(H)^{1/3}} \\
&=&\frac{\log \left[ G:H\right] }{\log z^{^{\prime }}(H)^{1/3}}+\alpha
\left( H\right) .
\end{eqnarray*}
\end{proof}

\bigskip

For normal subgroups we have the following basic result.

\bigskip

\begin{lemma}
For any nontrivial normal subgroup $H\vartriangleleft G$ and corresponding
factor group $G/H$, $\alpha (G)\leq \max (\alpha (H),\alpha (G/H))$.
\end{lemma}

\begin{proof}
Let $\left\langle m_{1}^{^{\prime }},m_{2}^{^{\prime }},m_{3}^{^{\prime
}}\right\rangle $ $\epsilon $ $\mathfrak{S}(H)$ and $\left\langle
p_{1}^{^{\prime }},p_{2}^{^{\prime }},p_{3}^{^{\prime }}\right\rangle $ $%
\epsilon $ $\mathfrak{S}(G/H)$ be the maximal tensors realized by $%
H\vartriangleleft G$ and its factor group $G/H$, respectively, and let $%
\left\langle n_{1}^{^{\prime }},n_{2}^{^{\prime }},n_{3}^{^{\prime
}}\right\rangle $ $\epsilon $ $\mathfrak{S}(G)$ be the maximal tensor
realized by $G$. \ Then, by \textit{(4.18)}, we have the identities:%
\begin{equation*}
z^{^{\prime }}(H)^{\alpha (H)}=\left\vert H\right\vert ^{3},\text{ }%
z^{^{\prime }}(G/H)^{\alpha (G/H)}=\left\vert G/H\right\vert ^{3},\text{ }%
z^{^{\prime }}(G)^{\alpha (G)}=\left\vert G\right\vert ^{3}.
\end{equation*}%
By \textit{Lemma 4.11} $\left\langle m_{1}^{^{\prime }}p_{1}^{^{\prime
}},m_{2}^{^{\prime }}p_{2}^{^{\prime }},m_{3}^{^{\prime }}p_{3}^{^{\prime
}}\right\rangle =\left\langle m_{1}^{^{\prime }},m_{2}^{^{\prime
}},m_{3}^{^{\prime }}\right\rangle \cdot \left\langle p_{1}^{^{\prime
}},p_{2}^{^{\prime }},p_{3}^{^{\prime }}\right\rangle $ $\epsilon $ $%
\mathfrak{S}(G)$, and therefore:%
\begin{eqnarray*}
&&z^{^{\prime }}(H)z^{^{\prime }}(G/H) \\
&=&z\left( \left\langle m_{1}^{^{\prime }},m_{2}^{^{\prime
}},m_{3}^{^{\prime }}\right\rangle \cdot \left\langle p_{1}^{^{\prime
}},p_{2}^{^{\prime }},p_{3}^{^{\prime }}\right\rangle \right) \\
&=&z\left( \left\langle m_{1}^{^{\prime }}p_{1}^{^{\prime }},m_{2}^{^{\prime
}}p_{2}^{^{\prime }},m_{3}^{^{\prime }}p_{3}^{^{\prime }}\right\rangle
\right) \\
&=&m_{1}^{^{\prime }}p_{1}^{^{\prime }}m_{2}^{^{\prime }}p_{2}^{^{\prime
}}m_{3}^{^{\prime }}p_{3}^{^{\prime }} \\
&\leq &n_{1}^{^{\prime }}n_{2}^{^{\prime }}n_{3}^{^{\prime }} \\
&=&z^{^{\prime }}(G)
\end{eqnarray*}%
by the maximality of $\left\langle n_{1}^{^{\prime }},n_{2}^{^{\prime
}},n_{3}^{^{\prime }}\right\rangle $ for $G$. \ Then also%
\begin{equation*}
\left( z^{^{\prime }}(H)z^{^{\prime }}(G/H)\right) ^{\alpha (G)}=z^{^{\prime
}}(H)^{\alpha (G)}z^{^{\prime }}(G/H)^{\alpha (G)}\leq z^{^{\prime
}}(G)^{\alpha (G)}.
\end{equation*}%
Using the identities $\left\vert H\right\vert \left\vert G/H\right\vert
=\left\vert G\right\vert $, and $\left\vert H\right\vert ^{3}\left\vert
G/H\right\vert ^{3}=\left\vert G\right\vert ^{3}$, we see that:\ 
\begin{equation*}
z^{^{\prime }}(G)^{\alpha (G)}=z^{^{\prime }}(H)^{\alpha (H)}z^{^{\prime
}}(G/H)^{\alpha (G/H)}.
\end{equation*}%
If $\max (\alpha (H),\alpha (G/H))=\alpha (H)$ and $\alpha (G)>\alpha (H)$
then:%
\begin{eqnarray*}
\left( z^{^{\prime }}(H)z^{^{\prime }}(G/H)\right) ^{\alpha (G)}
&=&z^{^{\prime }}(H)^{\alpha (G)}z^{^{\prime }}(G/H)^{\alpha (G)} \\
&>&z^{^{\prime }}(H)^{\alpha (H)}z^{^{\prime }}(G/H)^{\alpha (H)} \\
&\geq &z^{^{\prime }}(H)^{\alpha (H)}z^{^{\prime }}(G/H)^{\alpha (G/H)} \\
&=&z^{^{\prime }}(G)^{\alpha (G)}
\end{eqnarray*}%
i.e. a contradiction. \ Similarly, if $\max (\alpha (H),\alpha (G/H))=\alpha
(G/H)$ and $\alpha (G)>\alpha (G/H)$ then:%
\begin{eqnarray*}
\left( z^{^{\prime }}(H)z^{^{\prime }}(G/H)\right) ^{\alpha (G)}
&=&z^{^{\prime }}(H)^{\alpha (G)}z^{^{\prime }}(G/H)^{\alpha (G)} \\
&>&z^{^{\prime }}(H)^{\alpha (G/H)}z^{^{\prime }}(G/H)^{\alpha (G/H)} \\
&\geq &z^{^{\prime }}(H)^{\alpha (H)}z^{^{\prime }}(G/H)^{\alpha (G/H)} \\
&=&z^{^{\prime }}(G)^{\alpha (G)}
\end{eqnarray*}%
also a contradiction. \ Thus it must be that $\alpha (G)\leq \alpha (H),$ $%
\alpha (G/H)$, which means that:%
\begin{equation*}
\alpha (G)\leq \max (\alpha (H),\alpha \left( G/H\right) ).
\end{equation*}%
Equality above holds trivially if $G$ is Abelian.
\end{proof}

\bigskip

For direct product groups we have the following result.

\bigskip

\begin{lemma}
$\alpha \left( G^{\times k}\right) \leq \alpha \left( G\right) $, where $%
G^{\times k}$ is the $k$-fold direct product of $G$.
\end{lemma}

\begin{proof}
By \textit{(4.20)} $z^{^{\prime }}(G^{\times k})\geq z^{^{\prime }}(G)^{k}$.
\ Then 
\begin{eqnarray*}
\alpha \left( G^{\times k}\right) &=&\log _{z^{^{\prime }}(G^{\times
k})^{1/3}}\left\vert G^{\times k}\right\vert \\
&\leq &\log _{z^{^{\prime }}(G)^{k/3}}\left\vert G\right\vert ^{k} \\
&=&\frac{\NEG{k}}{\NEG{k}}\log _{z^{^{\prime }}(G)^{1/3}}\left\vert
G\right\vert \\
&=&\alpha \left( G\right) \text{.}
\end{eqnarray*}%
\bigskip
\end{proof}

\bigskip

This means that for $k=1,2,3,....$ we have a descending sequence of
pseudoexponent inequalities, $...\leq \alpha \left( G^{\times 3}\right) \leq
\alpha \left( G^{\times 2}\right) \leq \alpha \left( G\right) $. \ We do not
know of general conditions on the $G^{\times k}$ for making this sequence
strict, though it would require a strict ascending sequence, $z^{^{\prime
}}(G)<z^{^{\prime }}(G^{\times 2})<z^{^{\prime }}(G^{\times 3})<\ldots $,
for the corresponding maximal tensors $z^{^{\prime }}(G^{\times k})$ of the $%
G^{\times k}$.

\bigskip

For explicit estimates of the exponents of specific types of groups we refer
the reader to sections 5-7 in [CU2003], since for the derivation of
estimates of the exponent $\omega $ we have found the simultaneous triple
product property more useful. \ However, we do give some estimates for the
exponents of the symmetric groups in sections \textit{6.1.2-6.1.3} in 
\textit{Chapter 6}.

\subsection{\textit{The Parameters }$\protect\gamma $}

\bigskip

Let $d^{^{\prime }}(G)$ be the largest degree of an irreducible character of
a group $G$. \ We define for $G$ the number $\gamma (G)$ by:

\bigskip

\begin{description}
\item[\textit{(4.24)}] $\gamma (G):=$ $\inf \left\{ \gamma \text{ }\epsilon 
\text{ }%
\mathbb{R}
^{+}\text{ }|\text{ }\left\vert G\right\vert ^{\frac{1}{\gamma }%
}=d^{^{\prime }}(G)\right\} .$
\end{description}

\bigskip

By \textit{Theorem 3.10 }$d^{^{\prime }}(G)\leq \left( \left\vert
G\right\vert -1\right) ^{1/2}<\left\vert G\right\vert ^{1/2}$, hence it
follows from the definition \textit{(4.24)} that $\gamma (G)>2$. \ Since for
a fixed group $G$, $\underset{\gamma \longrightarrow \infty }{\lim }%
\left\vert G\right\vert ^{1/\gamma }=1$, and $d^{^{\prime }}(G)=1$ for an
Abelian group $G$, we, therefore, define $\gamma (G)=\infty $ if $G$ is an
Abelian group. \ We note that $\gamma (G)$ can always be computed as:

\bigskip

\begin{description}
\item[\textit{(4.25)}] $\gamma (G)=\log _{d^{^{\prime }}(G)}\left\vert
G\right\vert $
\end{description}

\bigskip

provided we know $d^{^{\prime }}(G)$. \ An example: $d^{^{\prime
}}(Sym_{3})= $ $2$, and therefore $\gamma (Sym_{3})=\log _{2}6\approx 2.585$.

\bigskip

By \textit{Theorem 3.10} we know that\textit{\ }$\left( \frac{\left\vert
G\right\vert }{c(G)}\right) ^{1/2}<d^{^{\prime }}(G)<\left( \left\vert
G\right\vert -1\right) ^{1/2}$ iff $G$ is non-Abelian, and, therefore, the
bounds for the $\gamma (G)$ of non-Abelian groups $G$, are given by:

\bigskip

\begin{description}
\item[\textit{(4.26)}] $2\frac{\log \left\vert G\right\vert }{\log \left(
\left\vert G\right\vert -1\right) }<\gamma (G)<2\frac{\log \left\vert
G\right\vert }{\log \left\vert G\right\vert -\log c(G)}$.
\end{description}

\bigskip

For example, $c(Sym_{3})=3$, and therefore $2\frac{\log \left\vert
Sym_{3}\right\vert }{\log \left( \left\vert Sym_{3}\right\vert -1\right) }=2%
\frac{\log 6}{\log 5}\approx 2.226<\gamma (Sym_{3})=\log _{2}6\approx
2.585<5.17\approx 2\frac{\log 6}{\log 6-\log 3}=2\frac{\log \left\vert
Sym_{3}\right\vert }{\log \left\vert Sym_{3}\right\vert -\log c(Sym_{3})}$.

\bigskip

This is a basic result for normal subgroups.

\bigskip

\begin{lemma}
For any nontrivial normal subgroup $H\vartriangleleft G$ such that $%
d^{^{\prime }}(G)\geq \max \left( d^{^{\prime }}(H),d^{^{\prime
}}(G/H)\right) $, $\gamma (G)\leq \gamma (H)+\gamma (G/H)$.
\end{lemma}

\begin{proof}
For a nontrivial normal subgroup $H\vartriangleleft G$ the assumption of $%
d^{^{\prime }}(G)\geq \max \left( d^{^{\prime }}(H),d^{^{\prime
}}(G/H)\right) $ means that%
\begin{eqnarray*}
\gamma (G) &=&\frac{\log \left\vert G\right\vert }{\log d^{^{\prime }}(G)} \\
&=&\frac{\log \left\vert H\right\vert }{\log d^{^{\prime }}(G)}+\frac{\log
\left\vert G/H\right\vert }{\log d^{^{\prime }}(G)} \\
&\leq &\frac{\log \left\vert H\right\vert }{\log d^{^{\prime }}(H)}+\frac{%
\log \left\vert G/H\right\vert }{\log d^{^{\prime }}(G/H)} \\
&=&\gamma (H)+\gamma (G/H).
\end{eqnarray*}
\end{proof}

\bigskip

This is a basic result for direct product groups.

\bigskip

\begin{lemma}
$\gamma (G^{\times k})=\gamma (G)$, where $G^{\times k}$ is the $k$-fold
direct product of $G$.
\end{lemma}

\begin{proof}
We observe first that $d^{^{\prime }}\left( G^{\times k}\right) =d^{^{\prime
}}\left( G\right) ^{k}$, [SER1977]. \ Then%
\begin{eqnarray*}
\gamma (G^{\times k}) &=&\log _{d^{^{\prime }}(G^{\times k})}\left\vert
G^{\times k}\right\vert \\
&=&\log _{d^{^{\prime }}(G)^{k}}\left\vert G\right\vert ^{k} \\
&=&\frac{\NEG{k}}{\NEG{k}}\log _{d^{^{\prime }}(G)}\left\vert G\right\vert \\
&=&\gamma (G).
\end{eqnarray*}
\end{proof}

\section{{\protect\Large Fundamental Relations between }${\protect\Large 
\protect\alpha },${\protect\Large \ }${\protect\Large \protect\gamma }$%
{\protect\Large \ and the Exponent }${\protect\Large \protect\omega }$}

\bigskip

Here we derive important relations between the exponent $\omega $ and the
parameters $\alpha $ and $\gamma $.

\bigskip

\subsection{\textit{Preliminaries}}

\bigskip

Let $\omega $ be the usual exponent of matrix multiplication over $%
\mathbb{C}
$. \ The following is an important result.

\bigskip

\begin{theorem}
$\left\vert G\right\vert ^{\frac{\omega }{\alpha (G)}}=D_{2}(G)^{\frac{%
\omega }{\alpha (G)}}\leq D_{\omega }(G)$.
\end{theorem}

\begin{proof}
Let $\left\langle n^{^{\prime }},m^{^{\prime }},p^{^{\prime }}\right\rangle $
$\epsilon $ $\mathfrak{S}(G)$ be the maximal tensor realized by $G$ of size $%
z^{^{\prime }}(G)=n^{^{\prime }}m^{^{\prime }}p^{^{\prime }}$ uniquely
determining $\alpha (G)$, and by definition (\textit{4.18}) $z^{^{\prime
}}(G)=n^{^{\prime }}m^{^{\prime }}p^{^{\prime }}=|G|^{\frac{3}{\alpha (G)}}$%
. \ Using \textit{Proposition 2.13 }and \textit{Corollary 4.14}:%
\begin{eqnarray*}
\left( n^{^{\prime }}m^{^{\prime }}p^{^{\prime }}\right) ^{\frac{\omega }{3}%
} &=&|G|^{\frac{\omega }{\alpha (G)}}=D_{2}(G)^{\frac{\omega }{\alpha (G)}}%
\hspace{0.65in}(\ast ) \\
&\leq &\Re \left( \left\langle n^{^{\prime }},m^{^{\prime }},p^{^{\prime
}}\right\rangle \right) \\
&\leq &\text{ }\underset{\varrho \text{ }\epsilon \text{ }Irrep(G)}{\sum }%
\Re \left( \left\langle d_{\varrho },d_{\varrho },d_{\varrho }\right\rangle
\right) .
\end{eqnarray*}%
From \textit{Theorem 4.13 }$\left\langle n,m,p\right\rangle \leq _{K}%
\underset{\varrho \text{ }\epsilon \text{ }Irrep(G)}{\oplus }\left\langle
d_{\varrho },d_{\varrho },d_{\varrho }\right\rangle $ and taking $r^{th}$
tensor product powers on either side we obtain%
\begin{equation*}
\left\langle n^{r},m^{r},p^{r}\right\rangle \leq _{K}\underset{\varrho
_{1,...,}\varrho _{r}\text{ }\epsilon \text{ }Irrep(G)}{\oplus }\left\langle
d_{\varrho _{1}}\cdot \cdot \cdot d_{\varrho _{r}},d_{\varrho _{1}}\cdot
\cdot \cdot d_{\varrho _{r}},d_{\varrho _{1}}\cdot \cdot \cdot d_{\varrho
_{r}}\right\rangle .
\end{equation*}%
By \textit{Proposition 2.12} for each $\varepsilon >0$ there exists a
constant $C_{\varepsilon }\geq 1$ such that for all $k$ we have:%
\begin{equation*}
\Re \left( \left\langle k,k,k\right\rangle \right) \leq C_{\varepsilon
}k^{\omega +\varepsilon }.
\end{equation*}%
Hence, taking ranks on either side of $(\ast )$ we obtain that: 
\begin{eqnarray*}
D_{2}\left( G\right) ^{\frac{r\omega }{\alpha (G)}} &\leq &C_{\varepsilon
}\left( \underset{\varrho \text{ }\epsilon \text{ }Irrep(G)}{\sum }%
d_{\varrho }^{\omega +\varepsilon }\right) ^{r} \\
&=&C_{\varepsilon }D_{\omega +\varepsilon }\left( G\right) ^{r}
\end{eqnarray*}%
for some $C_{\varepsilon }>0$ depending on some $\varepsilon >0$. \ Taking
the $r^{th}$ root on either side we obtain%
\begin{equation*}
D_{2}(G)^{\frac{\omega }{\alpha (G)}}\leq \sqrt[r]{C_{\varepsilon }}%
D_{\omega +\varepsilon }\left( G\right) .
\end{equation*}%
If we take the limit as $r\longrightarrow \infty $, and then the limit as $%
\varepsilon \longrightarrow 0$ we obtain finally that:%
\begin{equation*}
D_{2}(G)^{\frac{\omega }{\alpha (G)}}\leq \text{ }D_{\omega }(G).
\end{equation*}
\end{proof}

\bigskip

By the maximality of $\left\langle n^{^{\prime }},m^{^{\prime }},p^{^{\prime
}}\right\rangle $ in $\mathfrak{S}(G)$ we have an elementary corollary.

\bigskip

\begin{corollary}
If $\left\langle n,m,p\right\rangle $ $\epsilon $ $\mathfrak{S}(G)$ then $%
(1) $ $\left( nmp\right) ^{\frac{1}{3}}\leq D_{\omega }(G)^{\frac{1}{\omega }%
}$, and $(2)$ $\left( nmp\right) ^{\frac{1}{3}}\leq d^{^{\prime }}(G)^{1-%
\frac{2}{\omega }}\left\vert G\right\vert ^{\frac{1}{\omega }}$.
\end{corollary}

\begin{proof}
For any $\left\langle n,m,p\right\rangle $ $\epsilon $ $\mathfrak{S}(G)$, by 
\textit{(4.16)} and \textit{Theorem 4.23}, $\left( nmp\right) ^{\frac{\omega 
}{3}}\leq z^{^{\prime }}(G)^{\frac{\omega }{3}}\leq D_{2}(G)^{\frac{\omega }{%
\alpha (G)}}\leq D_{\omega }(G)$, and taking $\omega ^{th}$ roots, we have
the result. \ Part $(2)$ is a consequence of applying \textit{(3.12) }to the
right-hand side of $(1)$.
\end{proof}

\bigskip

\textit{Theorem 4.23} can be reexpressed as the relation.

\bigskip

\begin{description}
\item[\textit{(4.27)}] $\left\vert G\right\vert ^{\frac{1}{\alpha (G)}}\leq
D_{\omega }(G)^{\frac{1}{\omega }}$.
\end{description}

\bigskip

Using \textit{(4.18) }we know:

\bigskip

\begin{description}
\item[\textit{(4.28)}] $\left\vert G\right\vert ^{\frac{1}{\alpha (G)}}\leq
D_{\omega }(G)^{\frac{1}{\omega }}\leq \left\vert G\right\vert ^{\frac{1}{2}%
} $.
\end{description}

\bigskip

In the impossible case that $\alpha (G)=2$ we would have that $\left\vert
G\right\vert ^{\frac{1}{2}}\leq D_{\omega }(G)^{\frac{1}{\omega }}\leq
\left\vert G\right\vert ^{\frac{1}{2}}$, which would imply that $\omega =2$.
\ However, since by \textit{Lemma 4.15} $\alpha (G)>2$ always it follows
that $\left\vert G\right\vert ^{\frac{1}{\alpha (G)}}<\left\vert
G\right\vert ^{\frac{1}{2}}$ always and, therefore, the first estimate in 
\textit{(4.28)} will be strict if $\omega $ could be pushed to $2$.

\bigskip

The following is a useful result.

\bigskip

\begin{corollary}
$D_{r}(G)\leq \left\vert G\right\vert ^{\frac{r-2}{\gamma (G)}+1}$ for $%
r\geq 2$.
\end{corollary}

\begin{proof}
By \textit{(3.12) }$D_{r}(G)\leq d^{^{\prime }}(G)^{r-2}\left\vert
G\right\vert $ for $r\geq 2$, and by \textit{(4.25)} $d^{^{\prime
}}(G)=\left\vert G\right\vert ^{\frac{1}{\gamma (G)}}$. \ Then 
\begin{eqnarray*}
D_{r}(G) &\leq &\left( \left\vert G\right\vert ^{\frac{1}{\gamma (G)}%
}\right) ^{r-2}\left\vert G\right\vert \\
&=&\left\vert G\right\vert ^{\frac{r-2}{\gamma (G)}+1}.
\end{eqnarray*}
\end{proof}

\bigskip

Using the above, we can prove the following fundamental relation.

\bigskip

\begin{corollary}
If $G$ is a non-Abelian group such that $\alpha (G)<\gamma (G)$ then $\omega
\leq \alpha \left( G\right) \left( \frac{\gamma (G)-2}{\gamma (G)-\alpha (G)}%
\right) $.
\end{corollary}

\begin{proof}
By \textit{Theorem 4.23 }$\left\vert G\right\vert ^{\frac{\omega }{\alpha (G)%
}}\leq D_{\omega }(G)$. $\ $For $r=\omega $ in \textit{Corollary 4.25 }$%
D_{\omega }(G)\leq \left\vert G\right\vert ^{\frac{\omega -2}{\gamma (G)}+1}$%
, from which we derive $\left\vert G\right\vert ^{\frac{\omega }{\alpha (G)}%
}\leq \left\vert G\right\vert ^{\frac{\omega -2}{\gamma (G)}+1}$. \ This
implies that $\frac{\omega }{\alpha (G)}\leq \frac{\omega -2}{\gamma (G)}+1$%
, which is equivalent to $\omega \left( \frac{1}{\alpha (G)}-\frac{1}{\gamma
(G)}\right) \leq 1-\frac{2}{\gamma (G)}$. \ Since $\gamma (G)-\alpha (G)>0$
by assumption, it follows that $\left( \frac{1}{\alpha (G)}-\frac{1}{\gamma
(G)}\right) >0$. \ Dividing both sides of the previous estimate by $\left( 
\frac{1}{\alpha (G)}-\frac{1}{\gamma (G)}\right) $ we get $\omega \leq
\left( 1-\frac{2}{\gamma (G)}\right) /\left( \frac{1}{\alpha (G)}-\frac{1}{%
\gamma (G)}\right) =\alpha \left( G\right) \left( \frac{\gamma (G)-2}{\gamma
(G)-\alpha (G)}\right) $.
\end{proof}

\bigskip

\subsection{\textit{Fundamental Results for the Exponent }$\protect\omega $%
\textit{\ using Single non-Abelian Groups}}

\bigskip

\begin{corollary}
If $G$ is a non-Abelian group such that $\left\vert G\right\vert ^{\frac{1}{%
\alpha (G)}}>D_{3}(G)^{\frac{1}{3}}$ then $\omega <3$. \ Equivalently, $%
\omega <3$ if $z^{^{\prime }}(G)^{\frac{1}{3}}>D_{3}(G)^{\frac{1}{3}}$.
\end{corollary}

\begin{proof}
If $\left\vert G\right\vert ^{\frac{1}{\alpha (G)}}>D_{3}(G)^{\frac{1}{3}}$,
by \textit{(4.28)}, $D_{3}(G)^{\frac{1}{3}}<\left\vert G\right\vert ^{\frac{1%
}{\alpha (G)}}\leq D_{\omega }(G)^{\frac{1}{\omega }}$, which implies that
by the convexity property \textit{(3.11)}, $\omega <3$. \ The second part
follows from the fact that $\left\vert G\right\vert ^{\frac{1}{\alpha (G)}%
}=z^{^{\prime }}(G)^{\frac{1}{3}}$ (see \textit{(4.18)}).
\end{proof}

\bigskip

This is a trivial result because it has been proven that $\omega <2.38$
[CW1990]. \ More useful is the following.

\bigskip

\begin{corollary}
$(1)$ $\omega \leq t<3$ for some $t>2$, if there is a non-Abelian group $G$
such that $\alpha (G)<\gamma (G)$ and $\alpha (G)\left( \frac{\gamma (G)-2}{%
\gamma (G)-\alpha (G)}\right) \leq t$. \ $(2)$ An equivalent, but more
precise statement is that $\omega \leq t<3$, for some $t>2$, if there is a
non-Abelian group $G$ such that $z^{^{\prime }}(G)^{\frac{1}{3}}>d^{^{\prime
}}(G)$ and $\frac{z^{^{\prime }}(G)^{\frac{t}{3}}}{d^{^{\prime }}(G)^{t-2}}%
\geq G$.
\end{corollary}

\begin{proof}
By \textit{Corollary 4.26 }$\omega \leq \alpha (G)\left( \frac{\gamma (G)-2}{%
\gamma (G)-\alpha (G)}\right) $ if $\alpha (G)<\gamma (G)$. \ This means
that if, in addition, $\alpha (G)\left( \frac{\gamma (G)-2}{\gamma
(G)-\alpha (G)}\right) \leq t$ for some $2<t<3$ then $\omega \leq t$. \ $(2)$
From \textit{(4.17)}, \textit{(4.23)}, and \textit{(4.25)}, we see that $%
\alpha (G)<\gamma (G)$ is equivalent to $z^{^{\prime }}(G)^{\frac{1}{3}%
}>d^{^{\prime }}(G)$. \ By $(1)$ the additional condition $\alpha (G)\left( 
\frac{\gamma (G)-2}{\gamma (G)-\alpha (G)}\right) \leq t$ needed to prove $%
\omega \leq t$ is equivalent, to $\frac{\log \left\vert G\right\vert }{\log
z^{^{\prime }}(G)^{1/3}}\left( \frac{\log \left\vert G\right\vert }{\log
d^{^{\prime }}(G)}-2\right) \leq \left( \frac{\log \left\vert G\right\vert }{%
\log d^{^{\prime }}(G)}-\frac{\log \left\vert G\right\vert }{\log
z^{^{\prime }}(G)^{1/3}}\right) t$. \ Dividing both sides of the inequality
by $\log \left\vert G\right\vert $ we will still have a term on the left
with numerator $\log \left\vert G\right\vert $, and making it the subject of
the inequality on the left, this becomes $\log \left\vert G\right\vert \leq
t\log z^{^{\prime }}(G)^{\frac{1}{3}}+\left( 2-t\right) \log d^{^{\prime
}}(G)=\log z^{^{\prime }}(G)^{\frac{t}{3}}-\log d^{^{\prime
}}(G)^{(t-2)}=\log \frac{z^{^{\prime }}(G)^{\frac{t}{3}}}{d^{^{\prime
}}(G)^{(t-2)}}$. \ Taking antilogarithms we have the result.
\end{proof}

\bigskip

For any given group $G$, from \textit{(4.23) }we see that $\left\vert
G\right\vert ^{\frac{1}{3}}<z^{^{\prime }}(G)^{\frac{1}{3}}<\left\vert
G\right\vert ^{\frac{1}{2}}$, and from part $(5)$ of \textit{Theorem 3.10 }$%
c(G)^{-\frac{1}{2}}\left\vert G\right\vert ^{\frac{1}{2}}<d^{^{\prime
}}(G)<\left( \left\vert G\right\vert -1\right) ^{\frac{1}{2}}$ if $G$ is
non-Abelian. \ The intersection of these intervals lies in the open interval 
$\left( c(G)^{-\frac{1}{2}}\left\vert G\right\vert ^{\frac{1}{2}},\left\vert
G\right\vert ^{\frac{1}{2}}\right) $, and \textit{Corollary 4.29} will let
us prove $\omega \leq t<3$ for some $t>2$ if we can find a non-Abelian group 
$G$ realizing matrix multiplication of largest size $z^{^{\prime }}(G)$ and
with an irreducible character of largest degree $d^{^{\prime }}(G)$ such
that $c(G)^{-\frac{1}{2}}\left\vert G\right\vert ^{\frac{1}{2}}<d^{^{\prime
}}(G)<z^{^{\prime }}(G)^{\frac{1}{3}}<\left\vert G\right\vert ^{\frac{1}{2}%
}\leq \frac{z^{^{\prime }}(G)^{\frac{t}{6}}}{d^{^{\prime }}(G)^{\frac{t-2}{2}%
}}$.

\bigskip

\subsection{\textit{Fundamental Results for the Exponent }$\protect\omega $%
\textit{\ using Families of non-Abelian Groups}}

\bigskip

The following results describe ways of proving that $\omega =2$ via families
of non-Abelian groups. \ The first is as follows.

\bigskip

\begin{corollary}
If $\left\{ G_{k}\right\} $ is a family of non-Abelian groups such that $%
\alpha \left( G_{k}\right) \equiv \alpha _{k}=2+o(1)$, and $\gamma \left(
G_{k}\right) \equiv \gamma _{k}=2+o(1)$, and $\alpha _{k}-2=o(\gamma _{k}-2)$%
, as $k\longrightarrow \infty $, then $\omega =2$.
\end{corollary}

\begin{proof}
Assuming the conditions $\alpha _{k}=2+o(1)$, $\gamma _{k}=2+o(1)$, $\alpha
_{k}-2=o(\gamma _{k}-2)$, as $k\longrightarrow \infty $, for the family $%
\left\{ G_{k}\right\} $, we will have 
\begin{eqnarray*}
&&\left( \alpha _{k}-\gamma _{k}\right) \\
&=&\left( \alpha _{k}-2\right) -\left( \gamma _{k}-2\right) \\
&=&o\left( \gamma _{k}-2\right) -\left( \gamma _{k}-2\right) \\
&<&-\frac{1}{2}\left( \gamma _{k}-2\right) <0
\end{eqnarray*}%
for sufficiently large $k$. \ Then by \textit{Corollary 4.26}, as $%
k\longrightarrow \infty $, we will have 
\begin{eqnarray*}
\omega &\leq &\alpha _{k}\left( \frac{\gamma _{k}-2}{\gamma _{k}-\alpha _{k}}%
\right) \\
&=&\frac{\alpha _{k}}{\left( 1-\left( \alpha _{k}-2\right) /\left( \gamma
_{k}-2\right) \right) } \\
&=&\frac{2+o\left( 1\right) }{1-o\left( 1\right) /o\left( 1\right) } \\
&=&\frac{2+o\left( 1\right) }{1-o\left( 1\right) } \\
&\longrightarrow &2^{+}.
\end{eqnarray*}
\end{proof}

\bigskip

From \textit{(4.18)}, $z^{^{\prime }}(G)^{\frac{1}{3}}=\left\vert
G\right\vert ^{\frac{1}{\alpha (G)}}$, and from \textit{(4.25)}, $%
d^{^{\prime }}(G)=\left\vert G\right\vert ^{\frac{1}{\gamma (G)}}$, and for
a family $\left\{ G_{k}\right\} $ of non-Abelian groups $G_{k}$ realizing
matrix multiplications of maximal sizes $z_{k}^{^{\prime }}\equiv
z^{^{\prime }}(G_{k})$ and maximal irreducible character degrees $%
d_{k}^{^{\prime }}\equiv d^{^{\prime }}(G_{k})$, the conditions $\alpha
_{k}=2+o(1)$ and $\gamma _{k}=2+o(1)$ are equivalent to the conditions $%
\frac{\log \left\vert G_{k}\right\vert ^{\frac{1}{2}}}{\log z_{k}^{^{\prime }%
\frac{1}{3}}}=1+o(1)$ and $\frac{\log \left( \left\vert G_{k}\right\vert
-1\right) ^{\frac{1}{2}}}{\log d_{k}^{^{\prime }}}=1+o(1)$, respectively,
which are implied by the conditions $\left\vert G_{k}\right\vert ^{\frac{1}{2%
}}-z_{k}^{^{\prime }\frac{1}{3}}=o(1)$ and $\left( \left\vert
G_{k}\right\vert -1\right) ^{\frac{1}{2}}-d_{k}^{^{\prime }}=o(1)$,
respectively. \ Here, we describe a specific result for a family of
non-Abelian groups satisfying the latter conditions.

\bigskip

\begin{theorem}
Let $\left\{ G_{k}\right\} $ be a family of non-Abelian groups $G_{k}$
realizing matrix multiplications of maximal sizes $z_{k}^{^{\prime }}\equiv
z^{^{\prime }}(G_{k})$ and with maximal irreducible character degrees $%
d_{k}^{^{\prime }}\equiv d_{k}^{^{\prime }}(G)$. \ Then, $(1)$ $\omega =2$
if $\left\vert G_{k}\right\vert ^{\frac{1}{2}}-z_{k}^{^{\prime }\frac{1}{3}%
}=o(1)$ and $\left( \left\vert G_{k}\right\vert -1\right) ^{\frac{1}{2}%
}-d_{k}^{^{\prime }}=o(1)$ such that $\left\vert G_{k}\right\vert ^{\frac{1}{%
2}}-z_{k}^{^{\prime }\frac{1}{3}}=$ $o\left( \left( \left\vert
G_{k}\right\vert -1\right) ^{\frac{1}{2}}-d_{k}^{^{\prime }}\right) $ as $%
k\longrightarrow \infty $. \ And more generally, $(2)$ $\omega =2$ if $%
\left\vert G_{k}\right\vert \longrightarrow \infty $ as $k\longrightarrow
\infty $ and there exists a sequence $\left\{ C_{k}\right\} $ of constants $%
C_{k}$ for the $G_{k}$ such that $2\leq C_{k}\leq \left\vert
G_{k}\right\vert -1$, $\left\vert G_{k}\right\vert \geq C_{k}\left( 1+\frac{1%
}{C_{k}-1}\right) $, $C_{k}\longrightarrow \infty $, $C_{k}=o(\left\vert
G_{k}\right\vert )$, $\left( \left\vert G_{k}\right\vert -C_{k}\right) ^{%
\frac{1}{2}}-d_{k}^{^{\prime }}=o(1)$, $\left\vert G_{k}\right\vert ^{\frac{1%
}{2}}-z_{k}^{^{\prime }\frac{1}{3}}=o(1)$ and $\left\vert G_{k}\right\vert ^{%
\frac{1}{2}}-z_{k}^{^{\prime }\frac{1}{3}}=$ $o\left( \left( \left\vert
G_{k}\right\vert -C_{k}\right) ^{\frac{1}{2}}-d_{k}^{^{\prime }}\right) $,
as $k\longrightarrow \infty $.
\end{theorem}

\begin{proof}
For an arbitrary group $G$, if we compare \textit{(4.16)} and \textit{(4.25)}%
, we see that $\alpha (G)<\gamma (G)$ iff $z^{^{\prime }}(G)^{\frac{1}{3}%
}>d_{k}^{^{\prime }}(G)$, $\alpha (G)=\gamma (G)$ iff $z^{^{\prime }}(G)^{%
\frac{1}{3}}=d^{^{\prime }}(G)$, and $\gamma (G)<\alpha (G)$ iff $%
z^{^{\prime }}(G)^{\frac{1}{3}}<d^{^{\prime }}(G)$. \ Moreover, $\alpha (G)$
is close to $2$ iff $z^{^{\prime }}(G)^{\frac{1}{3}}$ is close to $%
\left\vert G\right\vert ^{\frac{1}{2}}$, and $\gamma (G)$ is close to $2$
iff $d^{^{\prime }}(G)$ is close to $\left\vert G\right\vert ^{\frac{1}{2}}$.%
\newline

$(1)$ \ By \textit{Theorem 3.10}, $\left( \frac{\left\vert G\right\vert }{%
c(G)}\right) ^{\frac{1}{2}}<d^{^{\prime }}(G)<\left( \left\vert G\right\vert
-1\right) ^{\frac{1}{2}}<\left\vert G\right\vert ^{\frac{1}{2}}$ for a
non-Abelian group $G$, therefore, for the non-Abelian family $\left\{
G_{k}\right\} $ the most that we can have is $d_{k}^{^{\prime
}}\longrightarrow $ $\left( \left\vert G_{k}\right\vert -1\right) ^{\frac{1}{%
2}-}$, as $k\longrightarrow \infty $. \ But $d_{k}^{^{\prime
}}\longrightarrow $ $\left( \left\vert G_{k}\right\vert -1\right) ^{\frac{1}{%
2}-}$, as $k\longrightarrow \infty $ implies $\gamma _{k}\longrightarrow
\log _{\left( \left\vert G_{k}\right\vert -1\right) ^{1/2}}\left\vert
G_{k}\right\vert =2\log _{\left( \left\vert G_{k}\right\vert -1\right)
}\left\vert G_{k}\right\vert \longrightarrow 2^{+}$, i.e. $\gamma
_{k}=2+o(1) $, and if, in addition, $z_{k}^{^{\prime }\frac{1}{3}%
}\longrightarrow $.$\left\vert G_{k}\right\vert ^{\frac{1}{2}-}$ faster than 
$d_{k}^{^{\prime }}\longrightarrow $ $\left( \left\vert G_{k}\right\vert
-1\right) ^{\frac{1}{2}-}$ as $k\longrightarrow \infty $ then $\alpha
_{k}=2+o(1)$ and $\gamma _{k}=2+o(1)$ such that $\alpha _{k}-2=o(\gamma
_{k}-2)$. \ This can also be written as $\left\vert G_{k}\right\vert ^{\frac{%
1}{2}}-z_{k}^{^{\prime }\frac{1}{3}}=o(1),$ $\left( \left\vert
G_{k}\right\vert -1\right) ^{\frac{1}{2}}-d_{k}^{^{\prime }}=o(1)$ and $%
\left\vert G_{k}\right\vert ^{\frac{3}{2}}-z_{k}^{^{\prime }\frac{1}{3}}=$ $%
o\left( \left( \left\vert G_{k}\right\vert -1\right) ^{\frac{1}{2}%
}-d_{k}^{^{\prime }}\right) $, as $k\longrightarrow \infty $, which implies $%
\alpha _{k}=2+o(1),$ $\gamma _{k}=2+o(1)$ and $\alpha _{k}-2=o(\gamma
_{k}-2) $, as $k\longrightarrow \infty $, which, by \textit{Corollary 4.29},
implies that $\omega =2$.\newline

$(2)$ \ Assume all the conditions in $(2)$. \ In particular, for the
constants $2\leq C_{k}\leq \left\vert G_{k}\right\vert -1$, the condition $%
\left\vert G_{k}\right\vert \geq C_{k}\left( 1+\frac{1}{C_{k}-1}\right) $
means that $\gamma _{k}=\log _{d_{k}^{^{\prime }}}\left\vert
G_{k}\right\vert \leq \log _{d_{k}^{^{\prime }}}\left( \left\vert
G_{k}\right\vert -C_{k}\right) +\log _{d_{k}^{^{\prime }}}C_{k}$. \ Then,
the condition $\left( \left\vert G_{k}\right\vert -C_{k}\right) ^{\frac{1}{2}%
}-d_{k}^{^{\prime }}=o(1)$, implying $d_{k}^{^{\prime }}\longrightarrow
\left( \left\vert G_{k}\right\vert -C_{k}\right) ^{\frac{1}{2}-}$, as $%
k\longrightarrow \infty $, together with $\left\vert G_{k}\right\vert
\longrightarrow \infty $, $C_{k}\longrightarrow \infty $, $%
C_{k}=o(\left\vert G_{k}\right\vert )$, as $k\longrightarrow \infty $,
implies that $\gamma _{k}\longrightarrow 2\left( \log _{\left( \left\vert
G_{k}\right\vert -C_{k}\right) }\left( \left\vert G_{k}\right\vert
-C_{k}\right) +\log _{\left( \left\vert G_{k}\right\vert -C_{k}\right)
}C_{k}\right) =2+o(1)$, as $k\longrightarrow \infty $. \ In addition, the
conditions $\left\vert G_{k}\right\vert ^{\frac{1}{2}}-z_{k}^{^{\prime }%
\frac{1}{3}}=o(1)$ and $\left\vert G_{k}\right\vert ^{\frac{1}{2}%
}-z_{k}^{^{\prime }\frac{1}{3}}=$ $o\left( \left( \left\vert
G_{k}\right\vert -C_{k}\right) ^{\frac{1}{2}}-d_{k}^{^{\prime }}\right) $, $%
k\longrightarrow \infty $, imply that $\alpha _{k}\longrightarrow 2^{+}$
faster than $\gamma _{k}\longrightarrow 2^{+}$ as $k\longrightarrow \infty $%
, i.e. $\alpha _{k}=2+o(1)$ and $\gamma _{k}=2+o(1)$ such that $\alpha
_{k}-2=o(\gamma _{k}-2)$, as $k\longrightarrow \infty $, which, by \textit{%
Corollary 4.29} implies that $\omega =2$.
\end{proof}

\pagebreak \bigskip

\chapter{\protect\huge Groups and Matrix Multiplication II}

\bigskip

Here we extend the methods introduced in \textit{Chapter 4} to study the
complexity of simultaneous independent multiplications of several pairs of
matrices via a single group using the concept of simultaneous triple
products. \ We derive some important results that bound the exponent $\omega 
$ in terms of the sizes of simultaneously realized tensors of single groups
in relation to the sizes of the groups. \ In our analysis, it appears that
the sharpness of estimates for $\omega $ is positively related to the 
\textit{number} of simultaneous matrix multiplications supported by a group,
and also that the best groups, in this regard, seem to be groups which are
wreath products of Abelian groups with symmetric groups.

\bigskip

\section{\protect\Large Realizing Simultaneous, Independent Matrix
Multiplications in Groups}

\bigskip

\subsection{\textit{Groups and Families of Simultaneous Index Triples and
Tensors}}

\bigskip

We define the right-quotient set $Q(X,Y)$ of any pair of subsets $X,$ $Y$ of
a finite group $G$ by:

\bigskip

\begin{description}
\item[\textit{(5.1)}] $Q(X,Y)=\left\{ xy^{-1}\text{ }|\text{ }x\text{ }%
\epsilon \text{ }X,\text{ }y\text{ }\epsilon \text{ }Y\right\} .$
\end{description}

\bigskip

Let $I$ be a finite index set. \ A collection $\left\{
(S_{i},T_{i},U_{i})\right\} _{i\text{ }\epsilon \text{ }I}$ of triples $%
(S_{i},T_{i},U_{i})$ of subsets $S_{i},T_{i},U_{i}\subseteq G$, of sizes $%
\left\vert S_{i}\right\vert =m_{i}$, $\left\vert T_{i}\right\vert =p_{i}$, $%
\left\vert U_{i}\right\vert =q_{i}$ respectively, is said to satisfy the 
\textit{simultaneous triple product property} (STPP) iff it is the case that:

\bigskip

\begin{description}
\item[\textit{(5.2)}] each $(S_{i},T_{i},U_{i})$ satisfies the TPP and $%
s_{i}^{^{\prime }}s_{j}^{^{-1}}t_{j}^{^{\prime
}}t_{k}^{^{-1}}u_{k}^{^{\prime }}u_{i}^{^{-1}}=1_{G}\Longrightarrow i=j=k$
\end{description}

\bigskip

for all $s_{i}^{^{\prime }}s_{j}^{^{-1}}$ $\epsilon $ $Q(S_{i},S_{j}),$ $%
t_{j}^{^{\prime }}t_{k}^{^{-1}}$ $\epsilon $ $Q(T_{j},T_{k}),$ $%
u_{k}^{^{\prime }}u_{i}^{^{-1}}$ $\epsilon $ $Q(U_{k},U_{i}),$ $i,j,k$ $%
\epsilon $ $I$. \ In this case, $G$ is said to \textit{simultaneously
realize }the corresponding collection $\left\{ \left\langle
m_{i},p_{i},q_{i}\right\rangle \right\} _{i\text{ }\epsilon \text{ }I}$ of
tensors through the collection $\left\{ (S_{i},T_{i},U_{i})\right\} _{i\text{
}\epsilon \text{ }I}$, which is called a \textit{collection of simultaneous
index triples}. \ For such collections, the triple product property \textit{%
(4.3)} becomes a special case of the simultaneous triple product property
when $\left\vert I\right\vert =1$. \ Thus, every triple in a collection $%
\left\{ (S_{i},T_{i},U_{i})\right\} _{i\text{ }\epsilon \text{ }I}$ of
simultaneous index triples of $G$ is an index triple of $G$, though there is
no converse for collections of index triples. \ The $i^{th}$ tensor $%
\left\langle m_{i},p_{i},q_{i}\right\rangle $ in $\left\{ \left\langle
m_{i},p_{i},q_{i}\right\rangle \right\} _{i\text{ }\epsilon \text{ }I}$ is
the matrix multiplication map $%
\mathbb{C}
^{n_{i}\times m_{i}}\times 
\mathbb{C}
^{m_{i}\times p_{i}}\longrightarrow 
\mathbb{C}
^{n_{i}\times p_{i}}$, and the significance of the simultaneous triple
product property is that it describes the property of $G$ realizing the
collection of tensors $\left\{ \left\langle m_{i},p_{i},q_{i}\right\rangle
\right\} _{i\text{ }\epsilon \text{ }I}$ via a collection $\left\{
(S_{i},T_{i},U_{i})\right\} _{i\text{ }\epsilon \text{ }I}$ of index triples
in such a way that $\left\vert I\right\vert $ simultaneous, independent
matrix multiplications can be reduced to one multiplication in its regular
group algebra $%
\mathbb{C}
G$, with a complexity not exceeding the rank of the algebra.

\subsection{\textit{Group-Algebra Embedding and Complexity of Simultaneous,
Independent Matrix Multiplications}}

\bigskip

\begin{theorem}
If $\left\{ \left\langle m_{i},p_{i},q_{i}\right\rangle \right\} _{i\text{ }%
\epsilon \text{ }I}\subseteq \mathfrak{S}(G)$ is a collection of tensors
simultaneously realized by a group $G$ then%
\begin{equation*}
(1)\text{ }\mathfrak{R}\left( \underset{i}{\oplus }\left\langle
m_{i},p_{i},q_{i}\right\rangle \right) \leq \mathfrak{R}(\mathfrak{m}_{%
\mathbb{C}
G})\leq \underset{\varrho \text{ }\epsilon \text{ }Irrep(G)}{\sum }\mathfrak{%
R}\left( \left\langle d_{\varrho },d_{\varrho },d_{\varrho }\right\rangle
\right)
\end{equation*}%
and iff in addition $G$ is Abelian then 
\begin{equation*}
(2)\text{ }\mathfrak{R}\left( \underset{i}{\oplus }\left\langle
m_{i},p_{i},q_{i}\right\rangle \right) \leq \left\vert G\right\vert =%
\mathfrak{R}(\mathfrak{m}_{%
\mathbb{C}
G}).
\end{equation*}
\end{theorem}

\begin{proof}
The procedure used here is a natural generalization of \textit{Theorem 4.13}%
. \ If $S_{v},T_{v},U_{v}\subseteq G$, $1\leq v\leq r$, is a collection of $%
r $ triples satisfying the simultaneous triple product property, and $%
\left\{ \left( A_{v},B_{v}\right) \right\} _{v=1}^{r}$ is a given collection
of $r$ pairs of $m_{v}\times p_{v}$ and $p_{v}\times q_{v}$ matrices $%
A_{v}=\left( A_{i_{v}j_{v}}\right) $ and $B_{v}=\left( B_{j_{v}^{^{\prime
}}k_{v}}\right) $ respectively, then we embed these pairs in $%
\mathbb{C}
G$ as 
\begin{eqnarray*}
\mathfrak{a}_{v}\left( A_{v}\right) &=&\overline{A}_{v}=\underset{s_{v}\text{
}\epsilon \text{ }S_{v},\text{ }t_{v}\text{ }\epsilon \text{ }T_{v}}{\sum }%
A_{s_{v},t_{v}}s_{v}^{-1}t_{v} \\
\mathfrak{b}_{v}\left( B_{v}\right) &=&\overline{B}_{v}=\underset{%
t_{v}^{^{\prime }}\text{ }\epsilon \text{ }T_{v},\text{ }u_{v}\text{ }%
\epsilon \text{ }U_{v}}{\sum }B_{t_{v}^{^{\prime }},u_{v}}t_{v}^{^{\prime
}-1}u_{v}
\end{eqnarray*}%
\ via pairs of injective, linear embedding maps $\mathfrak{a}_{v}:%
\mathbb{C}
^{m_{v}\times p_{v}}\longrightarrow 
\mathbb{C}
G$ and $\mathfrak{b}_{v}:%
\mathbb{C}
^{p_{v}\times q_{v}}\longrightarrow 
\mathbb{C}
G$, using the triples $S_{v},T_{v},U_{v}$, one triple for each pair, just as
described in \textit{Theorem 4.13}. \ For each $v$, the $\left(
i_{v},k_{v}\right) ^{th}$ entry $C_{i_{v}k_{v}}$ of the product $%
C_{v}=A_{v}B_{v}$ is given by $C_{i_{v}j_{v}}=$ $\underset{%
j_{v}=j_{v}^{^{\prime }}}{\sum }A_{i_{v}j_{v}}B_{j_{v}k_{v}}$. \ The product
of $\overline{A}_{v}$ and $\overline{B}_{v}$ in $%
\mathbb{C}
G$ is given by%
\begin{eqnarray*}
\overline{A}_{v}\overline{B}_{v} &=&\underset{s_{v}\text{ }\epsilon \text{ }%
S_{v},\text{ }t_{v}\text{ }\epsilon \text{ }T_{v}}{\sum }%
A_{s_{v},t_{v}}s_{v}^{-1}t_{v}\cdot \underset{t_{v}^{^{\prime }}\text{ }%
\epsilon \text{ }T_{v},\text{ }u_{v}\text{ }\epsilon \text{ }U_{v}}{\sum }%
B_{t_{v}^{^{\prime }},u_{v}}t_{v}^{^{\prime }-1}u_{v} \\
&=&\underset{s_{v}\text{ }\epsilon \text{ }S_{v},\text{ }t_{v}\text{ }%
\epsilon \text{ }T_{v}\text{ }}{\sum }\underset{t_{v}^{^{\prime }}\text{ }%
\epsilon \text{ }T_{v},\text{ }u_{v}\text{ }\epsilon \text{ }U_{v}}{\sum }%
A_{s_{v},t_{v}}B_{t_{v}^{^{\prime }},u_{v}}s_{v}^{-1}t_{v}t_{v}^{^{\prime
}-1}u_{v} \\
&=&\underset{s_{v}\text{ }\epsilon \text{ }S_{v},\text{ }u_{v}\text{ }%
\epsilon \text{ }U_{v}\text{ }}{\sum }\left( \underset{t_{v},t_{v}^{^{\prime
}}\text{ }\epsilon \text{ }T_{v}\text{ }}{\sum }A_{s_{v},t_{v}}B_{t_{v}^{^{%
\prime }},u_{v}}t_{v}t_{v}^{^{\prime }-1}\right) s_{v}^{-1}u_{v}.
\end{eqnarray*}%
As in the proof of \textit{Theorem 4.13} for a given $v$, we can recover the
matrix product $C_{v}=A_{v}B_{v}$ from $\overline{A}_{v}\overline{B}_{v}$ by
a linear, injective extraction map $\mathfrak{x}_{v}:%
\mathbb{C}
^{m_{v}\times q_{v}}\longleftarrow 
\mathbb{C}
G$ defined on the fact that for arbitrary $s_{v}^{^{\prime }}$ $\epsilon $ $%
S_{v}$, $u_{v}^{^{\prime }}$ $\epsilon $ $U_{v}$, the sum of those terms of $%
\overline{A}_{v}\overline{B}_{v}$ for which $t_{v}^{^{\prime }}=t_{v}$ all
have the group term $s_{v}^{^{\prime }-1}u_{v}^{^{\prime }}$, and the sum of
coefficients of these terms, $\underset{t_{v}=t_{v}^{^{\prime }}\text{ }%
\epsilon \text{ }T_{v}}{\sum }A_{s_{v}^{^{\prime
}},t_{v}}B_{t_{v},u_{v}^{^{\prime }}}$, corresponds $1$-to-$1$ with the $%
(i_{v},k_{v})^{th}$ entry $C_{i_{v}k_{v}}$ of $C_{v}=A_{v}B_{v}$. \ If $%
\mathfrak{c}_{v}$ denotes the embedding map $%
\mathbb{C}
^{m_{v}\times q_{v}}\longrightarrow 
\mathbb{C}
G$, this shows that $\mathfrak{x}_{v}(\overline{A}_{v}\overline{B}%
_{v})=C_{v}=\mathfrak{c}_{v}^{-1}\left( \overline{C}_{v}\right) $. \ For
each $v$, we have the composition $\mathfrak{x}_{v}\circ \mathfrak{m}_{%
\mathbb{C}
G}\circ \left( \mathfrak{a}_{v}\times \mathfrak{b}_{v}\right) \left(
A_{v},B_{v}\right) =C_{v}$ $\epsilon $ $%
\mathbb{C}
^{m_{v}\times q_{v}}$, by which we have the restrictions $\mathfrak{x}%
_{v}\circ \mathfrak{m}_{%
\mathbb{C}
G}\circ \left( \mathfrak{a}_{v}\times \mathfrak{b}_{v}\right) =\left\langle
m_{v},p_{v},q_{v}\right\rangle $ of $\mathfrak{m}_{%
\mathbb{C}
G}$ to $\left\langle m_{v},p_{v},q_{v}\right\rangle $, i.e. $\left\langle
m_{v},p_{v},q_{v}\right\rangle \leq _{%
\mathbb{C}
}\mathfrak{m}_{%
\mathbb{C}
G}$, by which we deduce $\mathfrak{R}\left( \left\langle
m_{v},p_{v},q_{v}\right\rangle \right) \leq _{%
\mathbb{C}
}\mathfrak{R}\left( \mathfrak{m}_{%
\mathbb{C}
G}\right) $ (\textit{Proposition 2.2}).\newline

Now we prove that these restrictions are simultaneous and independent. \ We
define the direct sum matrices $A=\underset{v=1}{\overset{r}{\oplus }}A_{v}$ 
$\epsilon $ $\underset{v=1}{\overset{r}{\oplus }}%
\mathbb{C}
^{m_{v}\times p_{v}},$ and $B=\underset{v=1}{\overset{r}{\oplus }}B_{v}$ $%
\epsilon $ $\underset{v=1}{\overset{r}{\oplus }}%
\mathbb{C}
^{p_{v}\times q_{v}}$. \ The product $AB$ $\epsilon $ $\underset{v=1}{%
\overset{r}{\oplus }}%
\mathbb{C}
^{m_{v}\times q_{v}}$ is the direct sum $\underset{v=1}{\overset{r}{\oplus }}%
A_{v}B_{v}$ of the $r$ block products $A_{v}B_{v}$ of blocks $A_{v}$ and $%
B_{v}$ of dimensions $m_{v}\times p_{v}$ and $p_{v}\times q_{v}$
respectively. \ We embed $A$ and $B$ in $%
\mathbb{C}
G$ by linear embedding maps $\mathfrak{a}$ and $\mathfrak{b}$ defined by%
\begin{eqnarray*}
\mathfrak{a}\left( A\right) &=&\overline{A}=\underset{v=1}{\overset{r}{\sum }%
}\underset{s_{v}\text{ }\epsilon \text{ }S_{v},\text{ }t_{v}\text{ }\epsilon 
\text{ }T_{v}}{\sum }A_{s_{v},t_{v}}s_{v}^{-1}t_{v}=\underset{v=1}{\overset{r%
}{\sum }}\overline{A}_{v} \\
\mathfrak{b}\left( B\right) &=&\overline{B}=\underset{v=1}{\overset{r}{\sum }%
}\underset{t_{v}^{^{\prime }}\text{ }\epsilon \text{ }T_{v},\text{ }u_{v}%
\text{ }\epsilon \text{ }U_{v}}{\sum }B_{t_{v}^{^{\prime
}},u_{v}}t_{v}^{^{\prime }-1}u_{v}=\underset{v=1}{\overset{r}{\sum }}%
\overline{B}_{v}.
\end{eqnarray*}

Clearly, $\mathfrak{a}=\underset{v=1}{\overset{r}{\sum }}\mathfrak{a}_{v}$
and $\mathfrak{b}=\underset{v=1}{\overset{r}{\sum }}\mathfrak{b}_{v}$, and
are injective by the injectivity of the $\mathfrak{a}_{v}$ and $\mathfrak{b}%
_{v}$. \ The product of $\overline{A}$ and $\overline{B}$ in $%
\mathbb{C}
G$ is given by:\newline
\begin{eqnarray*}
\overline{A}\overline{B} &=&\underset{v=1}{\overset{r}{\sum }}\overline{A}%
_{v}\underset{v=1}{\overset{r}{\sum }}\overline{B}_{v} \\
&=&\underset{v=1}{\overset{r}{\sum }}\underset{w=1}{\overset{r}{\sum }}%
\underset{s_{v}\text{ }\epsilon \text{ }S_{v},\text{ }t_{v}\text{ }\epsilon 
\text{ }T_{v}\text{ }}{\sum }\underset{t_{w}^{^{\prime }}\text{ }\epsilon 
\text{ }T_{w},\text{ }u_{w}\text{ }\epsilon \text{ }U_{w}}{\sum }%
A_{s_{v},t_{v}}B_{t_{w}^{^{\prime }},u_{w}}s_{v}^{-1}t_{v}t_{w}^{^{\prime
}-1}u_{w} \\
&=&\underset{v=1}{\overset{r}{\sum }}\underset{w=1}{\overset{r}{\sum }}%
\left( \underset{s_{v}\text{ }\epsilon \text{ }S_{v},\text{ }t_{v}\text{ }%
\epsilon \text{ }T_{v}\text{ }}{\sum }\underset{t_{w}^{^{\prime }}\text{ }%
\epsilon \text{ }T_{w},\text{ }u_{w}\text{ }\epsilon \text{ }U_{w}}{\sum }%
A_{s_{v},t_{v}}B_{t_{w}^{^{\prime }},u_{w}}t_{v}t_{w}^{^{\prime }-1}\right)
s_{v}^{-1}u_{w}.
\end{eqnarray*}

If we use a third index $1\leq l\leq r$, then by the simultaneous triple
product property, for arbitrary $s_{l}^{^{\prime }}$ $\epsilon $ $S_{l}$, $%
u_{l}^{^{\prime }}$ $\epsilon $ $U_{l}$, it is the case that $%
s_{l}^{^{\prime }-1}u_{l}^{^{\prime }}=s_{v}^{-1}t_{v}t_{w}^{^{\prime
}-1}u_{w}\Longleftrightarrow s_{l}^{^{\prime
}}s_{v}^{-1}=t_{v}t_{w}^{^{\prime }-1}=u_{w}u_{l}^{^{\prime
}-1}=1\Longleftrightarrow l=v=w$. \ This means that for each $v$, and $s_{v}$
$\epsilon $ $S_{v}$ and $u_{v}$ $\epsilon $ $U_{v}$, the coefficient of the
term $s_{v}^{-1}u_{v}$ in the product $\overline{A}\overline{B}$ is $%
\underset{t_{v}\text{ }\epsilon \text{ }T_{v}}{\sum }%
A_{s_{v},t_{v}}B_{t_{v},u_{v}}=\left( A_{v}B_{v}\right) _{s_{v},u_{v}}$. \
In this way, we can recover the $r$ block products $%
A_{1}B_{1},A_{2}B_{2},....,A_{r}B_{r}$ \textit{simultaneously} from $%
\overline{A}\overline{B}$, and each block $A_{v}B_{v}$ will be the $v^{th}$
diagonal block on the block diagonal product $AB$. \ If we define an
extraction map $\mathfrak{x}:%
\mathbb{C}
^{m\times q}\longleftarrow 
\mathbb{C}
G$ \ based on this rule, then $\mathfrak{x}=\underset{v=1}{\overset{r}{%
\oplus }}\mathfrak{x}_{v}$ where $%
\mathbb{C}
^{m\times q}=\underset{v=1}{\overset{r}{\oplus }}%
\mathbb{C}
^{m_{v}\times q_{v}}$, and we have shown the following restriction $%
\mathfrak{x}\circ \mathfrak{m}_{%
\mathbb{C}
G}\circ \left( \mathfrak{a}\times \mathfrak{b}\right) \left( A,B\right) =AB$ 
$\epsilon $ $%
\mathbb{C}
^{m\times q}$ of $\mathfrak{m}_{%
\mathbb{C}
G}$ to $\left\langle m,p,q\right\rangle $, i.e. $\left\langle
m,p,q\right\rangle \leq _{%
\mathbb{C}
}\mathfrak{m}_{%
\mathbb{C}
G}$, by which we deduce $\mathfrak{R}\left( \left\langle m,p,q\right\rangle
\right) \leq \mathfrak{R}\left( \mathfrak{m}_{%
\mathbb{C}
G}\right) $ (\textit{Proposition 2.2}). \ By \textit{Proposition 2.5}, $%
\left\langle m,p,q\right\rangle =\left\langle \underset{v=1}{\overset{r}{%
\sum }}m_{v},\underset{v=1}{\overset{r}{\sum }}p_{v},\underset{v=1}{\overset{%
r}{\sum }}q_{v}\right\rangle \cong \underset{v=1}{\overset{r}{\oplus }}%
\left\langle m_{v},p_{v},q_{v}\right\rangle $, and by \textit{Proposition 2.2%
}%
\begin{equation*}
\mathfrak{R}\left( \underset{v=1}{\overset{r}{\oplus }}\left\langle
m_{v},p_{v},q_{v}\right\rangle \right) \leq \text{ }\mathfrak{R}\left( 
\mathfrak{m}_{%
\mathbb{C}
G}\right) .
\end{equation*}%
This takes care of $(1)$. \ For $(2)$ we note that note that $\mathfrak{R}%
\left( \mathfrak{m}_{%
\mathbb{C}
G}\right) =\left\vert G\right\vert $ iff $G$ is Abelian.
\end{proof}

\bigskip

An immediate consequence is the following.

\bigskip

\begin{corollary}
If $\left\{ \left\langle m_{i},p_{i},q_{i}\right\rangle \right\} _{i\text{ }%
\epsilon \text{ }I}\subseteq \mathfrak{S}(G)$ is a collection of tensors
simultaneously realized by a group $G$ then 
\begin{equation*}
(1)\underset{i\text{ }\epsilon \text{ }I}{\sum }\left(
m_{i}p_{i}q_{i}\right) ^{\frac{\omega }{3}}\leq D_{\omega }(G)
\end{equation*}%
and if $G$ is Abelian 
\begin{equation*}
(2)\text{ }\underset{i\text{ }\epsilon \text{ }I}{\sum }\left(
m_{i}p_{i}q_{i}\right) ^{\frac{\omega }{3}}\leq \left\vert G\right\vert .
\end{equation*}
\end{corollary}

\begin{proof}
Assume that $\left\{ \left\langle m_{i},p_{i},q_{i}\right\rangle \right\} _{i%
\text{ }\epsilon \text{ }I}\subseteq \mathfrak{S}(G)$ is a collection of
tensors simultaneously realized by $G$. \ $(1)$ By part $(1)$ of \textit{%
Theorem 5.1}, \textit{(3.18)} and \textit{Proposition 2.14}%
\begin{equation*}
\underset{i\text{ }\epsilon \text{ }I}{\sum }\left( m_{i}p_{i}q_{i}\right) ^{%
\frac{\omega }{3}}\leq \mathfrak{R}\left( \underset{i\text{ }\epsilon \text{ 
}I}{\oplus }\left\langle m_{i},p_{i},q_{i}\right\rangle \right) \leq 
\mathfrak{R}\left( \mathfrak{m}_{%
\mathbb{C}
G}\right) \leq \underset{\varrho \text{ }\epsilon \text{ }Irrep(G)}{\sum }%
\mathfrak{R}\left( \left\langle d_{\varrho },d_{\varrho },d_{\varrho
}\right\rangle \right) .
\end{equation*}%
$(1)$ On the right-hand side, by \textit{Proposition 3.13 }$D_{\omega
}\left( G\right) =\underset{\varrho \text{ }\epsilon \text{ }Irrep(G)}{\sum }%
d_{\varrho }^{\omega }\leq \underset{\varrho \text{ }\epsilon \text{ }%
Irrep(G)}{\sum }\mathfrak{R}\left( \left\langle d_{\varrho },d_{\varrho
},d_{\varrho }\right\rangle \right) $. \ Since $\left\langle
m_{i},p_{i},q_{i}\right\rangle $ $\epsilon $ $\mathfrak{S}(G)$, by \textit{%
Corollary 4.16}, it follows that $\left( m_{i}p_{i}q_{i}\right) ^{\frac{%
\omega }{3}}<\left\vert G\right\vert ^{\frac{\omega }{2}}=D_{2}\left(
G\right) ^{\frac{\omega }{2}}\leq D_{\omega }\left( G\right) ^{\frac{\omega 
}{2}}$. \ Thus, 
\begin{eqnarray*}
\underset{i\text{ }\epsilon \text{ }I}{\sum }\left( m_{i}p_{i}q_{i}\right) ^{%
\frac{\omega }{3}} &<&\left\vert I\right\vert D_{\omega }\left( G\right) ^{%
\frac{\omega }{2}} \\
&\Longleftrightarrow & \\
\underset{i\text{ }\epsilon \text{ }I}{\sum }\left( m_{i}p_{i}q_{i}\right) ^{%
\frac{\omega }{3}} &\leq &\left\vert I\right\vert ^{-\frac{2}{\omega }%
}\left( \underset{i\text{ }\epsilon \text{ }I}{\sum }\left(
m_{i}p_{i}q_{i}\right) ^{\frac{\omega }{3}}\right) ^{\frac{2}{\omega }%
}<D_{\omega }\left( G\right)
\end{eqnarray*}%
$(2)$ $D_{r}(G)=\left\vert G\right\vert $ for all $r\geq 1$ iff $G$ is
Abelian, and the result follows by $(1)$.
\end{proof}

\bigskip

Part $(2)$ of \textit{Corollary 5.2 }points to the usefulness of Abelian
groups for estimates of $\omega $, for which we have the following useful
corollary.

\bigskip

\begin{corollary}
If $\left\{ \left\langle n,n,n\right\rangle \right\} _{i=1}^{r}$ is a
collection of $r$ identical square tensors $\left\langle n,n,n\right\rangle $
simultaneously realized by an Abelian group $G$ then 
\begin{equation*}
(1)\text{ }\omega \leq \frac{\log \left\vert G\right\vert -\log r}{\log n}
\end{equation*}%
and%
\begin{equation*}
(2)\text{ }\omega =2\text{ if }\left\vert G\right\vert =n^{3}\text{ and }r=n.
\end{equation*}
\end{corollary}

\begin{proof}
Consequence of part $(2)$ of \textit{Corollary 5.2}.
\end{proof}

\subsection{\textit{Extension Results}}

\bigskip

The following is a basic extension of the simultaneous triple product
property to direct product groups.

\bigskip

\begin{lemma}
If groups $G$ and $G^{^{\prime }}$ have collections of simultaneous index
triples $\left\{ (S_{i},T_{i},U_{i})\right\} _{i\text{ }\epsilon \text{ }I},$
$\left\{ (S_{i^{\prime }}^{^{\prime }},T_{i^{\prime }}^{^{\prime
}},U_{i^{\prime }}^{^{\prime }})\right\} _{i^{\prime }\text{ }\epsilon \text{
}I^{^{\prime }}}$ of sizes $r$ and $r^{^{\prime }}$ resp., then their direct
product $G\times G^{^{\prime }}$ has the collection of $rr^{^{\prime }}$
simultaneous index triples $\left\{ (S_{i}\times S_{i^{\prime }}^{^{\prime
}},T_{i}\times T_{i^{\prime }}^{^{\prime }},U_{i}\times U_{i^{^{\prime
}}})\right\} _{i\text{ }\epsilon \text{ }I,\text{ }i^{\prime }\epsilon \text{
}I^{\prime }}$.
\end{lemma}

\begin{proof}
For arbitrary indices $i,j,k$ $\epsilon $ $I$ and $i^{^{\prime
}},j^{^{\prime }},k^{^{\prime }}$ $\epsilon $ $I^{^{\prime }}$, and elements 
$\left( s_{i},s_{i^{\prime }}^{^{\prime }}\right) $ $\epsilon $ $S_{i}\times
S_{i^{\prime }}^{^{\prime }}$, $\left( \overline{s}_{j},\overline{s}%
_{j^{\prime }}^{^{\prime }}\right) $ $\epsilon $ $S_{j}\times S_{j^{\prime
}}^{^{\prime }}$, $\left( t_{j},t_{j^{\prime }}^{^{\prime }}\right) $ $%
\epsilon $ $T_{j}\times T_{j^{\prime }}^{^{\prime }}$, $\left( \overline{t}%
_{k},\overline{t}_{k^{\prime }}^{^{\prime }}\right) $ $\epsilon $ $%
T_{k}\times T_{k^{\prime }}^{^{\prime }}$, $\left( u_{k},u_{k^{\prime
}}^{^{\prime }}\right) $ $\epsilon $ $U_{k}\times U_{k^{\prime }}^{^{\prime
}}$, $\left( \overline{u}_{i},\overline{u}_{i^{\prime }}^{^{\prime }}\right) 
$ $\epsilon $ $U_{i}\times U_{i^{\prime }}^{^{\prime }}$, and the assumption
of the simultaneous triple product property for both the collections $%
\left\{ (S_{i},T_{i},U_{i})\right\} _{i\text{ }\epsilon \text{ }I}$ and $%
\left\{ (S_{i^{\prime }}^{^{\prime }},T_{i^{\prime }}^{^{\prime
}},U_{i^{\prime }}^{^{\prime }})\right\} _{i^{\prime }\text{ }\epsilon \text{
}I^{^{\prime }}}$ it is the case that: 
\begin{eqnarray*}
&&\left( s_{i},s_{i^{\prime }}^{^{\prime }}\right) \left( \overline{s}_{j},%
\overline{s}_{j^{\prime }}^{^{\prime }}\right) ^{-1}\left(
t_{j},t_{j^{\prime }}^{^{\prime }}\right) \left( \overline{t}_{k},\overline{t%
}_{k^{\prime }}^{^{\prime }}\right) ^{-1}\left( u_{k},u_{k^{\prime
}}^{^{\prime }}\right) \left( \overline{u}_{i},\overline{u}_{i^{\prime
}}^{^{\prime }}\right) ^{-1} \\
&=&\left( s_{i}\overline{s}_{j}^{-1},s_{i^{\prime }}^{^{\prime }}\overline{s}%
_{j^{\prime }}^{^{\prime }-1}\right) \left( t_{j}\overline{t}%
_{k}^{-1},t_{j^{\prime }}^{^{\prime }}\overline{t}_{k^{\prime }}^{^{\prime
}-1}\right) \left( u_{k}\overline{u}_{i}^{-1},u_{k^{\prime }}^{^{\prime }}%
\overline{u}_{i^{\prime }}^{^{\prime -1}}\right) \\
&=&\left( s_{i}\overline{s}_{j}^{-1}t_{j}\overline{t}_{k}^{-1}u_{k}\overline{%
u}_{i}^{-1},s_{i^{\prime }}^{^{\prime }}\overline{s}_{j^{\prime }}^{^{\prime
-1}}t_{j^{\prime }}^{^{\prime }}\overline{t}_{k^{\prime }}^{^{\prime
-1}}u_{k^{\prime }}^{^{\prime }}\overline{u}_{i^{\prime ^{\prime
}}}^{^{\prime -1}}\right) =\left( 1_{G},1_{G^{\prime }}\right) \\
&\Longleftrightarrow &s_{i}\overline{s}_{j}^{-1}t_{j}\overline{t}%
_{k}^{-1}u_{k}\overline{u}_{i}^{-1}=1_{G},\text{ }s_{i^{\prime }}^{^{\prime
}}\overline{s}_{j^{\prime }}^{^{\prime -1}}t_{j^{\prime }}^{^{\prime }}%
\overline{t}_{k^{\prime }}^{^{\prime -1}}u_{k^{\prime }}^{^{\prime }}%
\overline{u}_{i^{\prime ^{\prime }}}^{^{\prime -1}}=1_{G^{\prime }} \\
&\Longrightarrow &s_{i}\overline{s}_{j}^{-1}t_{j}\overline{t}_{k}^{-1}u_{k}%
\overline{u}_{i}^{-1}=1_{G},\text{ }s_{i^{\prime }}^{^{\prime }}\overline{s}%
_{j^{\prime }}^{^{\prime -1}}t_{j^{\prime }}^{^{\prime }}\overline{t}%
_{k^{\prime }}^{^{\prime -1}}u_{k^{\prime }}^{^{\prime }}\overline{u}%
_{i^{\prime ^{\prime }}}^{^{\prime -1}}=1_{G^{\prime }}\text{ and} \\
i &=&j,\text{ }j=k,\text{ }k=i,\text{ }i^{^{\prime }}=j^{^{\prime }},\text{ }%
j^{^{\prime }}=k^{^{\prime }},\text{ }k^{^{\prime }}=i^{^{\prime }}.
\end{eqnarray*}
\end{proof}

\bigskip

This has an equivalent statement in terms of tensors.

\bigskip

\begin{corollary}
If groups $G$ and $G^{^{\prime }}$ have collections of simultaneously
realized tensors $\left\{ \left\langle m_{i},p_{i},q_{i}\right\rangle
\right\} _{i\text{ }\epsilon \text{ }I}$ and $\left\{ \left\langle
m_{i^{\prime }}^{^{\prime }},p_{i^{\prime }},q_{i^{\prime }}^{^{\prime
}}\right\rangle \right\} _{i^{^{\prime }}\epsilon \text{ }I^{^{\prime }}}$
of sizes $r$ and $r^{^{\prime }}$ resp., then their direct product $G\times
G^{^{\prime }}$ has the collection of $rr^{^{\prime }}$ simultaneously
realized pointwise product tensors $\left\{ \left\langle m_{i}m_{i^{\prime
}}^{^{\prime }},p_{i}p_{i^{\prime }}^{^{\prime }},q_{i}q_{i^{\prime
}}^{^{\prime }}\right\rangle \right\} _{i\text{ }\epsilon \text{ }I,\text{ }%
i^{\prime }\epsilon \text{ }I^{^{\prime }}}$.
\end{corollary}

\bigskip

\textit{Lemma 5.4} and \textit{Corollary 5.5 }are also independent
consequences of \textit{Lemma 4.8} and \textit{Lemma 4.12 }using the
simultaneous triple product property \textit{(5.2)}.

\bigskip

\section{\protect\Large Some Useful Groups}

\bigskip

Here we describe some special types of finite groups of particular interest
to our problem.

\bigskip

\subsection{\textit{The Triangle Set }$\Delta _{n}$ \textit{and the
Symmetric Group }$Sym_{n(n+1)/2}$}

\bigskip

For an arbitrary fixed $n\geq 1$, we define the \textit{triangle set }$%
\Delta _{n}$ by:

\bigskip

\begin{description}
\item[\textit{(5.3)}] $\Delta _{n}:=\left\{ x=\left(
x_{1},x_{2},x_{3}\right) \text{ }\epsilon \text{ }%
\mathbb{N}
^{3}\text{ }|\text{ }x_{1}+x_{2}+x_{3}=n-1\right\} $.
\end{description}

\bigskip

$\Delta _{n}$ is of size $\left\vert \Delta _{n}\right\vert =\underset{k=1}{%
\overset{n}{\sum }}k=1+2+\cdot \cdot \cdot \cdot +$ $n=n(n+1)/2$. \ The
following is a table for $\Delta _{5}$ ($n=5$) written lexicographically:

\begin{equation*}
\begin{tabular}{|l|l|l|l|l|}
\hline
& $\mathit{x}_{1}$ & $\mathit{x}_{2}$ & $\mathit{x}_{3}$ & $\mathit{x}_{1}%
\mathit{+x}_{2}\mathit{+x}_{3}$ \\ \hline
${\tiny 1.}$ & \multicolumn{1}{|c|}{$4$} & \multicolumn{1}{|c|}{$0$} & 
\multicolumn{1}{|c|}{$0$} & \multicolumn{1}{|c|}{$4$} \\ \hline
${\tiny 2.}$ & \multicolumn{1}{|c|}{$3$} & \multicolumn{1}{|c|}{$1$} & 
\multicolumn{1}{|c|}{$0$} & \multicolumn{1}{|c|}{$4$} \\ \hline
${\tiny 3.}$ & \multicolumn{1}{|c|}{$3$} & \multicolumn{1}{|c|}{$0$} & 
\multicolumn{1}{|c|}{$1$} & \multicolumn{1}{|c|}{$4$} \\ \hline
${\tiny 4.}$ & \multicolumn{1}{|c|}{$2$} & \multicolumn{1}{|c|}{$2$} & 
\multicolumn{1}{|c|}{$0$} & \multicolumn{1}{|c|}{$4$} \\ \hline
${\tiny 5.}$ & \multicolumn{1}{|c|}{$2$} & \multicolumn{1}{|c|}{$1$} & 
\multicolumn{1}{|c|}{$1$} & \multicolumn{1}{|c|}{$4$} \\ \hline
\end{tabular}%
\begin{tabular}{|l|l|l|l|l|}
\hline
& $\mathit{x}_{1}$ & $\mathit{x}_{2}$ & $\mathit{x}_{3}$ & $\mathit{x}_{1}%
\mathit{+x}_{2}\mathit{+x}_{3}$ \\ \hline
${\tiny 6.}$ & \multicolumn{1}{|c|}{$2$} & \multicolumn{1}{|c|}{$0$} & 
\multicolumn{1}{|c|}{$2$} & \multicolumn{1}{|c|}{$4$} \\ \hline
${\tiny 7.}$ & \multicolumn{1}{|c|}{$1$} & \multicolumn{1}{|c|}{$3$} & 
\multicolumn{1}{|c|}{$0$} & \multicolumn{1}{|c|}{$4$} \\ \hline
${\tiny 8.}$ & \multicolumn{1}{|c|}{$1$} & \multicolumn{1}{|c|}{$2$} & 
\multicolumn{1}{|c|}{$1$} & \multicolumn{1}{|c|}{$4$} \\ \hline
${\tiny 9.}$ & \multicolumn{1}{|c|}{$1$} & \multicolumn{1}{|c|}{$1$} & 
\multicolumn{1}{|c|}{$2$} & \multicolumn{1}{|c|}{$4$} \\ \hline
${\tiny 10.}$ & \multicolumn{1}{|c|}{$1$} & \multicolumn{1}{|c|}{$0$} & 
\multicolumn{1}{|c|}{$3$} & \multicolumn{1}{|c|}{$4$} \\ \hline
\end{tabular}%
\begin{tabular}{|l|c|c|c|c|}
\hline
& $\mathit{x}_{1}$ & $\mathit{x}_{2}$ & $\mathit{x}_{3}$ & $\mathit{x}_{1}%
\mathit{+x}_{2}\mathit{+x}_{3}$ \\ \hline
${\tiny 11.}$ & \multicolumn{1}{|c|}{$0$} & \multicolumn{1}{|c|}{$4$} & 
\multicolumn{1}{|c|}{$0$} & $4$ \\ \hline
${\tiny 12.}$ & \multicolumn{1}{|c|}{$0$} & \multicolumn{1}{|c|}{$3$} & 
\multicolumn{1}{|c|}{$1$} & $4$ \\ \hline
${\tiny 13.}$ & \multicolumn{1}{|c|}{$0$} & \multicolumn{1}{|c|}{$2$} & 
\multicolumn{1}{|c|}{$2$} & $4$ \\ \hline
${\tiny 14.}$ & \multicolumn{1}{|c|}{$0$} & \multicolumn{1}{|c|}{$1$} & 
\multicolumn{1}{|c|}{$3$} & $4$ \\ \hline
${\tiny 15.}$ & \multicolumn{1}{|c|}{$0$} & \multicolumn{1}{|c|}{$0$} & 
\multicolumn{1}{|c|}{$4$} & $4$ \\ \hline
\end{tabular}%
\end{equation*}

\bigskip

The smallest triple in $\Delta _{n}$ is $(0,0,n-1)$ and the largest $%
(n-1,0,0)$. \ The $i^{th}$ component $x_{i}$ of any triple $%
(x_{1},x_{2},x_{3})$ $\epsilon $ $\Delta _{n}$ can take any one of $n$
values $0,1,2,....,.n-1$, and for any value $0\leq k\leq n-1$ and component
index $1\leq i\leq 3$, there are exactly $n-k$ triples $(x_{1},x_{2},x_{3})$ 
$\epsilon $ $\Delta _{n}$ with the $i^{th}$ component $x_{i}=k$. \ This is a
representation of \ $\Delta _{5}$ as a triangular array or pyramid of dot
elements:

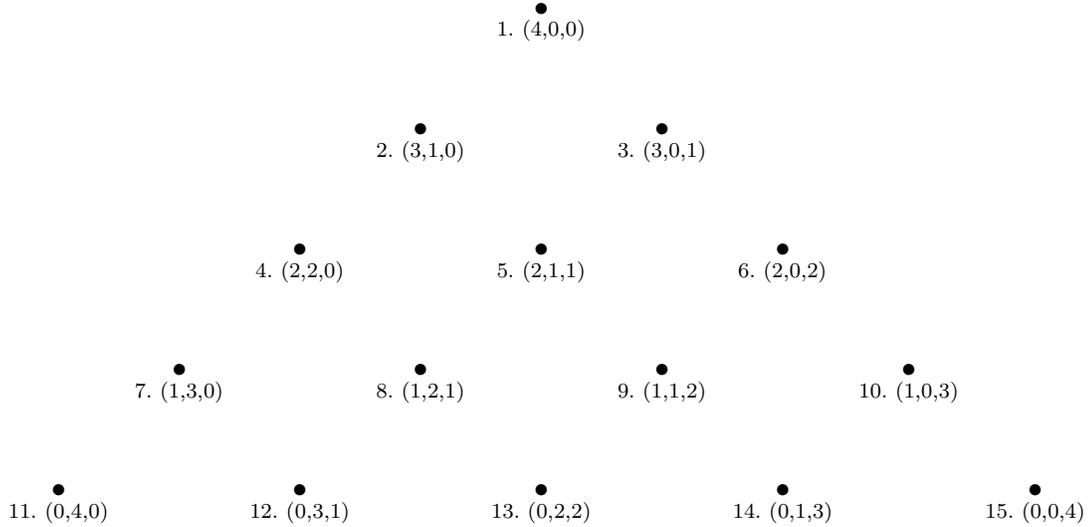
\begin{equation*}
\begin{array}{ccccccccc}
&  &  &  & \underset{1.\text{ }\left( 4,0,0\right) }{\bullet } &  &  &  & 
\\ 
&  &  &  &  &  &  &  &  \\ 
&  &  & \underset{2.\text{ }\left( 3,1,0\right) }{\bullet } &  & \underset{3.%
\text{ }\left( 3,0,1\right) }{\bullet } &  &  &  \\ 
&  &  &  &  &  &  &  &  \\ 
&  & \underset{4.\text{ }\left( 2,2,0\right) }{\bullet } &  & \underset{5.%
\text{ }\left( 2,1,1\right) }{\bullet } &  & \underset{6.\text{ }\left(
2,0,2\right) }{\bullet } &  &  \\ 
&  &  &  &  &  &  &  &  \\ 
& \underset{7.\text{ }\left( 1,3,0\right) }{\bullet } &  & \underset{8.\text{
}\left( 1,2,1\right) }{\bullet } &  & \underset{9.\text{ }\left(
1,1,2\right) }{\bullet } &  & \underset{10.\text{ }\left( 1,0,3\right) }{%
\bullet } &  \\ 
&  &  &  &  &  &  &  &  \\ 
\underset{11.\text{ }\left( 0,4,0\right) }{\bullet } &  & \underset{12.\text{
}\left( 0,3,1\right) }{\bullet } &  & \underset{13.\text{ }\left(
0,2,2\right) }{\bullet } &  & \underset{14.\text{ }\left( 0,1,3\right) }{%
\bullet } &  & \underset{15.\text{ }\left( 0,0,4\right) }{\bullet }%
\end{array}%
\end{equation*}

\bigskip

Counting the rows of this pyramid from the lowest, for $0\leq k\leq n-1=4$,
the $k^{th}$ row of dots correspond to the subset of triples in $\Delta _{5}$
with $1^{st}$ component $x_{1}=k$, and in the $k^{th}$ row each element is
ordered component-wise descending order from left to right. \ This is
precisely the lexicographic (dictionary) ordering of the elements of $\Delta
_{n}$. \ $Sym_{n(n+1)/2}$ can be understood as the permutation group of $%
\Delta _{n}$, and we may write $Sym_{n(n+1)/2}\equiv Sym(\Delta _{n})$. \ $%
Sym(\Delta _{n})$ is of order $\left( \frac{1}{2}n(n+1)\right) !$. \
Elements $\mu $ $\epsilon $ $Sym(\Delta _{n})$ are bijective maps $\Delta
_{n}\cong \Delta _{n}$ and their actions on the components of triples $%
x=(x_{1},x_{2},x_{3})$ $\epsilon $ $\Delta _{n}$ is defined by $%
x_{i}\longmapsto \mu (x)_{i}$, $1\leq i\leq 3$, where $\mu (x)_{i}$ denotes
the $i^{th}$ component of the permuted triple $\mu (x)$. \ Graphically, the
permutations $\mu $ $\epsilon $ $Sym(\Delta _{n})$ are bijective
transformations, such as rotations or reflections, of $\Delta _{n}$ or of
any subset of points of $\Delta _{n}$.

\bigskip

We define the subsets $Sym_{i}(\Delta _{n})\subseteq Sym(\Delta _{n})$ by:

\bigskip

\begin{description}
\item[\textit{(5.4)}] $Sym_{i}(\Delta _{n}):=\left\{ \mu \text{ }\epsilon 
\text{ }Sym(\Delta _{n})\text{ }|\text{ }\mu (x)_{i}=x_{i},\text{ }x\text{ }%
\epsilon \text{ }\Delta _{n}\right\} $,\qquad \qquad \qquad \qquad $1\leq
i\leq 3$.
\end{description}

\bigskip

Each $Sym_{i}(\Delta _{n})$ forms a fixed-point subgroup of $Sym(\Delta
_{n}) $ consisting of those permutations of $\Delta _{n}$ leaving the $%
i^{th} $ components of triples $x$ $\epsilon $ $\Delta _{n}$ fixed. \ For
any value $0\leq k\leq n-1$ and index $1\leq i\leq 3$, there are exactly $%
n-k $ triples in $\Delta _{n}$ with $i^{th}$ component $x_{i}=k$, there are $%
(n-k)!$ permutations of these triples, and for each $1\leq i\leq 3$, there
are $n!(n-1)!\cdot \cdot \cdot \cdot 2!1!$ permutations of $\Delta _{n}$
which fix the $i^{th}$ components of triples $x$ $\epsilon $ $\Delta _{n}$,
and $\left\vert Sym_{i}(\Delta _{n})\right\vert =n!(n-1)!\cdot \cdot \cdot
\cdot 2!1!$. \ Graphically, the subgroups $Sym_{i}(\Delta _{n})\leq
Sym\left( \Delta _{n}\right) $ are collections of permutations $\mu _{i}$ $%
\epsilon $ $Sym\left( \Delta _{n}\right) $ transforming $\Delta _{n}$ solely
along the diagonal rows parallel to its $i^{th}$ side, one subgroup for each
side. \ The following is a diagram for $Sym_{1}(\Delta _{5}),Sym_{2}(\Delta
_{5}),Sym_{3}(\Delta _{5})$ on $\Delta _{5}$.

\bigskip

\begin{equation*}
\FRAME{itbpF}{5.3921in}{3.2379in}{0in}{}{}{Figure 5.1}{\special{language
"Scientific Word";type "GRAPHIC";display "USEDEF";valid_file "T";width
5.3921in;height 3.2379in;depth 0in;original-width 9.2915in;original-height
5.2191in;cropleft "0.0245";croptop "1";cropright "1.0009";cropbottom
"0";tempfilename 'JNN6AN01.wmf';tempfile-properties "XPR";}}
\end{equation*}

\bigskip

The diagram above makes it clear that these fixed point subgroups have the
triple product property.

\bigskip

\begin{lemma}
The subgroups $Sym_{1}(\Delta _{n}),$ $Sym_{2}(\Delta _{n}),$ $%
Sym_{3}(\Delta _{n})\leq Sym\left( \Delta _{n}\right) $, defined in (5.2),
form an index triple of $Sym\left( \Delta _{n}\right) $.
\end{lemma}

\begin{proof}
Since $Sym_{i}(\Delta _{n})\leq Sym(\Delta _{n})$, $1\leq i\leq 3$, to prove
the triple product property for these, it suffices to prove for arbitrary $%
\mu _{1}$ $\epsilon $ $Sym_{1}(\Delta _{n})$, $\mu _{2}$ $\epsilon $ $%
Sym_{2}(\Delta _{n})$, and $\mu _{3}$ $\epsilon $ $Sym_{3}(\Delta _{n})$
that $\mu _{1}\mu _{2}\mu _{3}=1$ implies that $\mu _{1}=\mu _{2}=\mu _{3}=1$%
. \ For a $\mu $ $\epsilon $ $Sym(\Delta _{n})$, we define its fixed point
set as $fix\left( \mu \right) =\left\{ x\text{ }\epsilon \text{ }\Delta _{n}%
\text{ }|\text{ }\mu (x)=x\right\} \subseteq \Delta _{n}$, and its $i^{th}$
component fixed point set $fix_{i}\left( \mu \right) $ as $fix_{i}\left( \mu
\right) =\left\{ x\text{ }\epsilon \text{ }\Delta _{n}\text{ }|\text{ }\mu
(x)_{i}=x_{i}\right\} \subseteq \Delta _{n}$. \ For arbitrary $\mu $ $%
\epsilon $ $Sym(\Delta _{n})$, the sets $fix\left( \mu \right) $ and $%
fix_{i}\left( \mu \right) $ are such that $\mu =1$ iff $fix\left( \mu
\right) \cap fix_{i}\left( \mu \right) =\Delta _{n}$, for all $1\leq i\leq 3$%
, where $1$ is the identity permutation of $\Delta _{n}$. \ Moreover, $%
fix\left( \mu \right) \subseteq fix_{i}\left( \mu \right) $. \ Then, from $%
\mu _{1}\mu _{2}\mu _{3}=1$ it follows that $fix\left( \mu _{1}\mu _{2}\mu
_{3}\right) \cap fix_{i}\left( \mu _{1}\mu _{2}\mu _{3}\right) =fix\left(
\mu _{1}\mu _{2}\mu _{3}\right) =\Delta _{n}$, $1\leq i\leq 3$. \ Since $%
fix_{1}\left( \mu _{1}\right) =fix_{2}\left( \mu _{2}\right) =fix_{3}\left(
\mu _{3}\right) =\Delta _{n}$, it follows that $fix\left( \mu _{1}\right)
=fix\left( \mu _{2}\right) =fix\left( \mu _{3}\right) =\Delta _{n}$. \
Together, $fix\left( \mu _{1}\mu _{2}\mu _{3}\right) =\Delta _{n}$ and $%
fix\left( \mu _{i}\right) =\Delta _{n}$, $1\leq i\leq 3$, implies that $%
fix_{i}\left( \mu _{1}\mu _{2}\right) =fix_{i}\left( \mu _{2}\mu _{3}\right)
=fix_{i}\left( \mu _{1}\mu _{3}\right) =\Delta _{n}$, $1\leq i\leq 3$, which
implies that $fix\left( \mu _{1}\mu _{2}\right) =fix\left( \mu _{2}\mu
_{3}\right) =fix\left( \mu _{1}\mu _{3}\right) =\Delta _{n}$, $1\leq i\leq 3$%
, which implies that $\mu _{1}\mu _{2}=\mu _{2}\mu _{3}=\mu _{1}\mu _{3}=1$,
from which we deduce that $\mu _{1}=\mu _{2}=\mu _{3}=1$.
\end{proof}

\bigskip

This shows that $Sym(\Delta _{n})$ realizes the tensor $\left\langle \overset%
{n}{\underset{k=1}{\tprod }}k!,\overset{n}{\underset{k=1}{\tprod }}k!,%
\overset{n}{\underset{k=1}{\tprod }}k!\right\rangle $, which means that it
supports square matrix multiplication of order $\overset{n}{\underset{k=1}{%
\tprod }}k!$, and, therefore, by \textit{(4.21)}, we have:

\begin{description}
\item[\textit{(5.5)}] $\alpha (Sym\left( \Delta _{n}\right) )\leq \frac{\log
\left( \frac{1}{2}n(n+1)\right) !}{\log \left( n!(n-1)!\cdot \cdot \cdot
\cdot 2!1!\right) }.$
\end{description}

\bigskip

This yields concrete estimates for $\alpha (Sym_{m})$, to be described in 
\textit{Chapter 6}.

\bigskip

\subsection{\textit{Semidirect Product and Wreath Product Groups}}

\bigskip

A group $G$ is said to be the (\textit{internal})\textit{\ semidirect product%
} $A\rtimes B$ of a subgroup $B\leq G$ by a normal subgroup $%
A\vartriangleleft G$, if $A\cap B=\{1_{G}\}$ and $G=AB$. \ Each elements $g$ 
$\epsilon $ $G=A\rtimes B$ has the form $g=ab$ for a unique element $a$ $%
\epsilon $ $A$ and a $b$ $\epsilon $ $B$ depending on $a$, and by $%
A\vartriangleleft G$, it is the case that $b^{-1}g=b^{-1}ab$ $\epsilon $ $A$%
. \ For a fixed $b$ $\epsilon $ $B$, the mapping $a\longmapsto
bab^{-1}\equiv a^{b}$, $a$ $\epsilon $ $A$, defines an automorphism $\alpha
_{b}$ of $A$ which is conjugation of $A$ by $b$, and the mapping $%
b\longmapsto \alpha _{b}$, $b$ $\epsilon $ $B$, is a group homomorphism $%
\alpha :B\longrightarrow Aut(A)$ defining the conjugation action of $B$ on $%
A $, such that multiplication of elements $g=(ab),g^{^{\prime
}}=(a^{^{\prime }}b^{^{\prime }})$ $\epsilon $ $G$ can be expressed as $%
gg^{^{\prime }}=(ab)(a^{^{\prime }}b^{^{\prime }})=(aba^{^{\prime
}}b^{-1}bb^{^{\prime }})=(aa^{^{\prime }b})(bb^{^{\prime }})=(a\alpha
_{b}(a^{^{\prime }}))(bb^{^{\prime }})$, and inverses of elements $g=ab$ are
given by $g^{-1}=\alpha _{b^{-1}}(a^{-1})b^{-1}$. \ If $B\vartriangleleft G$
also then $G=A\rtimes B=A\times B$. $\ $If $A$ and $B$ are arbitrary groups,
then for any homomorphism $\phi :B\longrightarrow Aut(A)$ there is a unique (%
\textit{external})\textit{\ semidirect product} $A\rtimes B$ of $B$ by $A$,
with underlying set $A\times B$, for which $\alpha _{b}=\phi (b)$ for any $b$
$\epsilon $ $B$, and $A\cong A\rtimes \left\{ 1_{B}\right\} \equiv \lbrack
A]\vartriangleleft A\rtimes B$ and $B\cong $ $\left\{ 1_{A}\right\} \rtimes
B\equiv \lbrack B]\leq A\rtimes B$ such that $[A]\cap \lbrack B]=\left\{
1_{G}\right\} $ and $G=[A][B]$. \ Every external semidirect product $%
A\rtimes B$ of groups $A$ and $B$ is the internal semidirect product $\left[
A\right] \rtimes \left[ B\right] $ of the subgroups $[A]\vartriangleleft
A\rtimes B,[B]\leq A\rtimes B$. \ If $\pi :G\longrightarrow GL(V)$ is any
nontrivial representation of $G$, and $\iota _{A}$ denotes the inclusion
homomorphism $A\longrightarrow G=A\rtimes B$, then $\pi _{A}\equiv \pi \circ
\iota _{A}$ is the representation $A\longrightarrow GL(V)$ of $A$
equidimensional with $\pi $ and a nontrivial subrepresentation of $\pi $. \
If $\pi _{A}$ is irreducible then $\pi $ must be irreducible, while any
irreducible representation $\varrho _{B}$ of $B$ extends to a unique
irreducible representation $\varrho $ of $G$. \ For an Abelian subgroup $%
A\vartriangleleft G$ and a subgroup $B\leq G$ there is a proper subgroup $%
C<B $ such that an irreducible representation of $A\vartriangleleft G$%
.extends to an irreducible representation of $G=A\rtimes B$.

\bigskip

A special kind of semidirect product group $G=A\rtimes B$ exists when $%
A=H^{n}$, the $n$-fold direct product of $H$, with $H$ being a group, and $%
B=Sym_{n}$, where multiplication in $H^{n}$ is component-wise multiplication
of $n$-tuples $h=\left( h_{i}\right) _{i=1}^{n}$ of elements of $H$, and
multiplication in $Sym_{n}$ is the composition of permutations $\mu $ of $n$
elements. \ This group $G=H^{n}\rtimes Sym_{n}$ is called the \textit{wreath
product} of $Sym_{n}$ by $H^{n}$, denoted by $H\wr Sym_{n}$, where the
action of $Sym_{n}$ on $H_{n}$ is from the right, defined by the mapping $%
h\longmapsto h^{\mu }:=\left( h_{\mu i}\right) _{i=1}^{n}$, for $n$-tuples $%
h=\left( h_{i}\right) _{i=1}^{n}$ $\epsilon $ $H^{n}$ and permutations $\mu $
$\epsilon $ $Sym_{n}$, i.e. , $Sym_{n}$ acts on $H_{n}$ by permuting the
components of its $n$-tuples. \ We sometimes write $\left( h\right) _{i}$
for the $i^{th}$ coordinate $h_{i}$ of an $h=\left( h_{i}\right) _{i=1}^{n}$ 
$\epsilon $ $H^{n}$. \ $H$ is called the base group of $H\wr Sym_{n}$, and
if we identity $H^{n}$ with the subgroup $\left\{ h1_{Sym_{n}}\text{ }|\text{
}h\text{ }\epsilon \text{ }H^{n}\right\} \leq H\wr Sym_{n}$, then $%
H^{n}\vartriangleleft $ $H\wr Sym_{n}$ and $H\wr Sym_{n}$ becomes an
internal semidirect product. \ Multiplication in $H\wr Sym_{n}$ is given by $%
(h\mu )(h^{^{\prime }}\mu ^{^{\prime }})=(hh^{^{\prime }\mu ^{-1}})(\mu \mu
^{^{\prime }})=(h_{i}h_{\mu ^{-1}i}^{^{\prime }})_{i=1}^{n}(\mu \mu
^{^{\prime }})$, and inverses $(h\mu )^{-1}$ by $h^{\mu }\mu ^{-1}=\left(
h_{\mu i}\right) _{i=1}^{n}\mu ^{-1}$, for elements $(h\mu ),(h^{^{\prime
}}\mu ^{^{\prime }})$ $\epsilon $ $H\wr Sym_{n}$. \ The following is an
elementary result about the sums of $\omega ^{th}$ powers of the irreducible
character degrees of $H\wr Sym_{n}$.

\bigskip

\begin{lemma}
For an Abelian group $H$, $D_{\omega }(H\wr Sym_{n})\leq \left( n!\right)
^{\omega -1}\left\vert H\right\vert ^{n}$.
\end{lemma}

\begin{proof}
For an Abelian group $H$, and the wreath product group $H\wr Sym_{n}$, every
irreducible representation $\varrho $ $\epsilon $ $Irrep(H\wr Sym_{n})$ is
induced from an irreducible representation of the base group $%
H^{n}\vartriangleleft H\wr Sym_{n}$, which has the index $\left[ H\wr
Sym_{n}:H^{n}\right] =\left\vert Sym_{n}\right\vert =n!$, [HUP1998]. \
Therefore, the index $[H\wr Sym_{n}:N]$ of any maximal Abelian normal
subgroup $N\vartriangleleft H\wr Sym_{n}$ is at most $n!$, and by \textit{%
Theorem 3.7}, $Dim$ $\varrho \leq n!$ for any $\varrho $ $\epsilon $ $%
Irrep(H\wr Sym_{n})$. \ Then:%
\begin{eqnarray*}
D_{\omega }(H\wr Sym_{n}) &=&\underset{\varrho \text{ }\epsilon \text{ }%
Irrep(H\wr Sym_{n})}{\sum }d_{\varrho }^{\omega } \\
&=&\underset{\varrho \text{ }\epsilon \text{ }Irrep(H\wr Sym_{n})}{\sum }%
d_{\varrho }^{\omega -2}d_{\varrho }^{2} \\
&\leq &\left( n!\right) ^{\omega -2}\underset{\varrho \text{ }\epsilon \text{
}Irrep(H\wr Sym_{n})}{\sum }d_{\varrho }^{2} \\
&=&\left( n!\right) ^{\omega -2}\left\vert H\wr Sym_{n}\right\vert \\
&=&\left( n!\right) ^{\omega -2}(n!)\left\vert H^{n}\right\vert \\
&=&\left( n!\right) ^{\omega -1}\left\vert H\right\vert ^{^{n}} \\
&=&\left( n!\right) ^{\omega -1}D_{2}(H)^{n} \\
&\leq &\left( n!\right) ^{\omega -1}D_{\omega }(H)^{n}.
\end{eqnarray*}%
Since $H$ is Abelian $D_{\omega }(H)=\left\vert H\right\vert $, and the
result follows.
\end{proof}

\bigskip

We conclude with some extension results.

\bigskip

\begin{theorem}
If $\left\{ \left( S_{i},T_{i},U_{i}\right) \right\} _{i=1}^{n}$ $\subset 
\mathfrak{J}\left( H\right) $ is a collection of $n$ simultaneous index
triples of a group $H$ then the triple $\left( \overset{n}{\underset{i=1}{%
\tprod }}S_{i}\wr Sym_{n},\overset{n}{\underset{i=1}{\tprod }}T_{i}\wr
Sym_{n},\overset{n}{\underset{i=1}{\tprod }}U_{i}\wr Sym_{n}\right) $ $%
\epsilon $ $\mathfrak{J}\left( H\wr Sym_{n}\right) $.
\end{theorem}

\begin{proof}
Assume a collection $\left\{ \left( S_{i},T_{i},U_{i}\right) \right\}
_{i=1}^{n}$ $\subset \mathfrak{J}\left( H\right) $ of index triples of $H$.
\ By \textit{Lemma 4.8} the $n$-fold direct product of these triples, $%
\left( \overset{n}{\underset{i=1}{\tprod }}S_{i}=S,\overset{n}{\underset{i=1}%
{\tprod }}T_{i}=T,\overset{n}{\underset{i=1}{\tprod }}U_{i}=U\right) $ is an
index triple of $H^{n}$. \ Assume further that the triples $\left(
S_{i},T_{i},U_{i}\right) $ have the simultaneous triple product property
(STPP). \ We claim that the subsets 
\begin{eqnarray*}
S\wr Sym_{n} &:&=\left\{ \left( s\sigma \right) \text{ }|\text{ }s\text{ }%
\epsilon \text{ }S,\text{ }\sigma \text{ }\epsilon \text{ }Sym_{n}\right\} ,
\\
T\wr Sym_{n} &:&=\left\{ \left( t\tau \right) \text{ }|\text{ }t\text{ }%
\epsilon \text{ }T,\text{ }\tau \text{ }\epsilon \text{ }Sym_{n}\right\} , \\
U\wr Sym_{n} &:&=\left\{ \left( u\upsilon \right) \text{ }|\text{ }u\text{ }%
\epsilon \text{ }U,\text{ }\upsilon \text{ }\epsilon \text{ }Sym_{n}\right\}
,
\end{eqnarray*}%
satisfy the triple product property (TPP) in $H\wr Sym_{n}$. \ To see this,
let $s_{1}\sigma _{1},$ $s_{1}^{^{\prime }}\sigma _{1}^{^{\prime }}$ $%
\epsilon $ $S\wr Sym_{n},$ $t_{2}\tau _{2},$ $t_{2}^{^{\prime }}\tau
_{2}^{^{\prime }}$ $\epsilon $ $T\wr Sym_{n},u_{3}\upsilon _{3},$ $%
u_{3}^{^{\prime }}\upsilon _{3}^{^{\prime }}$ $\epsilon $ $U\wr Sym_{n}$ be
arbitrary elements. \ Then 
\begin{eqnarray*}
&&\left( s_{1}^{^{\prime }}\sigma _{1}^{^{\prime }}\right) \left(
s_{1}\sigma _{1}\right) ^{-1}\left( t_{2}^{^{\prime }}\tau _{2}^{^{\prime
}}\right) \left( t_{2}\tau _{2}\right) ^{-1}\left( u_{3}^{^{\prime
}}\upsilon _{3}^{^{\prime }}\right) \left( u_{3}\upsilon _{3}\right) ^{-1} \\
&=&s_{1}^{^{\prime }}\sigma _{1}^{^{\prime }}s_{1}^{\sigma _{1}}\sigma
_{1}^{-1}t_{2}^{^{\prime }}\tau _{2}^{^{\prime }}\tau _{2}^{\tau _{2}}\tau
_{2}^{-1}u_{3}^{^{\prime }}\upsilon _{3}^{^{\prime }}u_{3}^{\upsilon
_{3}}\upsilon _{3}^{-1} \\
&=&s_{1}^{^{\prime }}s_{1}^{\sigma _{1}^{^{\prime }}\sigma _{1}^{-1}}\sigma
_{1}^{^{\prime }}\sigma _{1}^{-1}t_{2}^{^{\prime }}t_{2}^{\tau
_{2}^{^{\prime }}\tau _{2}^{-1}}\tau _{2}^{^{\prime }}\tau
_{2}^{-1}u_{3}^{^{\prime }}u_{3}^{\upsilon _{3}^{^{\prime }}\upsilon
_{3}^{-1}}\upsilon _{3}^{^{\prime }}\upsilon _{3}^{-1} \\
&=&1 \\
&\Longrightarrow &\sigma _{1}^{^{\prime }}\sigma _{1}^{-1}\tau
_{2}^{^{\prime }}\tau _{2}^{-1}\upsilon _{3}^{^{\prime }}\upsilon
_{3}^{-1}=1.
\end{eqnarray*}%
Putting $\mu =\sigma _{1}^{^{\prime }}\sigma _{1}^{-1}$ and $\nu =\sigma
_{1}^{^{\prime }}\sigma _{1}^{-1}\tau _{2}^{^{\prime }}\tau _{2}^{-1}$ we
have that%
\begin{eqnarray*}
&&s_{1}^{^{\prime }}s_{1}^{\sigma _{1}^{^{\prime }}\sigma _{1}^{-1}}\sigma
_{1}^{^{\prime }}\sigma _{1}^{-1}t_{2}^{^{\prime }}t_{2}^{\tau
_{2}^{^{\prime }}\tau _{2}^{-1}}\tau _{2}^{^{\prime }}\tau
_{2}^{-1}u_{3}^{^{\prime }}u_{3}^{\upsilon _{3}^{^{\prime }}\upsilon
_{3}^{-1}}\upsilon _{3}^{^{\prime }}\upsilon _{3}^{-1} \\
&=&1 \\
&\Longrightarrow &u_{3}^{-1}s_{1}^{^{\prime }}\left(
s_{1}^{-1}t_{2}^{^{\prime }}\right) ^{\mu }\left( t_{2}^{-1}u_{3}^{^{\prime
}}\right) ^{\nu }=1 \\
&\Longleftrightarrow &\left( u_{3}^{-1}\right) _{i}\left( s_{1}^{^{\prime
}}\right) _{i}\left( s_{1}^{-1}\right) _{\mu i}\left( t_{2}^{^{\prime
}}\right) _{\mu i}\left( t_{2}^{-1}\right) _{\nu i}\left( u_{3}^{^{\prime
}}\right) _{\nu i}=1 \\
&\Longleftrightarrow &\mu i=\nu i=i,1\leq i\leq n\text{ (STPP for }\left\{
\left( S_{i},T_{i},U_{i}\right) \right\} _{i=1}^{n}\text{)} \\
&\Longleftrightarrow &\mu =\nu =1 \\
&\Longleftrightarrow &\sigma _{1}=\sigma _{1}^{^{\prime }},\tau _{2}=\tau
_{2}^{^{\prime }},\upsilon _{3}=\upsilon _{3}^{^{\prime }}.
\end{eqnarray*}%
Thus 
\begin{eqnarray*}
&&s_{1}^{^{\prime }}s_{1}^{\sigma _{1}^{^{\prime }}\sigma _{1}^{-1}}\sigma
_{1}^{^{\prime }}\sigma _{1}^{-1}t_{1}^{^{\prime }}t_{1}^{\tau
_{1}^{^{\prime }}\tau _{1}^{-1}}\tau _{1}^{^{\prime }}\tau
_{1}^{-1}u_{1}^{^{\prime }}u_{1}^{\upsilon _{1}^{^{\prime }}\upsilon
_{1}^{-1}}\upsilon _{1}^{^{\prime }}\upsilon _{1}^{-1} \\
&=&1 \\
&\Longrightarrow &s_{1}^{^{\prime }}s_{1}^{-1}t_{2}^{^{\prime
}}t_{2}^{-1}u_{3}^{^{\prime }}u_{3}^{-1}=1 \\
&\Longleftrightarrow &s_{1}^{^{\prime }}=s_{1},t_{2}^{^{\prime
}}=t_{2},u_{3}^{^{\prime }}=u_{3}\text{ (TPP for }\left\{ \left(
S_{i},T_{i},U_{i}\right) \right\} _{i=1}^{n}\text{)}
\end{eqnarray*}%
Putting these two together we deduce that for $s_{1}\sigma _{1},$ $%
s_{1}^{^{\prime }}\sigma _{1}^{^{\prime }}$ $\epsilon $ $S\wr Sym_{n},$ $%
t_{2}\tau _{2},$ $t_{2}^{^{\prime }}\tau _{2}^{^{\prime }}$ $\epsilon $ $%
T\wr Sym_{n},u_{3}\upsilon _{3},$ $u_{3}^{^{\prime }}\upsilon _{3}^{^{\prime
}}$ $\epsilon $ $U\wr Sym_{n}$ it is the case that 
\begin{eqnarray*}
&&\left( s_{1}^{^{\prime }}\sigma _{1}^{^{\prime }}\right) \left(
s_{1}\sigma _{1}\right) ^{-1}\left( t_{2}^{^{\prime }}\tau _{2}^{^{\prime
}}\right) \left( t_{2}\tau _{2}\right) ^{-1}\left( u_{3}^{^{\prime
}}\upsilon _{3}^{^{\prime }}\right) \left( u_{3}\upsilon _{3}\right) ^{-1} \\
&=&s_{1}^{^{\prime }}s_{1}^{\sigma _{1}^{^{\prime }}\sigma _{1}^{-1}}\sigma
_{1}^{^{\prime }}\sigma _{1}^{-1}t_{1}^{^{\prime }}t_{1}^{\tau
_{1}^{^{\prime }}\tau _{1}^{-1}}\tau _{1}^{^{\prime }}\tau
_{1}^{-1}u_{1}^{^{\prime }}u_{1}^{\upsilon _{1}^{^{\prime }}\upsilon
_{1}^{-1}}\upsilon _{1}^{^{\prime }}\upsilon _{1}^{-1} \\
&=&1 \\
&\Longleftrightarrow &s_{1}^{^{\prime }}\sigma _{1}^{^{\prime }}=s_{1}\sigma
_{1},t_{2}^{^{\prime }}\tau _{2}^{^{\prime }}=t_{2}\tau _{2},u_{3}^{^{\prime
}}\upsilon _{3}^{^{\prime }}=u_{3}\upsilon _{3}.
\end{eqnarray*}%
This proves our claim.
\end{proof}

\bigskip

\begin{corollary}
If $\left\{ \left\langle m_{i},p_{i},q_{i}\right\rangle \right\} _{i=1}^{n}$ 
$\subset \mathfrak{S}\left( H\right) $ is a collection of $n$ tensors
simultaneously realized by an Abelian group $H$ then 
\begin{equation*}
(1)\text{ }\left\langle n!\overset{n}{\underset{i=1}{\tprod }}m_{i},n!%
\overset{n}{\underset{i=1}{\tprod }}p_{i},n!\overset{n}{\underset{i=1}{%
\tprod }}q_{i}\right\rangle \text{ }\epsilon \text{ }\mathfrak{S}\left( H\wr
Sym_{n}\right)
\end{equation*}
and 
\begin{equation*}
(2)\text{ }\omega \leq \frac{n\log \left\vert H\right\vert -\log n!}{\log 
\sqrt[3]{\overset{n}{\underset{i=1}{\tprod }}m_{i}p_{i}q_{i}}}.
\end{equation*}
\end{corollary}

\begin{proof}
If $n$ triples $S_{i},T_{i,}U_{i}\subseteq H,$ $1\leq i\leq n$, of sizes $%
\left\vert S_{i}\right\vert =m_{i},$ $\left\vert T_{i,}\right\vert =p_{i},$ $%
\left\vert U_{i}\right\vert =q_{i},$ $1\leq i\leq n$, satisfy the STPP in an
Abelian group $H$ then by \textit{Theorem 5.8 }$H\wr Sym_{n}$ realizes the
product tensor $\left\langle n!\overset{n}{\underset{i=1}{\tprod }}m_{i},n!%
\overset{n}{\underset{i=1}{\tprod }}p_{i},n!\overset{n}{\underset{i=1}{%
\tprod }}q_{i}\right\rangle $ once, and in addition, by \textit{Corollary
5.2 }and \textit{Lemma 5.7}, 
\begin{eqnarray*}
&&\left( n!\overset{n}{\underset{i=1}{\tprod }}m_{i}\cdot n!\overset{n}{%
\underset{i=1}{\tprod }}p_{i}\cdot n!\overset{n}{\underset{i=1}{\tprod }}%
q_{i}\right) ^{\frac{\omega }{3}} \\
&=&\left( n!\sqrt[3]{\overset{n}{\underset{i=1}{\tprod }}m_{i}p_{i}q_{i}}%
\right) ^{\omega } \\
&\leq &D_{\omega }(H\wr Sym_{n}) \\
&\leq &\left( n!\right) ^{\omega -1}\left\vert H\right\vert ^{n}.
\end{eqnarray*}%
Taking logarithms, this is equivalent to 
\begin{eqnarray*}
\omega \log \left( n!\sqrt[3]{\overset{n}{\underset{i=1}{\tprod }}%
m_{i}p_{i}q_{i}}\right) &\leq &\left( \omega -1\right) \log n!+n\log
\left\vert H\right\vert \Longleftrightarrow \\
\omega \log n!+\omega \log \sqrt[3]{\overset{n}{\underset{i=1}{\tprod }}%
m_{i}p_{i}q_{i}} &\leq &\omega \log n!-\log n!+n\log \left\vert H\right\vert
\Longleftrightarrow \\
\omega &\leq &\frac{n\log \left\vert H\right\vert -\log n!}{\log \sqrt[3]{%
\overset{n}{\underset{i=1}{\tprod }}m_{i}p_{i}q_{i}}}.
\end{eqnarray*}
\end{proof}

\bigskip

The bound for $\omega $ in \textit{Corollary 5.9} suggests that $\omega $ is
close to $2$ if we could find an Abelian group $H$ simultaneously realizing $%
n$ tensors $\left\langle m_{i},p_{i},q_{i}\right\rangle $ such that $\frac{%
n\log \left\vert H\right\vert -\log n!}{\log \sqrt[3]{\overset{n}{\underset{%
i=1}{\tprod }}m_{i}p_{i}q_{i}}}$ is close to $2$, and the following
proposition is an obvious extension.

\bigskip

\begin{proposition}
For any $n$, given $n$ triples $S_{i},T_{i,}U_{i}\subseteq H$ of sizes $%
\left\vert S_{i}\right\vert =m_{i},\left\vert T_{i}\right\vert
=p_{i},\left\vert U_{i}\right\vert =q_{i},$ $1\leq i\leq n$ satisfying the
STPP in an Abelian group $H$, and the corresponding product triple $\overset{%
n}{\underset{i=1}{\tprod }}S_{i}\wr Sym_{n},\overset{n}{\underset{i=1}{%
\tprod }}T_{i}\wr Sym_{n},\overset{n}{\underset{i=1}{\tprod }}U_{i}\wr
Sym_{n}$ satisfying the TPP in the wreath product group $H\wr Sym_{n}$,
there is a maximum number $1\leq k_{n}\leq \left( n!\right) ^{3}$ triples of
permutations, $\sigma _{j},\tau _{j},\upsilon _{j}$ $\epsilon $ $Sym_{n},$ $%
1\leq j\leq k_{n}$, such that the $k_{n}$ permuted product triples $\overset{%
n}{\underset{i=1}{\tprod }}S_{\sigma _{j}\left( i\right) }\wr Sym_{n},%
\overset{n}{\underset{i=1}{\tprod }}T_{\tau _{j}\left( i\right) }\wr Sym_{n},%
\overset{n}{\underset{i=1}{\tprod }}U_{\upsilon _{j}\left( i\right) }\wr
Sym_{n}$, $1\leq j\leq k_{n}$, satisfy the STPP\ in $H\wr Sym_{n}$, and $%
H\wr Sym_{n}$ realizes the square product tensor $\left\langle n!\overset{n}{%
\underset{i=1}{\tprod }}m_{i},n!\overset{n}{\underset{i=1}{\tprod }}p_{i},n!%
\overset{n}{\underset{i=1}{\tprod }}q_{i}\right\rangle $ $k_{n}$ times
simultaneously, such that%
\begin{equation*}
\omega \leq \frac{n\log \left\vert H\right\vert -\log n!-\log k_{n}}{\log 
\sqrt[3]{\overset{n}{\underset{i=1}{\tprod }}m_{i}p_{i}q_{i}}}.
\end{equation*}%
\textit{\ }
\end{proposition}

The proof of this result once again uses the result $D_{\omega }(H\wr
Sym_{n})\leq \left( n!\right) ^{\omega -1}\left\vert H\right\vert ^{n}$ for
an Abelian group $H$ (\textit{Lemma 5.7}), in combination with \textit{%
Corollary 5.2}, as described in the proof of \textit{Corollary 5.9}. \ By 
\textit{Theorem 5.8} we know that $k_{n}\geq 1$ for any given $n$, given $n$
STPP$\ $triples $S_{i},T_{i,}U_{i}\subseteq H$ of sizes $\left\vert
S_{i}\right\vert =m_{i},\left\vert T_{i}\right\vert =p_{i},\left\vert
U_{i}\right\vert =q_{i},$ $1\leq i\leq n$. \ If the $\sigma _{j},\tau
_{j},\upsilon _{j}$ $\epsilon $ $Sym_{n}$, $1\leq j\leq k_{n}$, are taken
independently of each other in $Sym_{n}$, there are a maximum number $\left(
n!\right) ^{3}$ of permuted triples $\overset{n}{\underset{i=1}{\tprod }}%
S_{\sigma \left( i\right) }\wr Sym_{n},\overset{n}{\underset{i=1}{\tprod }}%
T_{\tau \left( i\right) }\wr Sym_{n},\overset{n}{\underset{i=1}{\tprod }}%
U_{\upsilon \left( i\right) }\wr Sym_{n}$, $\sigma ,\tau ,\upsilon $ $%
\epsilon $ $Sym_{n}$, and if these satisfy the STPP in $H\wr Sym_{n}$, it
leads to the conditional estimate $\omega <2.012$ using the wreath product
group $\left( Cyc_{6}^{\times 3}\right) ^{\times 6}\wr Sym_{2^{6}}$, as
described in section \textit{6.2.3}. \ For a given $n$, there is no known
general method of determining $k_{n}$ for the group $H\wr Sym_{n}$, with $H$
being Abelian and having a family of $n$ STPP triples $S_{i},T_{i,}U_{i}%
\subseteq H$ of sizes $\left\vert S_{i}\right\vert =m_{i},\left\vert
T_{i}\right\vert =p_{i},\left\vert U_{i}\right\vert =q_{i},$ $1\leq i\leq n$%
. \ The objective is to find the number $1\leq k_{n}\leq \left( n!\right)
^{3}$ of these triples of permutations in $Sym_{n}$ such that the bound $%
\omega \leq \frac{n\log \left\vert H\right\vert -\log n!-\log k_{n}}{\log 
\sqrt[3]{\overset{n}{\underset{i=1}{\tprod }}m_{i}p_{i}q_{i}}}$ is as tight
as possible. \ It so happens that \textit{Theorem 7.1}, [CUKS2005], is a
special case of \textit{Proposition 5.10 }for $k_{n}=1$, except there the
group is not required to be Abelian.

\pagebreak \bigskip

\chapter{\protect\huge Applications}

\bigskip

In this chapter, we apply the methods and general results in \textit{%
Chapters 4-5} to describe the general conditions needed to prove results for 
$\omega $ using the parameters $\alpha $ and $\gamma $ of concrete families
of non-Abelian groups or of single non-Abelian groups. \ We conclude with a
number of concrete upper estimates of $\omega $ in the region $2.82-2.93$. \
However, our most important result is a general estimate that $\omega \leq 
\frac{2^{n}\log n^{3n}-\log 2^{n}!-\log k_{2^{n}}}{2^{n}n\log \left(
n-1\right) }$, for some undetermined $1\leq k_{2^{n}}\leq \left(
2^{n}!\right) ^{3}$, where this $k_{2^{n}}$ is the number of times the
wreath product group $\left( Cyc_{n}^{\times 3}\right) ^{\times n}\wr
Sym_{2^{n}}$ (i.e. $\left( \left( Cyc_{n}^{\times 3}\right) ^{\times
n}\right) ^{\times 2^{n}}\rtimes Sym_{2^{n}}$) realizes the product tensor $%
\left\langle 2^{n}!\left( n-1\right) ^{n2^{n}},2^{n}!\left( n-1\right)
^{n2^{n}},2^{n}!\left( n-1\right) ^{n2^{n}}\right\rangle $ simultaneously,
and the closer $k_{2^{n}}$ is to $\left( 2^{n}!\right) ^{3}$ the closer $%
\omega $ is to $2.02$ (from the upper side).

\bigskip

\section{{\protect\Large Analysis of }$\protect\alpha ${\protect\Large \ and 
}$\protect\gamma ${\protect\Large \ for the Symmetric Groups}}

\bigskip

Here we derive upper estimates of $\alpha $ and $\gamma $ for the symmetric
groups $Sym\left( \Delta _{n}\right) \equiv Sym_{n(n+1)/2}$, and generally
for $Sym_{m}$. \ We start with $\gamma \left( Sym\left( \Delta _{n}\right)
\right) $.

\bigskip

\subsection{\textit{Estimates for }$\protect\gamma \left( Sym\left( \Delta
_{n}\right) \right) $}

\bigskip

Our first estimate for $\left( Sym\left( \Delta _{n}\right) \right) $
follows from McKay's estimate of $d^{^{\prime }}\left( Sym_{n}\right) $
[MCK1976, p. 631].

\bigskip

\begin{description}
\item[\textit{(6.1)}] $2\sqrt{6}\left( \frac{n(n+1)}{2}\right) ^{\frac{n(n+1)%
}{4}+1}e^{\sqrt{\frac{n(n+1)}{4}}\left( 1-\frac{\pi }{3}\sqrt{12}\right) -%
\frac{n(n+1)+1}{4}}\leq d^{^{\prime }}\left( Sym\left( \Delta _{n}\right)
\right) \leq \left( \sqrt{2\pi }e^{-\frac{n(n+1)}{2}}\right) ^{\frac{1}{2}%
}\left( \frac{n(n+1)}{2}\right) ^{\frac{n(n+1)+1}{4}}$.
\end{description}

\bigskip

This yields the following result for $\gamma (Sym\left( \Delta _{n}\right) )$%
.

\bigskip

\begin{corollary}
$\gamma (Sym\left( \Delta _{n}\right) )=2+O(\frac{1}{n})$.
\end{corollary}

\begin{proof}
Applying \textit{(4.26) }and \textit{(5.6) }to \textit{(6.1)}, and using
Stirling's formula,\textit{\ }we have the initial estimate 
\begin{eqnarray*}
{\small \gamma (Sym}\left( \Delta _{n}\right) {\small )} &{\small \geq }&%
\frac{n^{2}\log n-\frac{1}{2}n^{2}\left( 1+\log 2\right) +O\left( n\log
n\right) }{\log \left[ \left( \sqrt{2\pi }e^{-\frac{n(n+1)}{2}}\right) ^{%
\frac{1}{2}}\left( \frac{n(n+1)}{2}\right) ^{\frac{n(n+1)+1}{4}}\right] } \\
{\small \gamma (Sym}\left( \Delta _{n}\right) &\leq &\frac{n^{2}\log n-\frac{%
1}{2}n^{2}\left( 1+\log 2\right) +O\left( n\log n\right) }{\log \left[ 2%
\sqrt{6}\left( \frac{n(n+1)}{2}\right) ^{\frac{n(n+1)}{4}+1}e^{\sqrt{\frac{%
n(n+1)}{4}}\left( 1-\frac{\pi }{3}\sqrt{6}\right) -\frac{n(n+1)+1}{4}}\right]
}
\end{eqnarray*}%
Dividing numerator and denominator on both sides above by $n^{2}\log n-\frac{%
1}{2}n^{2}\left( 1+\log 2\right) $, and then dividing $O(n\log n)$ by $%
n^{2}\log n$, we arrive at the result.
\end{proof}

\bigskip

A second estimate of $\gamma \left( Sym\left( \Delta _{n}\right) \right) $
follows from Vershik and Kirov's estimate of $d^{^{\prime }}\left( Sym\left(
\Delta _{n}\right) \right) $ [VK1985, p. 21].

\bigskip

\begin{description}
\item[\textit{(6.2)}] $e^{-\frac{C_{1}}{2}\sqrt{\frac{n(n+1)}{2}}}\sqrt{%
\left( \frac{n(n+1)}{2}\right) !}\leq d^{^{\prime }}\left( Sym\left( \Delta
_{n}\right) \right) \leq e^{-\frac{C_{2}}{2}\sqrt{\frac{n(n+1)}{2}}}\sqrt{%
\left( \frac{n(n+1)}{2}\right) !},$
\end{description}

\bigskip

where $C_{1},$ $C_{2}>0$ constants independent of $n$.

\bigskip

\begin{corollary}
$\gamma (Sym\left( \Delta _{n}\right) )=2+\Theta \left( \frac{1}{n\log n}%
\right) $.
\end{corollary}

\begin{proof}
Consequence of \textit{(6.2). }%
\begin{eqnarray*}
\gamma (Sym\left( \Delta _{n}\right) ) &\geq &\frac{n^{2}\log n-\frac{1}{2}%
n^{2}\left( 1+\log 2\right) +O\left( n\log n\right) }{-\frac{C_{2}}{2}\sqrt{%
\frac{n(n+1)}{2}}+\frac{1}{2}n^{2}\log n-\frac{1}{4}n^{2}\left( 1+\log
2\right) +\frac{1}{2}O\left( n\log n\right) } \\
\gamma (Sym\left( \Delta _{n}\right) ) &\leq &\frac{n^{2}\log n-\frac{1}{2}%
n^{2}\left( 1+\log 2\right) +O\left( n\log n\right) }{-\frac{C_{1}}{2}\sqrt{%
\frac{n(n+1)}{2}}+\frac{1}{2}n^{2}\log n-\frac{1}{4}n^{2}\left( 1+\log
2\right) +\frac{1}{2}O\left( n\log n\right) }
\end{eqnarray*}%
for positive constants $C_{1}$ and $C_{2}$. \ Dividing numerator and
denominator on both sides by $n^{2}\log n$, we obtain the result.
\end{proof}

\bigskip

\subsection{\textit{An Upper Estimate for }$\protect\alpha (Sym\left( \Delta
_{n}\right) $}

\bigskip

The main result here is that $\alpha (Sym\left( \Delta _{n}\right) )\leq 2+O(%
\frac{1}{\log n})$.

\bigskip

\begin{lemma}
$\alpha (Sym\left( \Delta _{n}\right) )\leq 2+O(\frac{1}{\log n}).$
\end{lemma}

\begin{proof}
$\left\vert Sym(\Delta _{n})\right\vert =\left( \frac{1}{2}n(n+1)\right) !$,
and $\left\langle \overset{n}{\underset{k=1}{\tprod }}k!,\overset{n}{%
\underset{k=1}{\tprod }}k!,\overset{n}{\underset{k=1}{\tprod }}%
k!\right\rangle $ $\epsilon $ $\mathfrak{S}\left( Sym(\Delta _{n})\right) $ (%
\textit{Lemma 5.6), }and implies that $\alpha (Sym\left( \Delta _{n}\right)
)\leq \frac{\log \left( \frac{1}{2}n(n+1)\right) !}{\log \left(
n!(n-1)!\cdot \cdot \cdot \cdot 2!1!\right) }$ by \textit{(4.22)}. \ By
Stirling's asymptotic formula $\log n!\sim n\log n-n+O\left( \log n\right) $%
, as $n\longrightarrow \infty $, we have 
\begin{eqnarray*}
\log \left( \frac{1}{2}n(n+1)\right) ! &=&\frac{1}{2}n(n+1)\log \left( \frac{%
1}{2}n(n+1)\right) -\frac{1}{2}n(n+1)+O\left( \log \left( \frac{1}{2}%
n(n+1)\right) \right) \\
&=&n^{2}\log n-\frac{1}{2}n^{2}\left( 1+\log 2\right) +O\left( n\log
n\right) .
\end{eqnarray*}%
For $\log \left( n!(n-1)!\cdot \cdot \cdot \cdot 2!1!\right) $ we have the
estimate:%
\begin{eqnarray*}
&&\log \left( n!(n-1)!\cdot \cdot \cdot \cdot 2!1!\right) =\log \left(
2^{n-1}3^{n-2}....n\right) \\
&=&\left( n-1\right) \log 2+\left( n-2\right) \log 3+....+2\log \left(
n-1\right) +\log n \\
&=&n\left( \log 2+....+\log n\right) -\left( \log 2+2\log 3+....+\left(
n-1\right) \log n\right) \\
&=&n\log \left( n!\right) -\left( \log 2+2\log 3+....+\left( n-1\right) \log
n\right) .
\end{eqnarray*}%
Now, $n\log n!=n^{2}\log n-n^{2}+O\left( n\log n\right) $, and $\left( \log
2+2\log 3+....+\left( n-1\right) \log n\right) \ $\ is the result of the
evaluation of the definite integral $\int_{1}^{n}\left( x-1\right) \log x$ $%
dx+O\left( n\log n\right) $, which becomes $\frac{1}{2}\left[ \left(
x-1\right) ^{2}\log x\right] _{1}^{n}-\frac{1}{2}\int_{1}^{n}\frac{\left(
x-1\right) ^{2}}{x}$ $dx+O\left( n\log n\right) $ when we integrate by
parts. \ We find that $\frac{1}{2}\left[ \left( x-1\right) ^{2}\log x\right]
_{1}^{n}-\frac{1}{2}\int_{1}^{n}\frac{\left( x-1\right) ^{2}}{x}$ $%
dx+O\left( n\log n\right) =\frac{1}{2}n^{2}\log n-\frac{1}{4}n^{2}+O\left(
n\log n\right) $. \ Thus, $\left( \log 2+2\log 3+....+\left( n-1\right) \log
n\right) =\frac{1}{2}n^{2}\log n-\frac{1}{4}n^{2}+O\left( n\log n\right) $
which implies the estimate for $\log \left( n!(n-1)!\cdot \cdot \cdot \cdot
2!1!\right) $ of%
\begin{eqnarray*}
&&\log \left( n!(n-1)!\cdot \cdot \cdot \cdot 2!1!\right) \\
&=&n\log \left( n!\right) -\left( \log 2+2\log 3+....+\left( n-1\right) \log
n\right) \\
&=&n^{2}\log n-n^{2}+O\left( n\log n\right) -\frac{1}{2}n^{2}\log n+\frac{1}{%
4}n^{2}-O\left( n\log n\right) \\
&=&\frac{1}{2}n^{2}\log n-\frac{3}{4}n^{2}+O\left( n\log n\right) .
\end{eqnarray*}%
Then for $\frac{\log \left( \frac{1}{2}n(n+1)\right) !}{\log \left(
n!(n-1)!\cdot \cdot \cdot \cdot 2!1!\right) }$ we have the estimate:%
\begin{eqnarray*}
&&\frac{\log \left( \frac{1}{2}n(n+1)\right) !}{\log \left( n!(n-1)!\cdot
\cdot \cdot \cdot 2!1!\right) } \\
&=&\frac{n^{2}\log n-\frac{1}{2}n^{2}\left( 1+\log 2\right) +O\left( n\log
n\right) }{\frac{1}{2}n^{2}\log n-\frac{3}{4}n^{2}+O\left( n\log n\right) }
\\
&=&\frac{2-\frac{1+\log 2}{\log n}+O\left( \frac{1}{n}\right) }{1-\frac{3}{2}%
\frac{1}{\log n}+O\left( \frac{1}{n}\right) } \\
&=&\left( 2-\left( 1+\log 2\right) \frac{1}{\log n}\right) \left( 1+\frac{3}{%
2}\frac{1}{\log n}\right) +O\left( \frac{1}{\left( \log n\right) ^{2}}\right)
\\
&=&2+\frac{2-\log 2}{\log n}+O\left( \frac{1}{\left( \log n\right) ^{2}}%
\right) \\
&=&2+O\left( \frac{1}{\log n}\right) .
\end{eqnarray*}
\end{proof}

\bigskip

If we denote by $z^{^{\prime }}\left( Sym\left( \Delta _{n}\right) \right) $
the size of the maximal tensor realized by $Sym\left( \Delta _{n}\right) $,
then we have following corollary.

\bigskip

\begin{corollary}
$n^{\frac{1}{3}n^{2}+O(n)}\left( 2e\right) ^{-\frac{1}{6}n^{2}}\leq $ $%
z^{^{\prime }}\left( Sym\left( \Delta _{n}\right) \right) ^{\frac{1}{3}}<n^{%
\frac{1}{2}n^{2}+O(n)}\left( 2e\right) ^{-\frac{1}{4}n^{2}}.$
\end{corollary}

\begin{proof}
$\left\vert Sym\left( \Delta _{n}\right) \right\vert =\left( \frac{1}{2}%
n(n+1)\right) !$, and by \textit{(4.23) }$\left( \frac{1}{2}n(n+1)\right)
!\leq z^{^{\prime }}\left( Sym\left( \Delta _{n}\right) \right) <\left(
\left( \frac{1}{2}n(n+1)\right) !\right) ^{\frac{3}{2}}$. \ Taking
logarithms on all sides, and substituting the estimate $\log \left( \frac{1}{%
2}n(n+1)\right) !=n^{2}\log n-\frac{1}{2}n^{2}\left( 1+\log 2\right)
+O\left( n\log n\right) $ from \textit{Lemma 6.3}, and taking
antilogarithms, we obtain the result\textit{. \ }$z^{^{\prime }}\left(
Sym\left( \Delta _{n}\right) \right) ^{\frac{1}{3}}$ is the maximal mean
size of matrix multiplication supported by $Sym\left( \Delta _{n}\right) $,
and the result describes bounds for these in terms of the order of $%
Sym\left( \Delta _{n}\right) $, which grows exponentially with $n$.
\end{proof}

\bigskip

\textit{Lemma 6.3 }shows that it suffices to work with the symmetric groups $%
Sym\left( \Delta _{n}\right) $ since for every integer $m\geq 2$, there
exists an integer $n\geq 1$ such that $\left( \frac{1}{2}n(n+1)\right) !$ $|$
$m!$.

\bigskip

\subsection{\textit{An Upper Estimate for }$\protect\alpha (Sym_{n})$}

\bigskip

Following \textit{Lemma 6.3 }we show that a similar estimate applies to the $%
\alpha (Sym_{n})$ of arbitrary symmetric groups $Sym_{n}$.

\bigskip

\begin{corollary}
$\alpha (Sym_{n})\leq 2+O\left( \frac{1}{\log m}\right) +O\left( \frac{1}{%
\left( \log m\right) ^{2}}\right) +O\left( \frac{1}{m\left( m+1\right) }%
\right) $, for some $m<n$.
\end{corollary}

\begin{proof}
It suffices to prove that $\alpha (Sym_{n})\leq \alpha \left( Sym\left(
\Delta _{m}\right) \right) +O\left( \frac{1}{m\left( m+1\right) }\right) $,
for some integer $m<n$. \ By Lagrange's theorem the order $\left\vert
H\right\vert $ of a proper subgroup $H<G$ is a proper divisor of the order $%
\left\vert G\right\vert $ of $G$. \ For every integer $n\geq 2$, there
exists an integer $m\geq 1$ such that $\left( \frac{1}{2}m(m+1)\right) !$ $|$
$n!$. \ This is most obviously the case when $n=$ $\underset{i=1}{\overset{m}%
{\sum }}k=Z_{m}=\left( \frac{1}{2}m(m+1)\right) $, i.e. when $n$ is the sum
of the first $m$ positive integers for some $m$, in which case $%
Sym_{n}=Sym\left( \Delta _{m}\right) $, and $\alpha (Sym_{n})=\alpha \left(
Sym\left( \Delta _{m}\right) \right) $. \ If $Z_{m}<n<Z_{m+1}$ for some $m$,
then $\frac{1}{2}m(m+1)\leq n-1\leq \frac{1}{2}\left( m+1\right) \left(
m+2\right) $, and this $m$ is the least integer such that $\left( \frac{1}{2}%
m(m+1)\right) !$ is the largest proper divisor of $n!$. \ Consequently, for
this $m$, $Sym\left( \Delta _{m}\right) $ occurs as the maximal subgroup of $%
Sym_{n}$ of order $\left\vert Sym\left( \Delta _{m}\right) \right\vert
=\left( \frac{1}{2}m(m+1)\right) !\leq \left( n-1\right) !$. \ We can deduce
from these facts, by using \textit{Lemma 4.19} that:%
\begin{eqnarray*}
\alpha (Sym_{n}) &\leq &\alpha (Sym\left( \Delta _{m}\right) )+\log
_{z^{^{\prime }}(Sym\left( \Delta _{m}\right) )^{1/3}}\left[
Sym_{n}:Sym\left( \Delta _{m}\right) \right] \\
&\leq &\alpha (Sym\left( \Delta _{m}\right) )+\frac{\log \frac{1}{2}m(m+1)+1%
}{\log \left( z^{^{\prime }}(Sym\left( \Delta _{m}\right) )\right) ^{\frac{1%
}{3}}} \\
&\leq &\alpha (Sym\left( \Delta _{m}\right) )+\frac{\log \frac{1}{2}m(m+1)+1%
}{3^{-1}\log \left( \frac{1}{2}m(m+1\right) !}\hspace{0.65in}\text{(using 
\textit{(4.23)})} \\
&\sim &\alpha (Sym\left( \Delta _{m}\right) )+\frac{3}{2^{-1}m(m+1)}\frac{%
\log \frac{1}{2}m(m+1)+1}{\log \frac{1}{2}m(m+1)-1} \\
&\sim &\alpha (Sym\left( \Delta _{m}\right) )+\frac{3}{2^{-1}m(m+1)} \\
&=&\alpha (Sym\left( \Delta _{m}\right) )+O\left( \frac{1}{m\left(
m+1\right) }\right) \\
&\leq &2+\frac{2-\log 2}{\log m}+O\left( \frac{1}{\left( \log m\right) ^{2}}%
\right) +O\left( \frac{1}{m\left( m+1\right) }\right) .
\end{eqnarray*}
\end{proof}

\bigskip

\textit{Lemma 6.5 }shows that it suffices to work with the symmetric groups $%
Sym\left( \Delta _{n}\right) $.

\bigskip

\subsection{\textit{Applications to }$\protect\omega $}

\bigskip

Here we examine the implications of the above analysis for $\omega $ using
the groups $Sym\left( \Delta _{n}\right) $.

\bigskip

By \textit{Lemma 6.3} $\alpha (Sym\left( \Delta _{n}\right) )=2+\frac{2-\log
2}{\log n}+O\left( \frac{1}{\left( \log n\right) ^{2}}\right) =2+O(\frac{1}{%
\log n})$. \ For as small as $n\geq 4$, the leading term $2+\frac{2-\log 2}{%
\log n}<3$. \ The following is a table of values of $2+\frac{2-\log 2}{\log n%
}$ for $2\leq n\leq 10$.

\bigskip

\begin{tabular}{|l|r|l|}
\hline
$n$ & $\left\vert Sym\left( \Delta _{n}\right) \right\vert $ & $\frac{\log
\left( \frac{1}{2}n(n+1)\right) !}{\log \left( n!(n-1)!\cdot \cdot \cdot
\cdot 2!1!\right) }$ \\ \hline
$2$ & $6$ & \multicolumn{1}{|c|}{$3.88539$} \\ \hline
$3$ & $720$ & \multicolumn{1}{|c|}{$3.18955$} \\ \hline
$4$ & $3,628,800$ & \multicolumn{1}{|c|}{$2.94270$} \\ \hline
$5$ & $1.30767\times 10^{12}$ & \multicolumn{1}{|c|}{$2.81199$} \\ \hline
$6$ & $5.10909\times 10^{19}$ & \multicolumn{1}{|c|}{$2.72937$} \\ \hline
$7$ & $3.04888\times 10^{29}$ & \multicolumn{1}{|c|}{$2.67159$} \\ \hline
$8$ & $3.71993\times 10^{41}$ & \multicolumn{1}{|c|}{$2.62846$} \\ \hline
$9$ & $1.19622\times 10^{56}$ & \multicolumn{1}{|c|}{$2.59477$} \\ \hline
$10$ & $1.26964\times 10^{73}$ & \multicolumn{1}{|c|}{$2.56756$} \\ \hline
\end{tabular}

\bigskip

$\left\vert Sym\left( \Delta _{n}\right) \right\vert \longrightarrow \infty $
of the order of $n^{n^{2}}$, much faster than $2+\frac{2-\log 2}{\log n}%
\longrightarrow 2$. \ \textit{Lemma 6.3} (or \textit{Lemma 6.5}) shows that $%
\alpha (Sym\left( \Delta _{n}\right) )-2=O\left( \gamma \left( Sym\left(
\Delta _{n}\right) \right) \right) $, which is contrary to the limit
condition of \textit{Corollary 4.28} needed to prove $\omega =2$, i.e. we
cannot prove $\omega =2$ using the limit results \textit{Corollary 4.30 }or 
\textit{Theorem 4.31} for the family of groups $Sym\left( \Delta _{n}\right) 
$.

\bigskip

However, the application of \textit{Corollary 4.29}, and the estimate $%
\left\vert Sym\left( \Delta _{n}\right) \right\vert =\left( \frac{n(n+1)}{2}%
\right) !=n^{\frac{3}{2}n^{2}+Kn}\left( 2e\right) ^{-\frac{1}{2}n^{2}}$,
where $K$ is some constant independent of $n$, this yields the following
open question.

\bigskip

\begin{problem}
Is there a symmetric group $Sym\left( \Delta _{n}\right) $, for some $n>1$,
such that $z^{^{\prime }}(Sym\left( \Delta _{n}\right) )^{\frac{1}{3}%
}>d^{^{\prime }}(Sym\left( \Delta _{n}\right) )$ and $\frac{z^{^{\prime
}}(Sym\left( \Delta _{n}\right) )^{\frac{t}{3}}}{d^{^{\prime }}(Sym\left(
\Delta _{n}\right) )^{t-2}}\geq n^{\frac{3}{2}n^{2}+Kn}\left( 2e\right) ^{-%
\frac{1}{2}n^{2}},$ for some $2<t<3$, where $K$ is a constant independent of 
$n?$
\end{problem}

\bigskip

\section{{\protect\Large Some Estimates for the Exponent }$\protect\omega $}

\bigskip

Here we derive a number of estimates for $\omega $, using wreath products of
Abelian groups with symmetric groups. \ We start with the Abelian group $%
Cyc_{n}^{\times 3}\equiv Cyc_{n}\times Cyc_{n}\times Cyc_{n}$, for which we
start with a basic lemma.

\bigskip

\begin{lemma}
For the Abelian group $Cyc_{n}^{\times 3}$ the subset triples $%
(S_{1},T_{1},U_{1})$ and $\left( S_{2},T_{2},U_{2}\right) $ defined by%
\begin{eqnarray*}
S_{1} &:&=Cyc_{n}\backslash \left\{ 1\right\} \times \left\{ 1\right\}
\times \left\{ 1\right\} ,\text{ }T_{1}:=\left\{ 1\right\} \times
Cyc_{n}\backslash \left\{ 1\right\} \times \left\{ 1\right\} ,\text{ }%
U_{1}:=\left\{ 1\right\} \times \left\{ 1\right\} \times Cyc_{n}\backslash
\left\{ 1\right\} , \\
S_{2} &:&=\left\{ 1\right\} \times Cyc_{n}\backslash \left\{ 1\right\}
\times \left\{ 1\right\} ,\text{ }T_{2}:=\left\{ 1\right\} \times \left\{
1\right\} \times Cyc_{n}\backslash \left\{ 1\right\} ,\text{ }%
U_{2}:=Cyc_{n}\backslash \left\{ 1\right\} \times \left\{ 1\right\} \times
\left\{ 1\right\} ,
\end{eqnarray*}%
$(1)$ have the triple product property, and $(2)$ have the simultaneous
triple product property.
\end{lemma}

\begin{proof}
$(1)$ For arbitrary elements $\left( s^{^{\prime }},1,1\right) ,\left(
s,1,1\right) $ $\epsilon $ $S_{1}=Cyc_{n}\backslash \left\{ 1\right\} \times
\left\{ 1\right\} \times \left\{ 1\right\} $, $\left( 1,t^{^{\prime
}},1\right) ,\left( 1,t,1\right) $ $\epsilon $ $T_{1}=\left\{ 1\right\}
\times Cyc_{n}\backslash \left\{ 1\right\} \times \left\{ 1\right\} ,$ $%
\left( 1,1,u^{^{\prime }}\right) ,\left( 1,1,u\right) $ $\epsilon $ $%
U=\left\{ 1\right\} \times \left\{ 1\right\} \times Cyc_{n}\backslash
\left\{ 1\right\} $, the condition $\left( s^{^{\prime }},1,1\right) \left(
s,1,1\right) ^{-1}\left( 1,t^{^{\prime }},1\right) \left( 1,t,1\right)
^{-1}\left( 1,1,u^{^{\prime }}\right) \left( 1,1,u\right) ^{-1}=\left(
s^{^{\prime }}s^{-1},t^{^{\prime }}t^{-1},u^{^{\prime }}u^{-1}\right)
=\left( 1,1,1\right) $ can only occur if $s^{^{\prime }}=s,$ $t^{^{\prime
}}=t,$ $u^{^{\prime }}=u$, which implies that $\left( s^{^{\prime
}},1,1\right) =\left( s,1,1\right) ,$ $\left( 1,t^{^{\prime }},1\right)
=\left( 1,t,1\right) ,$ $\left( 1,1,u^{^{\prime }}\right) =\left(
1,1,u\right) $. \ Thus, $\left( S_{1},T_{1},U_{1}\right) $ has the triple
product property, and we can prove the same for $\left(
S_{2},T_{2},U_{2}\right) $.\newline
$(2)$ Refer to \textit{Proposition 5.2}, [CUKS2005].
\end{proof}

\bigskip

The following is an elementary corollary.

\bigskip

\begin{corollary}
The Abelian group $Cyc_{n}^{\times 3}$ realizes the tensor $\left\langle
n-1,n-1,n-1\right\rangle $ $2$ times simultaneously.
\end{corollary}

\begin{proof}
Consequence of definition \textit{(5.2) }with respect to the triples in 
\textit{Lemma 6.11}.
\end{proof}

\bigskip

\subsection{$\protect\omega <2.82$\textit{\ via }$Cyc_{16}^{\times 3}$}

\bigskip

By \textit{Corollary 6.12 }$Cyc_{n}^{\times 3}$ realizes the identical
tensors $\left\langle n-1,n-1,n-1\right\rangle $ and $\left\langle
n-1,n-1,n-1\right\rangle $ simultaneously, and, thereby, supports two
independent, simultaneous multiplications of square matrices of order $n-1$,
and by \textit{Corollary 5.3}, we have the inequality

\begin{equation*}
\omega \leq \frac{\log n^{3}-\log 2}{\log \left( n-1\right) }\text{.}
\end{equation*}

The expression $\frac{\log n^{3}-\log 2}{\log \left( n-1\right) }$ achieves
a minimum $2.81553...$ for $n=16$, i.e. $\omega <2.81554$.

\bigskip

\subsection{$\protect\omega <2.93$\textit{\ via }$Cyc_{41}^{\times 3}\wr
Sym_{2}$}

\bigskip

From above, $Cyc_{n}^{\times 3}$ realizes the tensor $\left\langle
n-1,n-1,n-1\right\rangle $ $2$ times simultaneously. \ By part $(1)$ of 
\textit{Corollary 5.9 }$Cyc_{n}^{\times 3}\wr Sym_{2}$ realizes the tensor $%
\left\langle 2\left( n-1\right) ^{2},2\left( n-1\right) ^{2},2\left(
n-1\right) ^{2}\right\rangle $ once, and by part $(2)$

\begin{equation*}
\omega \leq \frac{6\log n-\log 2}{2\log \left( n-1\right) }.
\end{equation*}

The minimum value $\omega \leq 2.92613048...$ is attained for $n=41$.

\bigskip

More generally, by \textit{Proposition 5.10 }$Cyc_{n}^{\times 3}\wr Sym_{2}$
realizes the product tensor $\left\langle 2\left( n-1\right) ^{2},2\left(
n-1\right) ^{2},2\left( n-1\right) ^{2}\right\rangle $ some $1\leq
k_{2}<\left( 2!\right) ^{3}$ times simultaneously such that

\begin{eqnarray*}
\omega &\leq &\frac{6\log n-\log 2-\log k_{2}}{2\log \left( n-1\right) } \\
&\leq &\frac{6\log n-\log 2}{2\log \left( n-1\right) } \\
&<&2.93.
\end{eqnarray*}

If, for example, $k_{2}=\left( 2!\right) ^{3}$ we would have $\frac{6\log
n-\log 2-\log k_{2}}{2\log \left( n-1\right) }=\frac{6\log n-\log 16}{2\log
\left( n-1\right) }$, which achieves a minimum of $2.478495...$ at $n=6$. \
The following table gives minima for the expression $\frac{6\log n-\log
2-\log k_{2}}{2\log \left( n-1\right) }$ for the values of $1\leq k_{2}\leq
\left( 2!\right) ^{3}$:

\bigskip

\begin{tabular}{|l|l|}
\hline
$k_{2}$ & $\omega \leq \min \left( \frac{6\log n-\log 2-\log k_{2}}{2\log
\left( n-1\right) }\right) $ \\ \hline
$1$ & $2.9261$ \\ \hline
$2$ & $2.8163$ \\ \hline
$3$ & $2.7351$ \\ \hline
$4$ & $2.6700$ \\ \hline
$5$ & $2.6142$ \\ \hline
$6$ & $2.5647$ \\ \hline
$7$ & $2.5200$ \\ \hline
$8$ & $2.4785$ \\ \hline
\end{tabular}

\bigskip

Here we know only the case $k_{2}=1$ based ultimately on \textit{Theorem 5.8}%
, and we conclude that $\omega \leq 2.9261$, and for the values $2\leq
k_{2}\leq 8$ the table's bounds for $\omega $ are \textit{conditional} on
those values of $k_{2}$.

\bigskip

\subsection{$\protect\omega <2.82$\textit{\ via }$\left( Cyc_{25}^{\times
3}\right) ^{\times 25}\wr Sym_{2^{25}}$}

\bigskip

As before, we start we start with $Cyc_{n}^{\times 3}$, which realizes the
tensor $\left\langle n-1,n-1,n-1\right\rangle $ $2$ times simultaneously. \
By \textit{Corollary 5.5 }the $m$-fold direct product $\left(
Cyc_{n}^{\times 3}\right) ^{\times m}$ realizes the pointwise product tensor 
$\left\langle \left( n-1\right) ^{m},\left( n-1\right) ^{m},\left(
n-1\right) ^{m}\right\rangle $ $2^{m}$ times simultaneously. \ By part $(1)$
of \textit{Corollary 5.9} the group $\left( Cyc_{n}^{\times 3}\right)
^{\times m}\wr Sym_{2^{m}}$ (i.e. $\left( \left( Cyc_{n}^{\times 3}\right)
^{\times m}\right) ^{\times 2^{m}}\rtimes Sym_{2^{m}}$) realizes the product
tensor $\left\langle 2^{m}!\left( n-1\right) ^{2^{m}m},2^{m}!\left(
n-1\right) ^{2^{m}m},2^{m}!\left( n-1\right) ^{2^{m}m}\right\rangle $ once,
and by part $(2)$%
\begin{eqnarray*}
\omega &\leq &\frac{2^{m}\log n^{3m}-\log 2^{m}!}{2^{m}m\log \left(
n-1\right) } \\
&\sim &\frac{3m\log n-m\log 2+1}{m\log \left( n-1\right) }.
\end{eqnarray*}

If we let $m\longrightarrow \infty $ the right hand side tends to $2.815...$
we derive that $\omega <2.82$, as good a result as in section \textit{6.2.1}.

\bigskip

If we consider the $n$-fold direct product $\left( Cyc_{n}^{\times 3}\right)
^{\times n}$, then more generally, by \textit{Proposition 5.10 }$\left(
Cyc_{n}^{\times 3}\right) ^{\times n}\wr Sym_{2^{n}}$ (i.e. $\left( \left(
Cyc_{n}^{\times 3}\right) ^{\times n}\right) ^{\times 2^{n}}\rtimes
Sym_{2^{n}}$) realizes the product tensor $\left\langle 2^{n}!\left(
n-1\right) ^{n2^{n}},2^{n}!\left( n-1\right) ^{n2^{n}},2^{n}!\left(
n-1\right) ^{n2^{n}}\right\rangle $ some $1\leq k_{2^{n}}\leq \left(
2^{n}!\right) ^{3}$ times simultaneously such that%
\begin{equation*}
\omega \leq \frac{2^{n}\log n^{3n}-\log 2^{n}!-\log k_{2^{n}}}{2^{n}n\log
\left( n-1\right) }.
\end{equation*}

If here $k_{2^{n}}=\left( 2^{n}!\right) ^{3}$ then $\omega \leq \frac{%
2^{n}\log n^{3n}-4\log 2^{n}!}{2^{n}n\log \left( n-1\right) }$, the latter
achieving a minimum of $2.012$ for $n=6$. \ In general, for the groups $%
\left( Cyc_{n}^{\times 3}\right) ^{\times n}\wr Sym_{2^{n}}$ the closer $%
k_{2^{n}}$ is to $\left( 2^{n}!\right) ^{3}$, the closer we could push $%
\omega $ down towards $2.012$ (from the upper side).

\bigskip

These results point to the utility of finding triples of permutations $%
\sigma _{j},\tau _{j},\upsilon _{j}$ $\epsilon $ $Sym_{n}$ in order that a
maximum number $1\leq k_{n}\leq \left( n!\right) ^{3}$ of permuted product
triples $\overset{n}{\underset{i=1}{\tprod }}S_{\sigma _{j}\left( i\right)
}\wr Sym_{n},\overset{n}{\underset{i=1}{\tprod }}T_{\tau _{j}\left( i\right)
}\wr Sym_{n},\overset{n}{\underset{i=1}{\tprod }}U_{\upsilon _{j}\left(
i\right) }\wr Sym_{n}$, $1\leq j\leq k_{n}$, satisfy the STPP\ in $H\wr
Sym_{n}$, where $H$ is Abelian with a given STPP\ family $\left\{ \left(
S_{i},T_{i,}U_{i}\right) \right\} _{i=1}^{n}$. \ Using such groups and their
triples in this way, the sharpest upper bounds for $\omega $ will occur
where the ratio $k_{n}/\left( n!\right) ^{3}$ is highest. \ Therefore, we
pose the following problem.

\bigskip

\begin{problem}
For any given $n$, and group $H\wr Sym_{n}$, with $H$ being Abelian, and a
given family $\left\{ \left( S_{i},T_{i,}U_{i}\right) \right\} _{i=1}^{n}$
of $n$ STPP\ triples of $H$, what is the largest number $1\leq k_{n}\leq
\left( n!\right) ^{3}$ such that there are $k_{n}$ triples of permutations $%
\sigma _{j},\tau _{j},\upsilon _{j}$ $\epsilon $ $Sym_{n}$ such that the $%
k_{n}$ triples $\overset{n}{\underset{i=1}{\tprod }}S_{\sigma _{j}\left(
i\right) }\wr Sym_{n},\overset{n}{\underset{i=1}{\tprod }}T_{\tau _{j}\left(
i\right) }\wr Sym_{n},\overset{n}{\underset{i=1}{\tprod }}U_{\upsilon
_{j}\left( i\right) }\wr Sym_{n}$, $1\leq j\leq k_{n}$, satisfy the STPP\ in 
$H\wr Sym_{n}$?
\end{problem}

\bigskip

\pagebreak \bigskip

{\huge Bibliography}

\bigskip

\begin{enumerate}
\item {\small [BOA1982], Boas, P. van Emde, \textit{Berekeningscomplexiteit
van Bilineaire en Kwadratische Vormen (Computational Complexity of Bilinear
and Quadratic Forms)}, in Vitanyi, P. M. V., van Leeuwen, J., \& van Emde
Boas, P. (eds.), Colloquium Complexiteit en Algorithmen MC Syllabi 48.2,
Amsterdam, 1982, pp.3--68.}

\item {\small [BCS1997], B\"{u}rgisser, P., Clausen, M., \& Shokrollahi, A., 
\textit{Algebraic Complexity Theory}, Springer, Berlin, 1997.}

\item {\small [COH2007], Cohn, H., 'Group-theoretic Algorithms for Matrix
Multiplication', (private email, 30/03/2007).}

\item {\small [CUKS2005], Cohn, H., Umans, C., Kleinberg, R., Szegedy, B.,
'Group-theoretic Algorithms for Matrix Multiplication', \textit{Proceedings
of the 46th Annual IEEE Symposium on Foundations of Computer Science 2005},
IEEE Computer Society, 2005, pp. 438--449, arXiv:math.GR/0511460.}

\item {\small [CU2003], Cohn, H. \& Umans, C., 'A Group-theoretic Approach
to Fast Matrix Multiplication', \textit{Proceedings of the 44th Annual IEEE
Symposium on Foundations of Computer Science 2003}, IEEE Computer Society,
2003, pp. 379--388, arXiv:math.GR/0307321.}

\item {\small [CW1990], Coppersmith, D., \& Winograd, S., 'Matrix
Multiplication via Arithmetic Progressions', \textit{Journal of Symbolic
Computation}, 9, 1990, pp. 251-280.}

\item {\small [DDHK2006], Demmel, J., Dumitriu, I., Holtz, O., \& Kleinberg,
R., 'Fast Matrix Multiplication is Stable', 2006, arXiv:math.NA/0603207.}

\item {\small [HUP1998], Huppert, B., \textit{Character Theory of Finite
Groups}, Walter de Gruyter, Berlin, 1998.}

\item {\small [LAN2006], Landsberg, J. M., 'Geometry and Complexity of
Matrix Multiplication', 2006.}

\item {\small [MIL2003], Milne, J. S., 'Group Theory', 2003,
http://www.jmilne.org/math/.}

\item {\small [MR2003], Maslen, D. K., \& Rockmore, D. N., 'The Cooley-Tukey
FFT and Group Theory', \textit{Modern Signal Processing}, 46:2003, pp.
281-301.}

\item {\small [MCK1976], McKay, J., 'The Largest Degrees of Irreducible
Characters of the Symmetric Group', \textit{Mathematics of Computation},
223:6, 1976, pp. 624-631.}

\item {\small [PAN1984], Pan, V, 'How can we Speed up Matrix
Multiplication', \textit{SIAM\ Review}, 26:3, 1984, pp. 393-415.}

\item {\small [SER1977], Serre, J. P., \textit{Linear Representations of
Finite Groups}, Springer-Verlag, New York, 1977.}

\item {\small [STR1969], Strassen, V., 'Gaussian Elimination is not
Optimal', \textit{Numerische Mathematik}, 13, 1969, pp. 354--356.}

\item {\small [UMA2007], Umans, C., 'Group-theoretic Algorithms for Matrix
Multiplication', (private email, 22/03/2007).}

\item {\small [VK1985], Vershik, A. M., Kerov, S. V., 'Asymptotics of the
Largest and the Typical Dimensions of Irreducible Representations of a
Symmetric Group', \textit{Funktsional'nyi Analiz i Ego Prilozheniya} (trans. 
\textit{Functional Analysis and Applications}), 19:1, 1985, pp. 25-36.}

\item {\small [WIN1971], Winograd, S., 'On Multiplication of }$2\times 2$ 
{\small Matrices, \textit{Linear Algebra and Applications}, 4:?, 1971, pp.
381-388.}
\end{enumerate}

\end{document}